\RequirePackage{fix-cm}
\documentclass[smallextended]{svjour3}       
\smartqed  
\usepackage{graphicx}
\usepackage{braket,amsfonts}

\usepackage{booktabs}
\usepackage{array}
\usepackage{makecell}

\usepackage{amsmath,amssymb}  

\usepackage{pgfplots}

\usepackage{algorithmic}
\usepackage{algorithm}

\usepackage{graphicx,epstopdf}

\usepackage{amsopn}

\usepackage{xspace}
\usepackage{bold-extra}
\usepackage[most]{tcolorbox}
\usepackage{placeins,diagbox }

\usepackage{wrapfig} 
\usepackage{tikz,pgfplots,mathtools,float}
\usepackage{booktabs,caption}
\usepackage[flushleft]{threeparttable}
\usepackage{multirow}
\usepackage{amssymb,latexsym}  
\usepackage{extarrows}
\usepackage{dsfont} 
\usepackage{amssymb}  
\usepackage{color}
\usepackage{xcolor}
\usepackage{cite}

\colorlet{texcscolor}{blue!50!black}
\colorlet{texemcolor}{red!70!black}
\colorlet{texpreamble}{red!70!black}
\colorlet{codebackground}{black!25!white!25}

\def\argmin{\mathop{\rm argmin}}

\newcommand\bs{\symbol{'134}} 

\lstdefinestyle{siamlatex}{%
  style=tcblatex,
  texcsstyle=*\color{texcscolor},
  texcsstyle=[2]\color{texemcolor},
  keywordstyle=[2]\color{texemcolor},
  moretexcs={cref,Cref,maketitle,mathcal,text,headers,email,url},
}

\tcbset{%
  colframe=black!75!white!75,
  coltitle=white,
  colback=codebackground, 
  colbacklower=white, 
  fonttitle=\bfseries,
  arc=0pt,outer arc=0pt,
  top=1pt,bottom=1pt,left=1mm,right=1mm,middle=1mm,boxsep=1mm,
  leftrule=0.3mm,rightrule=0.3mm,toprule=0.3mm,bottomrule=0.3mm,
  listing options={style=siamlatex}
}

\DeclareTotalTCBox{\code}{ v O{} }
{ 
  fontupper=\ttfamily\color{black},
  nobeforeafter,
  tcbox raise base,
  colback=codebackground,colframe=white,
  top=0pt,bottom=0pt,left=0mm,right=0mm,
  leftrule=0pt,rightrule=0pt,toprule=0mm,bottomrule=0mm,
  boxsep=0.5mm,
  #2}{#1}

\patchcmd\newpage{\vfil}{}{}{}
\usepackage{amsmath, amssymb, bbm, xspace}

\newcommand{\ie}{i.e.\@\xspace} 

\newcommand{\abs}[1]{\ensuremath{|#1|}}
\newcommand{\Real}{\ensuremath{\mathbb{R}}}

\newcommand{\dom}{\mathrm{dom}}

\DeclareMathOperator*{\st}{subject\;to}

\def\diag{\mathop{\hbox{\rm diag}}}

\def\sign{\mathop{\hbox{\rm sign}}}
\def\sgn{\mathop{\hbox{\rm sgn}}}

\def\spose#1{\hbox to 0pt{#1\hss}}

\def\text #1{\hbox{\quad#1\quad}}

\def\nthinsp{\mskip -2   mu}

\def\I{_{\scriptscriptstyle I}}

\def\superstar{^{\raise 0.5pt\hbox{$\nthinsp *$}}}
\def\SUPERSTAR{^{\raise 0.5pt\hbox{$*$}}}

\def\lamstarT {\lambda^{\raise 0.5pt\hbox{$\nthinsp *$}T}}

\def\Iscr{{\cal I}}

\def\Lscr{{\cal L}}

\def\Pscr{{\cal P}}

\def\hbar{\skew{4.2}\bar h}

		\def\bkE{{\rm I\kern-.17em E}}
		\def\bkE{\mathbb{E}}
		\def\bk1{{\rm 1\kern-.17em l}}
		\def\bkD{{\rm I\kern-.17em D}}
		\def\bkR{{\rm I\kern-.17em R}}
		\def\bkP{{\rm I\kern-.17em P}}
		\def\bkY{{\bf \kern-.17em Y}}
		\def\bkZ{{\bf \kern-.17em Z}}

		\def\beq{\begin{eqnarray}}
		\def\bc{\begin{center}}
		\def\be{\begin{enumerate}}
		\def\bi{\begin{itemize}}
		\def\bs{\begin{small}}
		\def\bS{\begin{slide}}
		\def\ec{\end{center}}
		\def\ee{\end{enumerate}}
		\def\ei{\end{itemize}}
		\def\es{\end{small}}
		\def\eS{\end{slide}}
		\def\eeq{\end{eqnarray}}
		\def\qed{\quad \vrule height7.5pt width4.17pt depth0pt}

	\def\cp2problem#1#2#3#4{\fbox
		 {\begin{tabular*}{0.9\textwidth}
			{@{}l@{\extracolsep{\fill}}l@{\extracolsep{6pt}}l@{\extracolsep{\fill}}c@{}}
				#1 & & $#4 $ 
			\end{tabular*}}}

\newcommand{\pmat}[1]{\begin{pmatrix} #1 \end{pmatrix}}
		
		\renewcommand{\emph}[1]{\textbf{#1}}

		\def\bk1{{\rm 1\kern-.17em l}}
		\def\bkD{{\rm I\kern-.17em D}}
		\def\bkR{{\rm I\kern-.17em R}}
		\def\bkP{{\rm I\kern-.17em P}}
		
		\def\bkZ{{\bf{Z}}}

\newcommand {\beeq}[1]{\begin{equation}\label{#1}}
\newcommand {\eeeq}{\end{equation}}
\newcommand {\bea}{\begin{eqnarray}}
\newcommand {\eea}{\end{eqnarray}}

\def\texitem#1{\par\smallskip\noindent\hangindent 25pt
               \hbox to 25pt {\hss #1 ~}\ignorespaces}

\def\st{\mbox{subject to}}

\renewcommand{\emph}[1]{\textbf{#1}}
	\def\bkE{{\rm I\kern-.17em E}}
	\def\bk1{{\rm 1\kern-.17em l}}
	\def\bkD{{\rm I\kern-.17em D}}
	\def\bkR{{\rm I\kern-.17em R}}
	\def\bkP{{\rm I\kern-.17em P}}
	
	\def\bkZ{{\bf{Z}}}
	\def\b12{(\beta_1,\beta_2)}

		\newcommand{\yx}[1]{{\color{black}#1}}
	
	\def\us#1{{{\color{black}#1}}}
	
	\newcommand{\uvs}[1]{{\color{black}#1}}
	\newcommand{\vvs}[1]{{\color{black}#1}}
	
	\newcommand{\yyxx}[1]{{\color{black}#1}}
    \newcommand{\yxI}[1]{{\color{black}#1}}
    \newcommand{\yxII}[1]{{\color{black}#1}}
    \newcommand{\yxIII}[1]{{\color{black}#1}}
    
    \newcommand{\diff}[1]{{\color{black}#1}}

	\newtheorem{Example}{Example}

	\def\sN{\mathcal{N}}
	\def\sZ{\mathcal{Z}}
	\def\sU{\mathcal{U}}

	\def\bR{\mathbb{R}}
	\def\bI{\mathbf{1}}
	\def\argmin{\mathrm{argmin}}
\newcommand{\scrL}{\mathcal{L}}
\newcommand{\scrN}{\mathcal{N}}
\newcommand{\scrH}{\mathcal{H}}
\newcommand{\tscrL}{\tilde{\mathcal{L}}}
\newcommand{\lb}{\bar{L}}

\newtheorem{assumption}{Assumption}
\allowdisplaybreaks

\begin{document}

\title{Tractable ADMM Schemes for Computing KKT Points and Local Minimizers for  $\ell_0$-Minimization Problems
}


\author{Yue Xie        \and
        Uday V. Shanbhag 
}


\institute{Y. Xie \at
              Penn State University \\
              \email{yux111@psu.edu}           
           \and
           U. V. Shanbhag \at
              Penn State University \\
              \email{udaybag@psu.edu}   
}

\date{Received: date / Accepted: date}

\maketitle

\begin{abstract}
We consider an $\ell_0$-minimization problem where $f(x) + \gamma \|x\|_0$ is minimized over a polyhedral set and the $\ell_0$-norm regularizer implicitly emphasizes sparsity of the solution. Such a setting captures a range of problems in image processing and statistical learning. Given the nonconvex and discontinuous nature of this norm, convex regularizers are often employed as substitutes. Therefore, far less is known about directly solving the $\ell_0$-minimization problem. Inspired by~\cite{feng2018complementarity}, we consider resolving an equivalent formulation of the $\ell_0$-minimization problem as a mathematical program with complementarity constraints (MPCC) and make the following contributions towards the characterization and computation of its KKT points: (i) First, we show that feasible points of this formulation satisfy the relatively weak Guignard constraint qualification. 
Furthermore, under the suitable convexity assumption on $f(x)$, an equivalence is derived between first-order KKT points and local minimizers of the MPCC
formulation. (ii) Next, we apply two alternating direction method of multiplier (ADMM) algorithms to exploit special structure of the MPCC formulation: (ADMM$_{\rm cf}^{\mu, \alpha, \rho}$) and (ADMM$_{\rm cf}$). These two ADMM schemes both have tractable subproblems. Specifically, in spite of the overall nonconvexity, we show that the first of the ADMM updates can be effectively reduced to a closed-form expression by recognizing a hidden convexity property while the second necessitates solving a convex program. In (ADMM$_{\rm cf}^{\mu, \alpha, \rho}$), we prove subsequential convergence to a perturbed KKT point under mild assumptions. Our preliminary numerical experiments suggest that the tractable ADMM schemes are more scalable than their standard counterpart and ADMM$_{\rm cf}$ compares well with its competitors to solve the $\ell_0$-minimization problem.

\end{abstract}
		\section{Introduction} \label{sec:intro}
In this paper, we consider the $\ell_0$-minimization 
problem:
	\begin{align} \label{L0min} 
	 \min_x \  f(x) + \gamma \| x \|_0 \quad \st \ Ax  \geq  b,
	\end{align}
where $x \in \Real^n$, $A \in \Real^{m \times n}$, $b \in \Real^m$,
and $\gamma > 0$. {\vvs{Suppose} $f(x) {\ \triangleq \ }
f_Q(x) + g(x) $, where $f_Q: \Real^n \to \Real$ is a {quadratic} function
and $g: \Real^n \to \Real$ is a smooth convex function.} The $\ell_0$-norm of a vector captures the  number of nonzero entries
while {an $\ell_0$-norm regularizer} implicitly emphasizes the sparsity of the
resulting minimizer. {{$\ell_0$-minimization problems of the form}
\eqref{L0min} \us{assume relevance in}  applications in image
processing and \us{statistical learning}
(cf.~\cite{donoho2006compressed,candes2008introduction,tibshirani1996regression}).}
The nonconvexity and discontinuity of the $\ell_0$-norm has prompted the
usage of convex $\ell_1$ or $\ell_2$-norm regularizers or other tractable
variants~\cite{tibshirani1996regression,bach2012optimization}. While relatively less is known about directly solving problem \eqref{L0min}, a solution of \eqref{L0min} may have better statistical property. In fact, global
solutions of \eqref{L0min} achieve model selection consistency and are known to be sparse under weaker conditions than when utilizing the $\ell_1$-norm
(cf.~\cite{10.2307/41714785}). Therefore, despite the computational challenges in addressing the  $\ell_0$-norm penalty, resolution of the
$\ell_0$-minimization problem is still desirable. In this work, we focus on
direct resolution of \eqref{L0min}. 

{\bf Related work.} To solve \eqref{L0min}, Feng, Mitchell, Pang, Shen, and
W\"{a}chter~\cite{feng2018complementarity} introduced two complementarity-based
formulations equivalent with \eqref{L0min} and processed them by standard
nonlinear programming solvers. Blumensath and Davis proposed an iterative
hard-thresholding (IHT) algorithm, applicable when $f(x)$ is a least-squares
metric and the constraint $Ax \geq b$ is absent~\cite{Blumensath2008}.
Convergence to a local minimizer may be claimed and performance of the scheme
can be improved if warm-started from a point computed by matching pursuit.
	 
A problem class closely related to \eqref{L0min} is the
$\ell_0$-constrained problem \eqref{Msparse}. Although they are not equivalent
due to nonconvexity of $\ell_0$-norm, solution method of \eqref{Msparse} may
inspire efficient algorithms to tackle \eqref{L0min}.
\begin{align}\label{Msparse}
	\min_x  \  f(x) \quad \st \  Ax  \geq b,  \| x \|_0  \leq M.
	\end{align}
This problem {finds application} in best subset
regression~\cite{bertsimas2009algorithm,bertsimas2016best}, cardinality
constrained portfolio
optimization~\cite{bertsimas2009algorithm}, and graphical model
estimation~\cite{Fang2019}.
To solve~\eqref{Msparse}, combining first-order methods and mixed-integer optimization~\cite{bertsimas2016best} was seen to be promising. By considering an equivalent complementarity
formulation of \eqref{Msparse}, Burdakov et al.~\cite{burdakov2016mathematical}
developed a regularization scheme. Moreover, a relatively weak constraint
qualification was shown to hold at every feasible point of this reformulation
and consequently KKT conditions are necessary at local minima.  

In addition to $\ell_0$-norm penalization, related work has  examined the usage of the $\ell_p$-norm ($p \in (0,1)$)~\cite{ge2011note,fung2011equivalence}, 
the smoothly clipped absolute deviation (SCAD)
penalty~\cite{fan2001variable,liu2016global},  \us{the minimax concave penalty
(MCP)~\cite{zhang10nearly}, and the capped-$\ell_1$
penalty~\cite{zhang10analysis}.} 
More recently, a generalization of the $\ell_0$-norm constraint was considered
in the form of an {\em affine sparsity constraint~\cite{dong17structural}}.

{\bf Nonconvex ADMM schemes.}  Since our focus lies in developing an ADMM framework to exploit the structure of an equivalent nonconvex formulation of \eqref{L0min}, we provide a brief review of the available convergence statements in the context of ADMM schemes for nonconvex programs.

Encouraged by the success of ADMM on convex problems, researchers
have tried to implement and analyze ADMM on nonconvex problems. Table~\ref{tab:
nonconADMM} lists some of the main theoretical findings regarding variants of ADMM schemes employed to address different types of nonconvex problems
\cite{doi:10.1137/18M1190689,DBLP:journals/siamjo/HongLR16,DBLP:journals/corr/JiangLMZ16,DBLP:journals/corr/WangYZ15}.
In the second column, we include the assumptions necessary for these authors to
prove convergence or derive complexity bounds. Note that these assumptions pertain to the problem itself but may not be sufficient to guarantee the final result. More assumptions on the parameter settings or iterates of the
algorithms may well be needed. Moreover, for some of the findings, it is shown that if the K{\L} property (\diff{See Definition~\ref{KL} in Appendix}) is assumed, convergence can be guaranteed
\cite{doi:10.1137/18M1190689,DBLP:journals/corr/WangYZ15}. Also note that all of the papers in Table~\ref{tab: nonconADMM} assume global resolution of each
subproblem of ADMM, even when the subproblem is nonconvex. Specifically, in
\cite{DBLP:journals/corr/WangYZ15}, it is explained that the proposed ADMM
scheme can address MPCC but requires globally resolving an MPCC at each step; this is in sharp contrast with the tractable structure of each update in our scheme in
this paper (in other words, we do not require resolving a nonconvex problem globally at each step).

\begin{table}
\scriptsize
  \caption{Main results on convergence of nonconvex ADMM}
  \label{tab: nonconADMM}
  \centering
  \begin{tabular}{llcl}
  \toprule
  Problem & \makecell{Necessary Assumptions} & Result & Lit. \\
  \hline 
  $
  \begin{array}{rl}
  \min & \phi(x_0, x_1, \hdots, x_p, y) \\
  \mbox{s.t.} & \sum_{i=0}^p A_i x_i + By = 0.
  \end{array} $
   & \makecell{$ \nabla_y \phi( \cdot, y) $ is Lipschitz \\
   continuous in $y$. \\
   ${\rm Im}([A_1, \hdots, A_p]) \subseteq {\rm Im}(B)$.}
  &  \makecell{Subsequential \\ convergence to \\ stationary points. } & \cite{DBLP:journals/corr/WangYZ15}  \\
  \hline
 $
  \begin{array}{rl}
\min & f(x_1, x_2, \hdots, x_N) \\
& \quad \quad \quad + \sum_{i=1}^{N-1} r_i(x_i) \\
\mbox{s.t.} & \sum_{i=1}^{N} A_i x_i = b, x_i \in \mathcal{X}_i, \\
& \forall i = 1, \hdots, N-1.
\end{array}
 $
  & \makecell{$f$ is differentiable. \\ 
For any $i$, $\mathcal{X}_i$ is convex. \\
$A_N$ has full row rank.
  } & \makecell{Iter. complexity \\ of $\mathcal{O}(1/\epsilon^2)$ to \\ obtain an  \\ $\epsilon$-stationary point.} & \cite{DBLP:journals/corr/JiangLMZ16} \\
\hline
$
\begin{array}{rl}
\min & \sum_{k=1}^K g_k(x_k) + h(x_0) \\
\mbox{s.t.} & x_k = x_0, \forall k = 1, \hdots, K,\\
& x_0 \in \mathcal{X}.
\end{array}
$
& \makecell{ For any $k = 1, \hdots, K$, \\
$ \nabla g_k(x)$ is Lipschitz continuous. \\ $h(\cdot)$ is convex. \\ 
$\mathcal{X}$ is convex and compact. } & \makecell{Subsequential \\ convergence to \\ stationary points}. & \cite{DBLP:journals/siamjo/HongLR16} \\
\hline
$
\begin{array}{rl}
\min & \sum_{k=1}^K g_k(x_k) + \ell(x_0) \\
\mbox{s.t.} & \sum_{k=1}^K A_k x_k = x_0, \\
& x_k \in \mathcal{X}_k,\forall k = 1, \hdots, K.
\end{array}
$
& \makecell{ $g_k(\cdot)$ is either convex \\ or Lipschitz continuously \\ differentiable. \\  $\nabla l(x)$ is Lipschitz continuous. \\ $\mathcal{X}_k$ is convex and compact. \\ $A_k$ has full column rank. } & \makecell{Subsequential \\ convergence to \\ stationary points}. & \cite{DBLP:journals/siamjo/HongLR16} \\
\hline
$
\begin{array}{rl}
\min & F(z) + G(y) + H(x, y) \\
\mbox{s.t.} & Ax - z = 0.
\end{array}
$
& \makecell{$F$ and $G$ are proper \\ lower semicontinuous. \\ $\nabla H$ is Lipschitz continuous. \\ $A$ is surjective. } & \makecell{Subsequential \\ convergence to \\ KKT points. } & \cite{doi:10.1137/18M1190689} \\
  \bottomrule
\end{tabular}
\end{table}

There have also been extensions of nonconvex ADMM schemes to the linearized
regime~\cite{DBLP:journals/corr/LiuSG17}, nonlinear equality-constrained
settings~\cite{DBLP:journals/corr/WangZ17h}, amongst
others~\cite{DBLP:journals/siamis/YangPC17, 2017arXiv170201850G,
2015arXiv150503063W, DBLP:journals/corr/0002DFSG16}. Despite all of these theoretical achievements on nonconvex ADMM, we point out that no scheme
introduced above can guarantee the convergence or even the boundedness of the
iterates when applying ADMM or its variants to the following formulation with
both blocks constrained and one being nonconvex:
\begin{align} \label{formofinterest}
\begin{aligned}
\min &  \quad F(x) + G(y) \\
\st &\quad x - y = 0, \\
& \quad x \in X \subsetneqq \bR^N, y \in Y \subsetneqq \bR^N,
\end{aligned}
\end{align}
where $F$ is a quadratic function, $G$ is smooth and convex, $X$ is a
nonconvex set defined by a quadratic equality constraint, while $Y$ is a convex set.
Formulation \eqref{formofinterest} is our focus in this paper because by
reformulating the problem of interest in this way and applying ADMM-type schemes, each subproblem will be tractable and may possibly allow for a closed-form
solution. Jiang et al.\cite{DBLP:journals/corr/JiangLMZ16} discussed how 
ADMM schemes may be applied to \eqref{formofinterest} to allow for deriving convergence guarantees. Yet it requires changing the formulation of
\eqref{formofinterest} through the addition of an unconstrained
auxiliary block and requires penalizing the auxiliary variable in the
objective function.


{\bf Motivation and contributions.} Despite the breadth of prior
research, less is known regarding the nature of solutions and  tractable convergent schemes for continuous reformulations
of~\eqref{L0min}. Motivated by this gap and inspired by \cite{feng2018complementarity}, we consider an equivalent MPCC reformulation of \eqref{L0min}:
\begin{equation}\label{mpcc}
	\begin{array}{rl}
	\displaystyle \min_{{x^{\pm},\xi}} & \quad {f(x^+ - x^-)}  + \gamma e^T(e - \xi)  \\
	\st  & \quad {A(x^+ - x^-)}  \geq  b, \quad 
	 (x^+ + x^-)^T \xi  = 0,\\
	& \quad x^+, x^-  \geq 0, \quad
	 0 \leq \xi_i   \leq  1,  \mbox{ for } i = 1,\hdots,n.
	\end{array}
	\end{equation} 
In particular, we focus on characterizing stationary points of \eqref{mpcc} as well as developing {\bf tractable} convergent scheme that may recover such solutions.

\noindent {\bf (i) Regularity properties and characterization of KKT points.}
In Section~\ref{sec:MPCC}, we show that a feasible point of the MPCC
reformulation satisfies the Guignard constraint qualification (GCQ). Under
convexity of $f$, we derive an equivalence between first-order KKT points and
local minimizers.

		\noindent {\bf (ii) ADMM schemes with tractable subproblems.} In Sections~\ref{sec: ADMM} and \ref{sec: perturbADMM}, we propose two ADMM schemes to exploit the special structure of the MPCC: (ADMM$_{\rm cf}^{\mu, \alpha, \rho}$) and (ADMM$_{\rm cf}$). In particular, we reformulate the MPCC \eqref{mpcc} in the form of
\eqref{formofinterest} and apply the ADMM frameworks. The algorithms require resolving two subproblems at each iteration where, one is convex and the other, while nonconvex, is shown to possess a hidden convexity property~\cite{Ben-Tal1996}, and allow for closed-form solutions. In the perturbed proximal ADMM scheme (ADMM$_{\rm cf}^{\mu, \alpha, \rho}$), the perturbation technique (inspired by Hajinezhad and Hong
\cite{Hajinezhad2019}) allows us to show subsequential convergence. We also show that a limit point of this scheme is a perturbed KKT point where the inexactness depends on the choice of the perturbation parameters of the algorithm.


	\noindent {\bf (iii) Numerics.} In Section~\ref{sec:num}, we present
some preliminary numerical experiments showing that the tractable ADMM schemes are more scalable than their standard counterpart and ADMM$_{\rm cf}$ competes well with other solution methods for a special case of the $\ell_0$-minimization problem.

	\textbf{Notation.} We let $e$ denote $(1; \hdots;1)$ for an appropriate dimension. 
	Given a set $Z$ and
	a vector $z$, $\bk1_Z(z) = 0$ if $z \in Z$ and $\infty$ otherwise.
The requirement $a \perp b$ is equivalent to $a_i b_i = 0$ for $i = 1,
	\hdots, n.$ The matrix $I_n$ denotes the $n-$dimensional identity
		matrix. $[1,n] \triangleq
		\{1,2,\hdots,n\}$. $|S|$ denotes the cardinality of
		set $S$. $(a)_i$ or $[a]_i$ denote the $i$th entry of vector $a$. We may also use $a_i$ to denote $i$th entry of vector $a$, but \uvs{often} $a_i$ \uvs{may have other connotations} such as \uvs{the} $i$th iterate in an algorithm, which will be specified.  \uvs{Let} the support
set of $x$ be defined as $\mbox{supp}(x) \triangleq \{ i \in
\{1,\hdots,n\} \mid x_i \neq 0 \}$. \yxIII{For any vector $ z \in \bR^N$, positive semidefinite matrix $M \in \bR^{N \times N}$, $\| z \|_M^2 \triangleq z^T M z.$} 

\section{Properties of the MPCC reformulation}\label{sec:MPCC}
In Section~\ref{sec:21}, we revisit the MPCC
formulation \eqref{mpcc} and study both its  regularity
properties (Section~\ref{sec:CQ&KKT}) and  the relation between KKT points and local minimizers (Section~\ref{sec:23}). 
\subsection{Complementarity-based reformulations}\label{sec:21}
	In~\cite{feng2018complementarity}, several complementarity-based reformulations of \eqref{L0min} are introduced:\\
{\bf Half-complementarity}
	\begin{equation}\label{Hadamard}
	\begin{array}{rlr}
	\displaystyle \min_{x,\xi} & \quad f(x)  + \gamma e^T(e - \xi) \\
	\st & \quad Ax \geq b, \ x_i \xi_i  =  0, \\
	& \quad 0 \leq \xi_i   \leq  1,  \mbox{ for } i = 1,\hdots,n. &
	\end{array}
	\end{equation}
	{The term ``half-complementarity'' arises from noting that the equality constraint may be recast as \uvs{$ x \perp \xi \geq 0$}.}\\ 
{\bf Full-complementarity}
	\begin{equation}\label{mpcc-unrel}
	\begin{array}{rl}
	\displaystyle \min_{{x,x^{\pm},\xi}} & \quad {f(x)}  + \gamma e^T(e - \xi) \\
	\st & \quad Ax  \geq  b, x^+ - x^-  = x, \\
	 & \quad (x^+)^Tx^-  = 0,\ (x^+ + x^-)^T \xi  = 0, \\
	& \quad x_i^+, x_i^-  \geq 0, \  0 \leq \xi_i   \leq  1,  \mbox{ for } i = 1,\hdots,n. 
	\end{array}
	\end{equation} 
	where $x^+,x^-,\xi \in \Real^n$. \eqref{mpcc-unrel} may be further simplified by relaxing {$(x^+)^Tx^-=0$}, resulting in \eqref{mpcc}. It can be formally shown that \eqref{mpcc} is a tight relaxation of
\eqref{mpcc-unrel} implying that \vvs{a solution of} \eqref{mpcc} is \vvs{a minimizer of }\eqref{mpcc-unrel} (\diff{See Lemma~\ref{Lm:TightRelax} in the Appendix}). 
Since equivalence between
\eqref{mpcc-unrel} and \eqref{L0min} has been established~\cite{feng2018complementarity}, the tightness of relaxation indicates equivalence
between \eqref{mpcc} and \eqref{L0min}. Moreover, the
following result shows that local minimizers of \eqref{L0min} can also be recovered
by local minimizers of \eqref{mpcc}. 
\begin{lemma}\label{EquivForm}
\rm \uvs{Given $\hat x, \hat x^+, \hat x^-, \hat \xi \in \bR^n$ \uvs{such that} $ \hat x = \hat x^+ - \hat x^-$ and  $(\hat x^+; \hat x^-; \hat \xi)$ is a local minimum of \eqref{mpcc}.  
Then $\hat x$ is a local minimum of \eqref{L0min}.}
\end{lemma}
\begin{proof}
Suppose $\sZ$ \uvs{denotes} the feasible region of \eqref{mpcc}. Since $\hat z
\triangleq (\hat x^+; \hat x^-; \hat \xi)$ is a local minimum of
\eqref{mpcc}, $\hat z \in \sZ$ and there exists an open neighbourhood $\sN
\triangleq B(\hat z, r) \triangleq \{ z \in \bR^{3n} \mid \| z - \hat z \| < r \}$ \uvs{such that for all} $(x^+; x^-; \xi) \in \sN
\cap \sZ$, $ f(x^+ - x^-) + \gamma e^T(e - \xi) \geq f( \hat x^+ - \hat
x^-) + \gamma e^T(e - \hat \xi)$. \uvs{Let} $X \triangleq \{ x \mid Ax \geq b
\}$. It suffices to show that \vvs{(a)} $\hat x \in X$  and \vvs{(b)} there exists an
open neighbourhood $\sU \ni \hat x$ \uvs{such that} for all $x \in \sU \cap
X$, $f(x) + \gamma \| x \|_0 \geq f(\hat x) + \gamma \| \hat x \|_0$. \uvs{Of these, \vvs{(a)} holds immediately by noting that $A \hat x = A(\hat x^+-\hat x^-) \geq b$ where the inequality follows from the feasibility of $(\hat x^+; \hat x^-; \hat \xi)$ with respect to \eqref{mpcc}.} \uvs{Suppose $\sU$ is defined as a sufficiently small set such that the following hold:} (i) For all $x \in \sU$,
$f(x) \geq f(\hat x) - \gamma$, \uvs{a consequence of the continuity of $f$};(ii) For all $x \in
\sU \uvs{ \ \cap X}$, $\hat x_i \neq 0 \Rightarrow x_i \neq 0$, $\forall i = 1,\hdots,n$;
(iii) $\sU \subseteq B(\hat x, r)$. Then (ii) implies $\mbox{supp}(x)
\supseteq \mbox{supp}(\hat x)$ for all $x \in \sU \uvs{ \ \cap X \ }$ (\uvs{or $\|\hat{x}\|_0 \leq \|x\|_0$}). \yxII{Therefore, the local optimality of $\hat x$ can be shown through the following two cases:} (I). If $\bar x \in \{ x \in \sU \cap
X \mid \mbox{supp}(x) \supsetneq \mbox{supp}(\hat x) \}$, then $\| \bar x \|_0 \geq \|
\hat x \|_0 + 1$ implying that $f(\bar x) + \gamma \| \bar x \|_0
\overset{(i)}{\geq} f(\hat x) - \gamma + \gamma (\| \hat x \|_0 + 1) =
f(\hat x) + \gamma \| \hat x \|_0$; (II). If $\bar x \in \{ x \in \sU
\cap X \mid \mbox{supp}(x) = \mbox{supp}(\hat x) \}$, \uvs{then} let $\bar x^+_i \uvs{ \ \triangleq \ } \hat x^+_i +
\max\{\bar x_i - \hat x_i ,0\}$, $\bar x^-_i \uvs{ \ \triangleq \ } \hat x^-_i - \min \{\bar x_i
- \hat x_i ,0\}$ for $i = 1,\hdots,n$. Then we see that
$\bar x = \bar x^+ - \bar x^-$ and $(\bar x^+; \bar x^-; \hat \xi) \in \sN
\cap \sZ$. 
Therefore, $f(\bar x^+ - \bar x^-) + \gamma e^T(e - \hat
\xi) \geq f(\hat x^+ - \hat x^-) + \gamma e^T(e - \hat \xi)$ implying
$f(\bar x^+ - \bar x^-) \geq f(\hat x^+ - \hat x^-)$ or $f(\bar x)
\geq f(\hat x)$.  \uvs{It follows that} $f(\bar x) + \gamma \| \bar x \|_0 \geq f(\hat
x) + \gamma \| \hat x \|_0$.  \qed \end{proof}

While \eqref{L0min} can now be reformulated as a continuous problem, \eqref{mpcc} is still an MPCC. 
It may be recalled that MPCCs are ill-posed nonconvex nonlinear programs in that 
standard regularity
conditions (such as LICQ or MFCQ) may fail to hold at any feasible
point~\cite{luo96mathematical}. Moreover, global resolution of such
problems is generally challenging. 
We now discuss what constraint qualifications do hold at a feasible point of \eqref{mpcc}.

\subsection{Constraint Qualifications} \label{sec:CQ&KKT}
In this subsection, we analyze whether regularity conditions hold at feasible points for the simplified full complementarity formulation \eqref{mpcc}. This allows for stating necessary conditions of optimality. 
Recall that some common CQs are related as follows.
\begin{align}\label{CQs}
 \textbf{(I)} \ \mbox{LICQ} \Rightarrow \mbox{CRCQ} \quad \mbox{ and } \quad 
 \textbf{(II)} \ \mbox{LICQ} \Rightarrow \mbox{MFCQ} \Rightarrow \mbox{ACQ} \Rightarrow \mbox{GCQ}
\end{align}
The first relation is obvious from the definition of LICQ and CRCQ (See \cite[Page 262]{facchinei2007finite}) while the proof of the second relation may be found in \cite{Burke2012note}. In the context of the half-complementarity formulation \eqref{Hadamard}, the constant rank constraint qualification (CRCQ)  is proven to hold at points satisfying certain nondegeneracy property while LICQ may fail~\cite{feng2018complementarity}. In this section, we focus on the simplified full-complementarity formulation~\eqref{mpcc}. It can be shown that GCQ may hold at every feasible point while ACQ may fail.

We begin our discussion with some definitions.
Suppose $g: \Real^n \rightarrow
	\Real^p$ and $h: \Real^n \rightarrow \Real^q$ are continuously differentiable functions while $\Omega$ is a set defined as follows.
	\begin{align}\label{Omega}
	\Omega \triangleq \left\{ x \in \Real^n : g(x) \leq 0 , h(x) = 0 \right\}.
	\end{align}
Then the tangent cone
	$T_{\Omega}(x^*) $ and {linearized} cone $L_{\Omega}(x^*)$ of
	$\Omega$ at $x^*$ and {the ACQ and the GCQ are} defined as follows:
	\begin{definition}[{\bf Abadie and Guignard CQ
		(ACQ, GCQ)}] \rm
If $\mathcal{I}(x^*) = \{i:g_i(x^*) = 0\}$, then
	$T_{\Omega}(x^*) $ and $L_{\Omega}(x^*)$ of
	$\Omega$ at $x^*$ are defined as follows:
	\begin{align} \label{def-Tomega}
	T_{\Omega}(x^*) & \triangleq \left\{d : \exists \{x_k\} \subseteq
	\Omega, \{t_k\} \downarrow 0, \mbox{ s.t. } x_k \rightarrow
	x^* \mbox{and } d = \lim_{k \rightarrow \infty} \frac{x_k-x^*}{t_k}
	\right\} \\
\label{def-Lomega}
	L_{\Omega}(x^*) & \triangleq \left\{d: \nabla g_i(x^*)^T d \leq 0, 
	\forall i \in \mathcal{I}(x^*), \nabla h^T_j(x^*)d = 0, j = 1,\hdots,q
		\right\}.
	\end{align}
	Then $x^*$ satisfies the \textbf{Abadie Constraint Qualification} (ACQ) iff $T_{\Omega}(x^*) = L_{\Omega}(x^*)$. {Further,
	$x^*$ satisfies the \textbf{Guignard Constraint Qualification} (GCQ) 
	iff $(T_\Omega(x^*))^*$ $ = (L_\Omega(x^*))^*,$ where for a cone $C
		\subseteq \Real^n$, $C^* \triangleq \{v: d^Tv \leq 0, \ \forall
		d \in C\}.$}
	\end{definition} 

Next, we prove that the GCQ holds at every feasible point of \eqref{mpcc}.
\begin{lemma}[\us{\bf GCQ holds at feasible points}] \label{GCQ} \rm
	Consider the problem \eqref{mpcc} and consider a 
	feasible point $x = (x^+;x^-;\xi)$. Then the GCQ holds at this point.
\end{lemma}
\begin{proof}
For the point $x = (x^+;x^-;\xi)$, define
\begin{align} \label{def-A-E}
 & A^T \triangleq  (a_1,\hdots,a_m) \mbox{ and }  {E(x)} \triangleq \{ i
	: a_i^T (x^+ - x^-) = b_i \}, \\
\label{def-Sparse sets}
& \begin{aligned}
& S(x) = \left\{ i
	: x^+_i = x^-_i = 0 \right\}, \\
& S_0(x) = 
 \left\{ i \in S(x)
	: \xi_i = 0 \right\}, S_1(x) =
 \left\{ i \in S(x)
	: \xi_i = 1 \right\}.
\end{aligned}	
\end{align}
In addition, define cones $C_1(x)$ and $C_2(x)$ as
\begin{align}
\notag
{C_2(x)} & \triangleq \left\{ \pmat{d_1\\d_2\\d_3}:
	 \begin{aligned}
	 (d_1)_i  & = 0, (d_2)_i   = 0, \forall i \in S(x) \setminus S_0(x);\\
	(d_1)_i & \geq 0 , \forall i \in S_0(x) \cup (S(x)^c \cap \{i : x^+_i = 0\} ) ; \\
	 (d_2)_i & \geq 0 , \forall i \in S_0(x) \cup (S(x)^c \cap \{i : x^-_i = 0\} ) ; \\
	 (d_3)_i & \geq 0, \forall i \in S_0(x); 
	 (d_3)_i  \leq 0, \forall i \in S_1(x);\\
	 (d_3)_i  & = 0 , \forall i \in S(x)^c; 
	 a_j^T d_1  - a_j^T d_2  \geq 0, \forall j \in {E(x)}
	\end{aligned}
	\right\}, \\
	\label{def-C1} 
C_1(x) & \triangleq C_2(x) \cap \left\{ d = (d_1; d_2; d_3):
     \left[ (d_1)_i + (d_2)_i \right] (d_3)_i = 0, \forall i \in S_0(x) \right\},
\end{align}
respectively where it may be noted that $C_1(x)$ is characterized by an
extra constraint $ \left[ (d_1)_i + (d_2)_i \right] (d_3)_i  = 0$ for
all $i \in S_0(x)$. Further, denote
	\begin{align*}
X & \triangleq \left\{
	(y^+;y^-;\zeta):
	\begin{aligned}
	& y^+,y^-,\zeta \in \Real^n, (y^+ + y^-)^T \zeta = 0, A(y^+ - y^-)
		\geq b, \\ 
	& y^+ \geq 0, y^- \geq 0, 0 \leq \zeta_i \leq 1, \forall i = 1,\hdots,n,
	\end{aligned}
	\right\}.
\end{align*}
 We proceed to show the following.\\
	\noindent {(i). $T_{X}(x) = {C_1(x)}$:}
	Suppose $d \in T_{X}(x)$. Then there exist sequences $\{x_k\}$ and
	$\{t_k\}$ such that  $\{x_k\} \subseteq
	X, x_k \rightarrow x $, $\{t_k\} \downarrow 0$ and $d = \lim_{k \rightarrow \infty}
	\frac{x_k-x}{t_k}$. Denote $x_k 
	\triangleq (x_{(k)}^+;x_{(k)}^-;\xi_{(k)} )$, where $x_{(k)}^+, x_{(k)}^-, \xi_{(k)} \in \bR^n$. Suppose that $d \triangleq (d_1;d_2;d_3), d_1,d_2,d_3 \in
	\Real^n$. Based on feasibility of $x_k, \forall k \geq 1$ and the fact that $x_k \rightarrow x$, we
	may claim the following:
	\begin{align*}
	\forall i & \in S(x) \setminus S_0(x), \exists K_1, \mbox{s.t.},
	\forall k \geq K_1, (x_{(k)}^+)_i = (x_{(k)}^-)_i = 0 \\
&  \implies (d_1)_i = (d_2)_i = 0, \forall i \in S(x) \setminus S_0(x) \\
	\mbox{ if } i & \in S_1(x), \mbox{ then } \xi_i = 1 \mbox{ and } (\xi_{(k)})_i
	\leq 1, \forall k \\
	&  \implies (\xi_{(k)})_i - {\xi_i} \leq 0, \forall k, (d_3)_i \leq 0, \forall i \in S_1(x).
	\end{align*}
	{Similarly we may claim the following:} 
	\begin{align*}
	& \qquad \forall i \in S(x)^c, \exists K_2, \mbox{ s.t.}  \forall k \geq K_2,
	(\xi_{(k)})_i = 0, (x_{(k)}^+)_i \geq 0, (x_{(k)}^-)_i \geq 0 \\
	&\implies  (d_3)_i = 0, \forall i \in S(x)^c; 
 \ (d_1)_i \geq 0 ,
	\forall i \in S(x)^c \cap \{ i: x^+_i = 0\}; \\
	& \mbox{ and } (d_2)_i \geq 0, \forall i \in S(x)^c \cap \{ i: x^-_i = 0\}.
	\end{align*}
	{For indices $i \in S_0(x)$, the following holds:}
	\begin{align*}
	& x^+_i = x^-_i = \xi_i  = 0  \Rightarrow  (x_{(k)}^+)_i - x^+_i \geq 0, (x_{(k)}^-)_i - x^-_i \geq 0, (\xi_{(k)})_i - \xi_i  \geq 0, \forall k,\\
	& \implies (d_1)_i \geq 0, (d_2)_i \geq 0, (d_3)_i \geq 0. \\
& 	\left[ (x_{(k)}^+)_i + (x_{(k)}^-)_i \right] (\xi_{(k)})_i = 0, \forall k \\
&	\implies (x_{(k)}^+)_i + (x_{(k)}^-)_i = 0, \mbox{ or } (\xi_{(k)})_i = 0, \mbox{ inf. often}; \\
&	 \implies (d_1)_i + (d_2)_i = 0, \mbox{ or } (d_3)_i = 0 
 \iff \left[ (d_1)_i + (d_2)_i \right] (d_3)_i = 0 \\
& \mbox{Furthermore,} \\
& \forall j \in {E(x)}, a_j^T x^+ - a_j^T x^- = b_j, a_j^T x_{(k)}^+ -
	a_j^T x_{(k)}^- \geq b_j,  \mbox{ for all } k \ \geq \ 1  \\
&	 \implies a_j^T (x_{(k)}^+ - x^+)  - a_j^T (x_{(k)}^- - x^-) \geq 0,  \forall j \in {E(x)} \mbox{ and }  k \ \geq \ 1 \\
& \implies a_j^T d_1 - a_j^T d_2 \geq 0, \forall j \in {E(x)}. 
	\end{align*}
	Therefore, {we may conclude from \eqref{def-C1} that} $d \in
	{C_1(x)}$
	and $T_{X}(x) \subseteq {C_1(x)}$. 
	
	We now
	proceed to show that ${C_1(x)} \subseteq T_X(x)$.  Choose any $d
	\in {C_1(x)}$. Then based on property of
			${C_1(x)}$, it is easy to see that we may choose $\lambda$ large enough such that $x + d/(k\lambda) \in X,
	\forall k \geq 1$. Let
	$x_k \triangleq x + d/(k\lambda), t_k \triangleq 1/(k\lambda)$ for all $k \geq 1$, implying
	that $\{x_k\} \subseteq X, x_k \rightarrow x, t_k \downarrow 0, d = \lim_{k \rightarrow \infty}
	\frac{x_k-x}{t_k}$ implying that $d  \in  T_{X}(x)$ which further implies ${C_1(x)} \subseteq  T_{X}(x).$\\
	\noindent {(ii). $L_X(x) = C_2(x)$:} The set $X$ \yx{contains the following active constraints.} 
	\begin{align*}
	& -y^+_i \leq 0, \forall i \in S(x) \cup \{i \in S(x)^c: x^+_i = 0 \}
	;\\
	& -y^-_i  \leq 0, \forall i \in S(x) \cup \{i \in S(x)^c: x^-_i = 0 \}; \\
	& -\zeta_i  \leq 0,  \forall i \in S_0(x) \cup S(x)^c; \quad
	\zeta_i   \leq 1,  \forall i \in S_1(x); \\
	  & a_i^T(y^+ - y^-)  \geq b_i, \forall i \in {E(x)}; \quad  
	  (y^+ + y^-)^T\zeta  = 0.
	\end{align*}
	{This allows for defining the linearized cone $L_X(x)$ at $x \in
		X$.} 
	\begin{align}\label{LinearCone-GCQ}
	L_X(x) { \triangleq } \left\{\pmat{d_1 \\ d_2 \\ d_3}: 
	\begin{aligned}
	 & -(d_1)_i   \leq 0, \forall i \in S(x) \cup \{i \in S(x)^c: x^+_i = 0 \};\\
	& -(d_2)_i   \leq 0, \forall i \in S(x) \cup \{i \in S(x)^c: x^-_i = 0 \};\\
	& (d_3)_i  \leq 0, \forall i \in S_1(x); 
	-(d_3)_i \leq 0, \forall i \in S_0(x) \cup S(x)^c;\\
	&a_i^T(d_1 - d_2)  \geq 0, \forall i \in {E(x)};\\ 
	& \xi^T(d_1+d_2) + (x^+ + x^-)^Td_3  = 0. 
	\end{aligned}
	\right\}
	\end{align}
	{Suppose} $d \in L_X(x)$. Then the following holds:
	\begin{align}	\notag & \qquad \xi^T(d_1+d_2) + (x^+ + x^-)^Td_3 = 0 \\
& \iff \sum_{i \in S(x) \setminus S_0(x)} \xi_i[(d_1)_i+(d_2)_i] + \sum_{i \in S(x)^c} (d_3)_i(x^+_i + x^-_i) = 0 \notag \\
	\label{eq1-GCQ}
	& \iff (d_1)_i=(d_2)_i = 0, \forall i \in S(x) \setminus S_0(x); \quad (d_3)_i = 0, \forall i \in S(x)^c,
	\end{align}
	{where the} first equivalence follows from the definition of
	$S_0(x)$ and $S(x)$ while the second follows from noting that $(d_1)_i \geq 0, (d_2)_i \geq 0, \xi_i > 0, \forall i \in S(x) \setminus S_0(x)$ and $(d_3)_i \geq 0, x^+_i + x^-_i > 0, \forall i \in S(x)^c$.
	Therefore, by replacing $\xi^T(d_1+d_2) + (x^+ + x^-)^Td_3 = 0$ with \eqref{eq1-GCQ} in the representation \eqref{LinearCone-GCQ}, we observe that $L_X(x)  = {C_2(x)}$.

\noindent (iii). {We conclude the proof by showing that $ {C_2(x) =
\mbox{cl}(\mbox{conv}(C_1(x)))}$. Since $C_2(x)$ is a polyhedral cone, it is closed and convex.} Furthermore, by definition,
	$C_2(x) \supseteq C_1(x)$, implying that $C_2(x) \supseteq
		\mbox{cl}({\mbox{conv}}(C_1(x)))$. To prove the reverse direction,
	choose any vector $d \triangleq (d_1;d_2;d_3) \in {C_2(x)}$
		where $d_1,d_2,d_3 \in \Real^n$. It is easy to verify that both
		vectors $\tilde{d} \triangleq (\bf{0_{n \times 1}}; \bf{0_{n
				\times 1}}; 2d_3)$ and $\hat{d}  \triangleq (2d_1; 2d_2;
					\bf{0_{n \times 1}})$ are in ${C_1(x)}$. Note
				that $d = \frac{1}{2} \tilde{d} + \frac{1}{2} \hat{d}
				\in {\mbox{cl}}({\mbox{conv}}(C_1(x)))$.
				Therefore, $C_2(x) \subseteq
				\mbox{cl}(\mbox{conv}(C_1(x)))$.
				
 By (iii) $L_X(x) = \mbox{cl}(\mbox{conv}(T_X(x)))$,
 implying that $T_X(x)^* = L_X(x)^*$. \qed
\end{proof}
\begin{remark}
(i) At a feasible point $x = (x^+;x^-;\xi)$ such that $[x^+ + x^-]_i = 0$ and $\xi_i = 0$ for some index $i$, ACQ may fail to hold. In fact, it is very likely that $T_X(x) = C_1(x) \subsetneqq C_2(x) = L_X(x)$. On the other hand, at all other points, $S_0(x) = \emptyset$ and ACQ holds.
(ii) KKT conditions are necessary at local minimum.
\end{remark}
\subsection{KKT conditions and local optimality}\label{sec:23}
In this subsection, we discuss the relation between (first-order) KKT conditions and local optimality. We begin with the definition of KKT conditions.
	\begin{definition}[{\bf KKT conditions}] \label{first order KKT} \rm
Consider the problem $\{ \min_{x \in \Omega}F(x) \}$,\\ where $F(x)$ is a continuously differentiable {function} 
and $\Omega$ is defined in \eqref{Omega}.
Suppose $x^*$ {denotes} a feasible solution of $\Omega$. Then $x^*$ satisfies the first-order KKT conditions if and only if 
there exists $\lambda \in
	 \Real^p_+, \mu \in \Real^q$ such that 
\begin{align}\label{kkt}
\begin{aligned}
\nabla F(x^*) + \sum_{i=1}^p \lambda_i \nabla g_i(x^*) + \sum_{j=1}^q
\mu_j \nabla h_j(x^*) & = 0, \\
\lambda_ig_i(x^*) & = 0, \qquad \forall i = 1,\hdots,p.
\end{aligned}
\end{align}
	\end{definition}

{
Then by \yxII{Definition}~\ref{first order KKT}, a point $x \triangleq (x^+; x^-; \xi)$ satisfies the {first-order KKT conditions of problem \eqref{mpcc} if there exist multipliers {$(\mu,\beta_1,\beta_2,\beta_3, \beta_4, \pi) \in \bR \times \bR^n \times \bR^n \times \bR^n \times \bR^n \times \bR^m$ }
		such that the following conditions hold:}
\begin{subequations} \label{KKT-F}
\begin{align} \label{KKT-F-1}
0  = \pmat{\nabla_x f (x^+ - x^-) \\ -\nabla_x f(x^+ - x^-) \\ -\gamma e} +
\mu \pmat{ \xi \\ \xi \\ x^+ + x^-} + \pmat{ -\beta_1 - A^T \pi \\ -\beta_2 + A^T \pi \\ \beta_4 - \beta_3} , \\
\label{KKT-F-2}
0  \leq \beta_1   \perp x^+ \geq 0, \\
\label{KKT-F-3}
  0  \leq \beta_2  \perp x^- \geq 0, \\
\label{KKT-F-4} 
  0  \leq \beta_3  \perp \xi \geq 0, \\
\label{KKT-F-5}
0  \leq \beta_4  \perp e - \xi  \geq 0,\\
\label{KKT-F-6}
0  \leq \pi \perp  A(x^+ - x^-) - b  \geq 0, \\
\label{KKT-F-7}
(x^+ + x^-)^T \xi  = 0.
\end{align}
\end{subequations}

Before presenting the main result, we point out  a non-degeneracy property of KKT points.
%

\begin{lemma}[{\bf Nondegeneracy of first-order KKT points}]\label{kkt-nondeg}
\rm Consider a point $x = (x^+;x^-;\xi)$ and a set of multipliers $(\mu,\beta_1,
		\beta_2,\beta_3, \beta_4, \pi)$ that satisfy the first-order
		KKT conditions \eqref{KKT-F} of \eqref{mpcc}. Then {$x$ satisfies the nondegeneracy property}:
\begin{align}
\label{nondegpt}
[x^+ + x^-]_i = 0 \Rightarrow \xi_i = 1.
\end{align}
\end{lemma}
\begin{proof}
 Suppose that $(x^+;x^-;\xi)$ verifies KKT
conditions \eqref{KKT-F} with multipliers $\mu, \beta_1, \beta_2,
		   \beta_3, \beta_4, \pi$. Then, by \eqref{KKT-F-1}, {we have that}
$
(x^+ + x^-)_i = 0 \Rightarrow (\beta_4 - \beta_3)_i = \gamma > 0. $
But for a given $i$, for both $[\beta_4]_i$ and $[\beta_3]_i$ to be
positive, we require that both $[\xi]_i = 0$ and $[1-\xi]_i = 0$ hold,
	which is impossible. It follows that the only possibility is that
	$[\beta_4]_i = \gamma$ and $[\beta_3]_i = 0$, implying that $[\xi]_i =
	1$.
It follows that $(x^+;x^-;\xi)$ satisfies the property \eqref{nondegpt}. \qed
\end{proof}


Lemma~\ref{kkt-nondeg} leads to the rather surprising
equivalence between local minimizers and (first-order) KKT points.

{\begin{theorem}[{\bf Equivalence between local minimizers and KKT
	points}] \label{SOC}
\em
Consider problem \eqref{mpcc}, and let $x = (x^+;x^-;\xi)$ denote a
feasible point. Assume that $f$ in \eqref{mpcc} is convex. Then the following statements are equivalent:
\be
\item[(a)] $x$ is a local minimizer of \eqref{mpcc};
\item[(b)] {There exist $\mu \in \Real$, $\beta_1, \beta_2,
	\beta_3,\beta_4 \in \Real^n$, and $\pi \in \Real^m$ such that the
first-order KKT conditions \eqref{KKT-F} hold;}
\ee
\end{theorem}
\begin{proof}
\noindent {\bf (a)$\Rightarrow$(b).} 
This is true because GCQ holds at every feasible point by Lemma \ref{GCQ}. \\
{\bf (b)$\Rightarrow$(a).} Suppose that $x = (x^+;x^-;\xi)$
satisfies KKT conditions \eqref{KKT-F} with multipliers $(\mu,
		\beta_1,\beta_2,\beta_3,\beta_4,\pi)$. Then by the nondegeneracy
property of a KKT point (\us{Lemma~\ref{kkt-nondeg}}), the set
\us{$\{1,\hdots,n\}$} can be partitioned into the following two sets, as in the same fashion when proving  the \uvs{CQ}:
$S(x)  \triangleq   \{i \in \{1,\hdots,n\}: x_i^+ = x_i^- = 0, \xi_i = 1\}$  and
$S^c(x)  \triangleq  \{i \in \{1,\hdots,n\}: x_i^+ + x_i^- > 0, \xi_i = 0\}.$
{We denote that  $A = (a_1,\hdots,
			a_n)$ (Note different notation from~\eqref{def-A-E})}. Then \eqref{KKT-F-1} implies
\begin{align*}
\left.
\begin{aligned}
{(\nabla_x f(x^+ - x^-))_i} - a_i^T \pi = \us{(\beta_1)_i} \geq 0
	\\
{-(\nabla_x f(x^+ - x^-))_i} + a_i^T \pi = \us{(\beta_2)_i} \geq 0
\end{aligned} \right\} \quad
	\forall i \in S^c(x). 
\end{align*}
because $\xi_i = 0$ for all $i \in S^c(x)$.
\us{Consequently, $(\beta_1)_i = -(\beta_2)_i$ where $\beta_1$ and
	$\beta_2$ are nonnegative. It follows that $(\beta_1)_i = (\beta_2)_i = 0$}, and 
\begin{align}\label{Eq 1: KKT = localmin}
(\nabla_{x} f(x^+ - x^-))_i = a_i^T \pi, \qquad \forall i \in S^c(x).
\end{align}	
We {proceed to prove} that $(x^+;x^-)$ is {a} global minimizer of the following program:
\begin{align}\label{CP}
\min & \ \tilde{f}(z) \triangleq f( z^+ - z^-), \quad \st \ z = (z^+;z^-) \in \tilde X(x),
\end{align}
where
\begin{align*}
\tilde X(x) \triangleq  \left\{( z^+; z^-) \mid z^+, z^- \in \bR^n_+,
 A( z^+ - z^- )  \geq b,   
 z^+_i = z^-_i  = 0, \forall i \in S(x) \right\}.
\end{align*}
{Consider} any feasible point $(\tilde{x}^+; \tilde{x}^-)$ of \eqref{CP}. \vvs{By applying}~\eqref{Eq 1: KKT = localmin} and noticing \vvs{$x^{\pm}_i,\tilde x^{\pm}_i = 0$} $\forall i \in S(x)$ (by def.), $\pi^TA(x^+ - x^-) = \pi^Tb, \pi \geq 0$ (by~\eqref{KKT-F}), and $A(\tilde x^+ - \tilde x^-) - b \geq 0$,
\begin{align*}
& \quad \pmat{\nabla_x f(x^+ - x^-) \\ -\nabla_x f(x^+ - x^-)}^T \left[ \pmat{\tilde{x}^+ \\ \tilde{x}^-} - \pmat{x^+ \\ x^-} \right]  \\
& = \nabla_x f(x^+ - x^-)^T [ (\tilde{x}^+ - \tilde{x}^-) - (x^+ - x^-)] 
 \\
& = \sum_{i \in S^c(x)} ( \nabla_x f(x^+ - x^-) )_i [ (\tilde{x}^+_i - \tilde{x}^-_i) - (x^+_i - x^-_i) ] \\
& =  \sum_{i \in S^c(x)} a_i^T \pi  (\tilde{x}^+_i - \tilde{x}^-_i) -  \sum_{i \in S^c(x)} a_i^T \pi (x^+_i - x^-_i) \\
& =  \sum_{{i \in S(x) \cup S^c(x)}} a_i^T \pi  (\tilde{x}^+_i - \tilde{x}^-_i)  -  \sum_{{i \in S(x) \cup S^c(x)}} a_i^T \pi (x^+_i - x^-_i) \\
& = \pi^T [ A(\tilde x^+ - \tilde x^-) - b] \geq 0.
\end{align*}
{It follows that $(x^+;x^-)$ is a solution of \yxI{VI$(\tilde X(x), \nabla_x \tilde{f})$.} By convexity of $f$ (thus $\tilde{f}$) and $\tilde X(x)$, $(x^+;x^-)$ is a global minimizer of \eqref{CP}.}  
 {Since} $ \xi_i = 1$ for $i \in S(x)$, \us{by the separability of the
	objective and the structure of the constraint sets},  {it follows that} $(x^+;x^-;\xi)$ is a minimizer of the tightened \eqref{mpcc} as follow:
\begin{align*}
\min \ f({\tilde x}_1 - {\tilde x}_2) + \gamma e^T (e - {\tilde x}_3) \quad \st \ \us{({\tilde x}_1;{\tilde x}_2;{\tilde x}_3)}  \in \us{X_{\rm tight}(x)},
\end{align*}
where
\begin{align*}
X_{\rm tight}(x) \triangleq \left\{  \pmat{\tilde x_1\\ \tilde x_2\\ \tilde x_3} : 
\begin{aligned}
{\tilde x_1, \tilde x_2} & \geq 0, \ 
 0 \leq {\tilde x}_3  \leq e, \ 
 A({\tilde x_1 - \tilde x_2})  \geq b, \\
 (\tilde x_1)_i = (\tilde x_2)_i & = 0, \qquad\forall i \in S(x),\\
 (\tilde x_3)_i & = 0, \qquad\forall i \in S^c(x) 
\end{aligned} \right\}.
\end{align*}
If $X$ denotes the feasible region 
in \eqref{mpcc}, then we can take a sufficiently small neighborhood
of $x$, denoted by $\mathcal{N}(x)$, such that $X \cap {\cal N}(x) = X_{\rm tight}(x) \cap {\cal N}(x)$.
	 Since
	 $x = (x^+;x^-;\xi)$ is \us{a} global minimizer of {$f(\tilde x_1 - \tilde x_2) + \gamma e^T (e - \tilde x_3)$ over
	 $\us{X_{\rm tight}({x})}$}, it is a global minimizer of  $f(\tilde{x}_1 - \tilde{x}_2) + \gamma e^T (e - \tilde{x}_3)$ over the smaller set
		 $\mathcal{N}(x) \cap X_{\rm tight}(\us{x})$. Since
			 $\mathcal{N}(x) \cap X_{\rm tight}(\us{x})= {\cal N}(x)
			 \cap X$, it follows that $x$ is a local
			 minimizer of \eqref{mpcc}. \qed
\end{proof}

\begin{remark}
Note that while convexity of $f$ is  observed for many loss functions, it does not guarantee the overall convexity of the problem and \eqref{mpcc} is still a nonconvex problem.
\end{remark}

\section{Tractable ADMM frameworks} \label{sec: ADMM}
In this section we discuss how to use ADMM to efficiently address MPCC \eqref{mpcc}. In Section~\ref{sec:31}, we present a perturbed proximal ADMM framework for obtaining a suitably defined solution of
\eqref{mpcc} and show in Section~\ref{sec:32} that both of the ADMM subproblems can be solved tractably, of which, one can be recast as a convex program,
while the other can be resolved in closed form. In Section~\ref{Vanilla
ADMM}, a basic ADMM framework will be presented, along with a discussion regarding why we consider its perturbed proximal variant. A standard ADMM applied to an alternative formulation of
\eqref{mpcc} is introduced in Section~\ref{subsec: alterform}. Note that ADMM applied to this formulation is easier to analyze but does have computational disadvantages arising from the intractability of the subproblem.

	\subsection{A perturbed proximal ADMM framework}\label{sec:31}
	We may reformulate \eqref{mpcc} as follows. 
	\begin{align} \label{Reformulation}
	\min & \quad f(x^+ - x^-) + \gamma \sum_{i=1}^n(1-\xi_i) + \bk1_{Z_1}(w) + \bk1_{Z_2}(w).
	\end{align}
Recall that $f(x) = f_Q(x) + g(x)$, where $f_Q(x) \triangleq x^TMx + d^Tx$, $g(x)$ is convex  and smooth, $M \in \Real^{n \times n}$ is a symmetric matrix, and $d \in \Real^n$. Let $Z_1, Z_2,$ and $w$ be defined as 
	\begin{align}  \label{def-Z1Z2}
	\begin{aligned}
Z_1 & \triangleq \left\{ (x^+; x^-; \xi) : (x^+ + x^-)^T\xi = 0 \right\},  \\
	Z_2 & \triangleq \left\{ \pmat{x^+\\x^-\\ \xi} : 
	\begin{aligned}
	0 & \leq   \xi_i  \leq 1, \forall i\\
  0 & \leq x^+, x^-\\
  b & \leq A(x^+ - x^-)
	\end{aligned}
	\right\},
	\end{aligned}
	\end{align} 
and  $w \triangleq \pmat{x^+;x^-;\xi}$, respectively. We introduce separability into the objective by adding a variable {$y \triangleq (y^+;y^-;\zeta), y^+,y^-,\zeta \in \bR^n$} and imposing an additional linear constraint.
	\begin{align}\label{tractable decomposition}
	\min_{w = y}  \ f_Q(x^+ - x^-)  + \gamma \sum_{i=1}^n(1-\xi_i) + \bk1_{Z_1}(w) + { g(y^+ - y^-) }+ \bk1_{Z_2}(y).
	\end{align}
Note that \eqref{tractable decomposition} is in the form of \eqref{formofinterest}. The intuition behind this formulation is that by separating the nonconvex set $Z_1$ from the convex polytope $Z_2$, we may potentially obtain easier subproblems when applying a splitting method. We now define a perturbed augmented Lagrangian function as follows.
\begin{align*}
\tscrL_{\rho,\alpha}(w,y,\lambda) & \triangleq {f_Q(x^+ - x^-)} + \gamma \sum_{i=1}^n(1-\xi_i) + { g(y^+ - y^-)} \\
& + (1-\rho \alpha)\lambda^T(w - y  - \alpha \lambda) + \frac{\rho}{2} \| w - y \|^2,
\end{align*}
where $w \triangleq (x^+; x^-; \xi), \alpha > 0, \rho > 0$. The perturbed proximal ADMM algorithm is presented as Algorithm~\ref{pb-admm}, denoted as ADMM$_{\rm cf}^{\mu,\alpha,\rho}$, where ``cf'' stands for ``complementarity formulation'', and $\mu, \alpha, \rho$ are algorithm parameters. The perturbation technique is inspired by Hajinezhad and Hong~\cite{Hajinezhad2019}. Note that (ADMM$_{\rm cf}^{\mu,\alpha,\rho}$) reduces to a basic ADMM when $\mu = \alpha = 0$, which will be discussed in Section~\ref{Vanilla ADMM}. We refer the readers to Remark~\ref{rmk: stopcrit.0} and Remark~\ref{rm: stopcrit.} for discussion of the stopping criteria.

\begin{algorithm}
\scriptsize
\caption{A perturbed proximal ADMM scheme: ADMM$_{\rm cf}^{\mu,\alpha,\rho}$}
\label{pb-admm}
	\begin{enumerate}
	\item[(0)] Given $w_0, y_0,\lambda_0$; Choose $\alpha, \rho, \mu, \epsilon_0 > 0$ such that $\rho\alpha \in (0,1), ( \rho + \mu ) I + 4M \succ 0$, and set $k := 0$. 
	\item[(1)] Let $w_{k+1}, y_{k+1}, \lambda_{k+1}$ be given by the following:
	\begin{align} \tag{Update-1}
	w_{k+1} & :=  {\displaystyle
		\mbox{arg}\hspace{-0.03in} \min_{w \in Z_1}} \quad \tscrL_{\rho,\alpha}(w,y_k,\lambda_k) + \frac{\mu}{2}\| w- w_k \|^2, \\
	 	\tag{Update-2}
	y_{k+1}  &:=  \mbox{arg}\hspace{-0.03in}\min_{y \in Z_2} \quad \tscrL_{\rho,\alpha}(w_{k+1},y,\lambda_k),\\
	\tag{Update-3}
	\lambda_{k+1}&  :=  (1-\rho \alpha)\lambda_k + \rho \left( w_{k+1} - y_{k+1} \right).
	\end{align}
	\item [(3)] If $\max\{ \| \rho(y_{k+1} - y_k) + \mu (w_{k+1} - w_k) \|, \| \lambda_{k+1} - \lambda_k \|/\rho \} < \epsilon_0$, STOP; 
	else $k:=k+1$ and return to (1).
	\end{enumerate}
\end{algorithm}	

We observe that there are indeed some benefits by considering
decomposition \eqref{tractable decomposition} that separates the nonconvex domain
$Z_1$ and the convex polytope $Z_2$. It turns out that this approach reduces the difficulty of both subproblems of the ADMM framework. Next, (Update-1) and
(Update-2) are shown to be tractable\footnote{By saying that an
optimization problem is tractable we mean that it either has a closed-form solution or lies in the range of convex programs that are polynomially solvable. We refer the
readers to \cite{ben2001lectures} for detailed discussion.}.
	\subsection{Tractable resolution of {ADMM Updates}}\label{sec:32}
	We now show that (Update-1) possesses a hidden
	convexity property~\cite{Ben-Tal1996}, allowing for claiming tractability of (Update-1) and obtaining its closed form solution.
	\begin{proposition}[{\bf Tractability of Update-1}]\label{closedformpro}\rm
Recall that $f_Q(x) =  x^TMx + d^Tx$ where {$M$} may be a symmetric indefinite matrix with real eigenvalues given by $s_1, \hdots, s_n$. Consider (Update-1) in scheme (ADMM$_{\rm cf}^{\mu, \alpha, \rho}$) at iteration $k+1$. {Then $(\rho + \mu) I + 4M \succ 0$ implies $ \rho + \mu + 4 s_i > 0$, $\forall i = 1,\hdots,n$}, and the following hold: 

	\noindent (i) The solution $w_{k+1} \triangleq (x^+_{k+1}; x^-_{k+1}; \xi_{k+1})$ can be obtained as a solution to a tractable convex program.

	\noindent (ii) The solution $w_{k+1}$ is available in closed form. In particular, let $h \triangleq (d;-d;-\gamma e) + (1- \rho \alpha)\lambda_k - \rho y_k - \mu w_k$, and let $V$ be an orthogonal matrix such that $V^TMV = \mbox{diag}(s_1,\hdots,s_n) \triangleq S$, and let
\begin{align}
	 G&  \triangleq \pmat{ \frac{1}{2} I_n & \frac{\sqrt{2}}{2} I_n &
		\frac{1}{2}I_n \\  \frac{1}{2}I_n & -\frac{\sqrt{2}}{2} I_n &
			\frac{1}{2}I_n \\ -\frac{\sqrt{2}}{2} I_n &&
			\frac{\sqrt{2}}{2} I_n} \pmat{I_n&&\\&V&\\&&I_n}.\label{G} 
\end{align}
Also let $q \triangleq G^T h \triangleq (q_1; q_2; q_3)$, $z \triangleq \pmat{z_1 ; z_2 ; z_3}$, $q_1, q_2, q_3, z_1, z_2, z_3 \in
	\Real^n, $ and 
	\begin{align}
	\notag
	 z_1 & \triangleq \begin{cases} 
			\frac{-(\|q_1\|+\|q_3\|)q_1}{2 (\rho + \mu) \|q_1\|}, & \|q_1\| > 0 \\
			\frac{\|q_3\|}{2 (\rho + \mu) }u,\| u \|=1, & \|q_1\| =  0
			\end{cases}, \ 
	z_3  \triangleq \begin{cases} 
			\frac{-(\|q_1\|+\|q_3\|)q_3}{2 (\rho + \mu) \|q_3\|}, & \|q_3\| > 0 \\
			\frac{\|q_1\|}{2 (\rho + \mu) }v,\|v\|=1, & \|q_3\| =  0
			\end{cases} \\
	\notag	 
	(z_2)_i & \triangleq -(q_2)_i/( \rho + \mu + 4s_i), \forall i = 1,\hdots,n.
	\end{align}
 Then
	$w_{k+1} \triangleq Gz$ is the solution to (Update-1).
	\end{proposition}
\begin{proof}
\noindent {\bf (i).} The first subproblem in (ADMM$_{\rm cf}$) is
	equivalent to the following:
	\begin{align}  
	\label{1-update} & \min_{w \in Z_1}  \ \tscrL_{\rho, \alpha} (w,y_k,\uvs{\lambda}_k) + \frac{\mu}{2} \| w - w_k \|^2  \\
\notag
	 \equiv  & \min_{(x^+ + x^-)^T \xi = 0}  \left\{ f_Q(x^+ - x^-) + \gamma 
	\sum_{i=1}^n (1-\xi_i) + (1-\rho \alpha) \lambda_k^T w + \frac{\rho}{2} 
	\| w - y_k \|^2 \right. \\
\notag
	& \left. + \frac{\mu}{2} \| w - w_k \|^2 \right\} \\ 
\label{nonconQCQP}
  \equiv &  \min_{w^T \tilde{Q}  w = 0} \left\{ w^T H w + h^T w \right\},
	\end{align}
	where $ H  \triangleq \pmat{M + \frac{\rho + \mu}{2}I & -M & \\ -M & M + \frac{\rho + \mu}{2}I & \\ &&\frac{\rho + \mu}{2}I}, \tilde{Q}  \triangleq \pmat{&&I\\&&I\\I&I&}$. \diff{In fact, $H, \tilde{Q}$ can be simultaneously orthogonally diagonalized by using $G$ defined in \eqref{G} (See Lemma~\ref{lm: SimulDiag} in Appendix for the linear algebra). Therefore, by leveraging the hidden convexity (See discussion in Section~\ref{subsec: hiddenconvex} in Appendix), a global solution to this nonconvex QCQP \eqref{nonconQCQP} can be obtained by solving a tractable convex program (In \cite{Ben-Tal1996}, it is described how a polynomial time interior point method can be applied to solve this convex program).}

 \noindent{\bf (ii).} \diff{By substituting} $z = G^Tw, z \triangleq (z_1 ; z_2 ; z_3), z_1, z_2, z_3 \in \Real^n$,  \eqref{1-update} is equivalent to a simple QCQP,
	\begin{align}\label{simpleQCQP}
	\notag
	\min & \quad \frac{\rho + \mu}{2}\|z_1\|^2 +
	\sum_{i=1}^n\left(\frac{\rho + \mu}{2}+2s_i\right)(z_2)_i^2 + \frac{\rho + \mu}{2}\|z_3\|^2
	+ q^Tz \\
	\st & \quad \|z_1\|_2 = \|z_3\|_2.
	\end{align}
\diff{Again, this is a result by leveraging Lemma~\ref{lm: SimulDiag}.} To obtain an optimal solution of
		\eqref{simpleQCQP}, we require that {the objective value is
			bounded below}. {By completing squares}, a sufficient condition for boundedness of \eqref{simpleQCQP} is 
$\frac{\rho + \mu}{2} + 2s_i > 0, \forall i = 1,\hdots,n$
{because} $z_2$ is unconstrained. This is implied by the condition $(\rho + \mu) I_n + 4M \succ 0$.
	 The result of (ii) follows by noting that \textbf{all} optimal solutions $(z_1^*; z_2^*; z_3^*)$ of \eqref{simpleQCQP} can be characterized as follows:
	\begin{align}\label{closedform0}
	\notag
	z_1^* & = \begin{cases} 
			\frac{-(\|q_1\|+\|q_3\|)q_1}{2 (\rho + \mu) \|q_1\|}, & \|q_1\| > 0 \\
			\frac{\|q_3\|}{2 (\rho+\mu) }u,\| u \|=1, & \|q_1\| =  0
			\end{cases}, \ 
	z_3^*  = \begin{cases} 
			\frac{-(\|q_1\|+\|q_3\|)q_3}{2 (\rho + \mu) \|q_3\|}, & \|q_3\| > 0 \\
			\frac{\|q_1\|}{2(\rho + \mu)}v,\|v\|=1, & \|q_3\| =  0
			\end{cases} ,\\  
 (z_2^*)_i & = -(q_2)_i/(\rho + \mu + 4s_i), \quad \mbox{ for } i = 1,\hdots,n. 
	\end{align}
Next we show that this is true. Note that $(z_2^*)_i = -(q_2)_i/(\rho + \mu + 4s_i), \forall i$, because $z_2$ is unconstrained. Since the problem is separable with respect to $z_2$, it may be removed, leading to the problem of 
\begin{align*}
\min_{z_1, z_3} \ \frac{\rho + \mu}{2}\|z_1\|^2 + \frac{\rho + \mu}{2}\|z_3\|^2 + q_1^T z_1 + q_3^T z_3 \quad \st \ \|z_1\|^2 - \| z_3 \|^2  = 0.
\end{align*}
Since $z_1$ and $z_3$ have the same magnitude, let $z_1 \triangleq rd_1$ and $z_3 \triangleq rd_3$, where $\|d_1\| = \|d_3\| = 1$. Then the constraint may be removed and the problem is further simplified as
\begin{align*}
& \min_{r, d_1, d_3} \ (\rho + \mu) r^2 + r q_1^T d_1 + r q_3^T d_3 \quad \st \ r \geq 0, \| d_1\| = 1, \| d_3 \| = 1.
\end{align*}
It follows that $r^* = \argmin_{r \geq 0} \{ (\rho + \mu) r^2 - (\| q_1 \| + \| q_3 \|)r \} = (\|q_1\| + \|q_3\|)/(2 (\rho +\mu)).$ This leads to concluding that if $\|q_1\| >0, \|q_3\| > 0$, $z_1^* = -(\|q_1\| + \|q_3\|)q_1/(2 (\rho + \mu) \|q_1\|), z_3^* = -(\|q_1\| + \|q_3\|)q_3/(2 (\rho + \mu) \|q_3\|)$. If $\|q_1\|=0, \|q_3\|>0$, then $\|z_1^*\| =  \|q_3\|/(2(\rho + \mu))$ and $z_3^* = -q_3/(2(\rho + \mu))$ and $z_1^*$ can take any direction. If $\|q_3\|=0, \|q_1\|>0$, then $\|z_3^*\| =  \|q_1\|/(2(\rho + \mu))$ and $z_1^* = -q_1/(2(\rho + \mu))$ and $z_3^*$ can take any direction.
If $\|q_1\| = \|q_3\| = 0$, then $z_1^* = z_3^* = 0.$ \qed
	\end{proof}

\begin{remark}
Note that in order to compute the closed form solution, we do need eigendecomposition $M = V S V^{-1}$. However, this only needs to be done once instead of in every iteration.
\end{remark}

\uvs{Next, (Update-2) is shown to be tractable and its solution may be available in closed-form.}
	\begin{proposition}[{\bf Tractability of Update-2}] \label{tractableupdate-2} \rm  
	{Consider (Update-2) at iteration $k+1$ in (ADMM$_{\rm cf}^{\mu, \alpha,\rho}$). Then the following hold:}
 (i) \vvs{(Update-2)} is a convex program and can be computed tractably.
 (ii) If $g(x) \equiv 0$ and the constraints $Ax \geq b$ are absent, \vvs{(Update-2)} reduces to   
		{
		\begin{align} 
& y_{k+1} = \pmat{y_{k+1,1} \\ y_{k+1,2} \\ y_{k+1,3} }, \quad		
\left.
\begin{aligned}
			(y_{k+1,1})_i  &:=  \max\{(x_{k+1}^+)_i+ (1-\rho \alpha) (\lambda_{k,1})_i/\rho ,0\} \\
			(y_{k+1,2})_i  &:=  \max\{(x_{k+1}^-)_i + (1-\rho \alpha) (\lambda_{k,2})_i/\rho, 0\} \\
			(y_{k+1,3})_i  &:=  \Pi_{[0,1]} ( (\xi_{k+1})_i + (1-\rho \alpha) ( \lambda_{k,3} )_i / \rho )
			\end{aligned} \right\} \forall i,
			\label{Cart-update}
		\end{align} 
where ${\lambda_k} \triangleq (\lambda_{k,1} ; \lambda_{k,2};  \lambda_{k,3} )$, $y_{k+1,1}, y_{k+1,2}, y_{k+1,3}, {\lambda_{k,1}}, {\lambda_{k,2}}, {\lambda_{k,3}} \in \bR^n$, and $\Pi_Z(z)$ denotes the projection of $z$ onto set $Z$.
}
	\end{proposition}
	\begin{proof}
	\noindent {\bf (i). } (Update-2) can \vvs{be cast as a linearly constrained convex smooth program}:
$\min_{y \in Z_2} \left\{ g(y^+ - y^-) - (1-\rho \alpha) \lambda_k^Ty + \frac{\rho}{2} 
	\| (x^+_{k+1} ; x^-_{k+1}; \xi_{k+1}) - y \|^2 \right\}$, which may be tractably resolved \cite{ben2001lectures}. \\
\noindent {\bf (ii).} When $g \equiv 0$ and the constraints $Ax \geq b$ are absent, then (Update-2) can be viewed as a projection of $(x^+_{k+1} ; x^-_{k+1}; \xi_{k+1}) + (1-\rho \alpha)\lambda_k/ \rho$ onto a Cartesian set: 
	$
	\hat Z_2 \triangleq \left\{ (y_1; y_2; y_3) \mid
	y_1,y_2,y_3 \in \Real^n,
	y_1 \geq 0, y_2  \geq 0,
	0 \leq  (y_3)_i  \leq 1, \forall i
	\right\}.
	$
	 Consequently, the projection onto this set reduces to
	the update given by \eqref{Cart-update}. \qed
	\end{proof}
	
\subsection{A basic ADMM framework with tractable subproblems} \label{Vanilla ADMM}
In Algorithm~\ref{pb-admm}, a perturbation parameter and a proximal term are
introduced. In fact, as we see
in Section~\ref{sec: perturbADMM}, an explicit bound can be derived for the multiplier sequence
generated by Algorithm~\ref{pb-admm}. Therefore, we may show
subsequential convergence and estimate a norm of the limit point. In this subsection, we present the vanilla ADMM framework (denoted by ADMM$_{\rm cf}$ and defined in Algorithm~\ref{admm})
applied to the tractable decomposition~\eqref{tractable decomposition}. In (ADMM$_{\rm cf}$), the augmented Lagrangian function ${\cal L}_{\rho}$ is defined as follows.
	\begin{algorithm}
	\scriptsize
	\caption{ADMM$_{\rm cf}$}
	\label{admm}
	\begin{enumerate}
	\item[(0)] Given \vvs{$y_0,{\lambda}_0$, $\epsilon > 0$, $k :=0$; Choose $\rho_0,$ s.t. $\rho_0 I_n + 4M \succ 0$}; 
	\item[(1)] Let $x_{k+1}^+ , x_{k+1}^-,\xi_{k+1}, y_{k+1}, {\lambda}_{k+1}$ be given by the following: 
	\begin{align} \tag{Update-1}
	(x_{k+1}^+; x_{k+1}^-; \xi_{k+1})  \in & \  {\displaystyle \mbox{arg}\hspace{-0.08in}\min_{(x^{\pm} ,\xi)}} \quad \Lscr_{\rho_k}(x^+ , x^-,\xi,y_k,{\lambda_k}), \\
	\tag{Update-2}
	y_{k+1}   := & \  \mbox{arg}\hspace{-0.02in}\min_{y} \quad \Lscr_{\rho_k}(x_{k+1}^+ , x_{k+1}^-,\xi_{k+1},y,{\lambda_k}),\\
	\tag{Update-3}
	{\lambda}_{k+1}  := & \ 
	{\lambda}_k + {\rho_k} \left( (x_{k+1}^+;x_{k+1}^-;\xi_{k+1}) - y_{k+1} \right).
	\end{align}
	\item[(2)] Update $\rho_k$ and let $\rho_{k+1} \leftarrow \rho_k$;
	\item [(3)] If $\max\{ \| (x_{k+1}^+;x_{k+1}^-;\xi_{k+1}) - y_{k+1} \|,\rho_k\|y_{k+1}-y_k\| \} < \epsilon$, STOP;
	else $k:=k+1$, return to (1).
	\end{enumerate}
	\end{algorithm}
\begin{align*}
\scrL_{\rho}(x^+,x^-,\xi,y,\lambda) & \triangleq f_Q(x^+ - x^-) + \gamma \sum_{i=1}^n(1-\xi_i) + { g(y^+ - y^-)} \\
& + \lambda^T(w - y) + \frac{\rho}{2} \| w - y \|^2 + \bk1_{Z_1}((x^+;x^-;\xi)) + \bk1_{Z_2}(y).
\end{align*}
We will specify in Section~\ref{sec:num} the update rule for $\rho_k$ in Step 2. Note that if we let $\mu=0$, $\alpha = 0$ and replace $\rho$  by $\rho_k$, then ADMM$^{\mu,\alpha,\rho}_{\rm cf}$ reduces to (ADMM$_{\rm cf}$). The similarity
between these two algorithms allows for (ADMM$_{\rm cf}$) to maintain the property of
tractability of the subproblems (A special case of
proposition~\ref{closedformpro} and \ref{tractableupdate-2} when $\alpha = 0$,
$\mu = 0$). However, convergence analysis of (ADMM$_{\rm cf}$) is by no means
straightforward. Since \eqref{tractable decomposition} is in the form of \eqref{formofinterest}, we have discussed in Section~\ref{sec:intro} that no existing convergence theory for ADMM schemes in nonconvex regimes is applicable.
It turns out that even to show boundedness of the multiplier sequence is challenging. On the other hand, if we assume boundedness of a subsequence of the multiplier together with other assumptions, subsequential convergence of the algorithm may be obtained. Convergence properties can be further enhanced given the K{\L}
property. Details of these discussions are kept in \diff{Section~\ref{subsec: Anal.ADMMcf} in the Appendix} while an investigation of the numerical behavior of this algorithm is presented in
Section~\ref{sec:num}.

\subsection{A standard ADMM framework on an alternative formulation} \label{subsec: alterform}
In section~\ref{sec:31}, we consider reformulation \eqref{tractable decomposition} of \eqref{mpcc}, which allows for efficient resolution of the subproblem; Note that an alternative formulation of \eqref{mpcc} exists as specified next.
\begin{align}\label{intractable decomp}
\min \ \bk1_Z(w) + f(y^+ - y^-) + \gamma e^T(e - \zeta)  \quad \st \ w - y = 0 ,
\end{align}
where $w = (x^+; x^-; \xi), y = (y^+;y^-;\zeta), Z \triangleq Z_1 \cap Z_2 $,
$Z_1$ and $Z_2$ are defined as in \eqref{def-Z1Z2}. Note that in
\eqref{intractable decomp}, the $y$ block is unconstrained and has a smooth objective function. Such a reformulation of optimization over complementarity constraints is considered in \cite{DBLP:journals/corr/WangYZ15} and an ADMM scheme can be applied (See Algorithm~\ref{Alg: intractable} below). We referred to this framework as a standard ADMM framework (or (ADMM$_0$)), since this type of ADMM scheme is favored and most studied in literature due to clear convergence guarantee.
    \begin{algorithm}
	\caption{A standard ADMM framework: ADMM$_0$}
	\label{Alg: intractable}
	\begin{enumerate}
	\item[(0)] Given $y_0,\lambda_0, \rho_0 > 0, \epsilon > 0$, set $k :=0$.
	\item[(1)] Let $w_{k+1}, y_{k+1}, \lambda_{k+1}$ be given by the following:
	\begin{align} \tag{Update-1}
	w_{k+1} & \in  {\displaystyle
		\mbox{arg}\hspace{-0.03in}\min_{w \in Z} \quad \| w - y_k + \lambda_k/\rho_k \|^2 }, \\
	\tag{Update-2}
	y_{k+1}  &:=  { \displaystyle \mbox{arg}\hspace{-0.03in}\min_{y} \quad f(y^+ - y^-) + \gamma e^T( e - \zeta ) + \frac{\rho_k}{2} \| y - w_{k+1} - \lambda_k/\rho_k \|^2 },\\
	\tag{Update-3}
	\lambda_{k+1}&  :=  \lambda_k + \rho_k ( w_{k+1} - y_{k+1} ).
	\end{align}
	\item [(2)] Update $\rho_k$ and let $\rho_{k+1} \leftarrow \rho_k$;
	\item [(3)] If $\max( \| w_{k+1} - y_{k+1} \|,\rho_k\|y_{k+1}-y_k\|) < \epsilon$, stop;
	else $k:=k+1$ and return to (1).
	\end{enumerate}
	\end{algorithm}
As indicated in \cite[Corollary 3]{DBLP:journals/corr/WangYZ15}, if $\rho_k \equiv \rho$ is large enough and the augmented Lagrangian function has the K{\L} property, (ADMM$_0$) generates a sequence convergent to a stationary point. However, such a framework is potentially slow because (Update-1) requires globally resolving an  MPCC and may render the scheme impractical. We will further explain with numerical experiments in Section~\ref{sec:num}.

\section{ Convergence Analysis } \label{sec: perturbADMM}
In the prior section, we consider the formulation \eqref{tractable decomposition} to resolve \eqref{mpcc} and present a perturbed proximal ADMM framework (ADMM$_{\rm cf}^{\mu,\alpha,\rho}$) reliant on tractable updates at each iteration.
In this section, we analyze the convergence property of this framework.
Specifically, we show that under mild assumptions, the sequence $\{ \lambda_k \}$ is bounded and a subsequence of $\{ ( w_k, y_k, \lambda_k) \}_{k \geq 1}$ converges to a perturbed KKT point of \eqref{tractable decomposition}. The main results are Theorem~\ref{thm: Subconv}, Corollary~\ref{Corr: Parameter choosing} and Theorem~\ref{KKTwithError}.
First we present some definitions used in this section. We refer interested readers to \cite{10.2307/40801236} and \cite{rockafellar2009variational} for more details.
\begin{definition}[(Limiting) subdifferential and critical point] \label{def: limsub&crit-pt}
Let $F: \bR^n \rightarrow \bR \cup \{ + \infty \}$ be a proper lower semicontinuous function and let $\bar{\partial} F (x)$ denote the Fr{\'e}chet subdifferential of $F$ at $x$, i.e.,
\begin{align*}
\bar{\partial} F (x) \triangleq \left\{ d : \liminf\limits_{z \neq x, z \rightarrow x} \frac{1}{\| z - x \|} [ F(z) - F(x) - d^T(z-x) ] \geq 0 \right\},
\end{align*}
for $x \in \dom F = \{ x : F(x) < +\infty \}$ and $\bar{\partial} F(x) = \emptyset$ if $x \notin \dom F$. Then,
\be
\item[(i).] The (limiting) subdifferential of $F$ at $x \in \dom F$, is defined as follows.
\begin{align*}
& \partial F (x) \triangleq \\
& \{ d \in \bR^n : \exists \{ x_k \}_{k \ge 1}, \mbox{ s.t. } x_k \rightarrow x, \ F(x_k) \rightarrow F(x), \ d_k \rightarrow d, \ d_k \in \bar{\partial} F(x_k) \}.
\end{align*}
\item[(ii).] $x$ is a critical point of $F$ if and only if $0 \in \partial F(x)$.
\ee
\end{definition}
\begin{definition}[(Limiting) normal cone]\label{def: limcone}
Suppose $Z$ is a nonempty closed subset of $\bR^n$, and $\bar N_Z(x)$ denotes the Fr{\'e}chet normal cone to $Z$ at $x$, so
\begin{align*}
\bar N_Z(x) = \{ v \in \bR^n : v^T(z - x) \leq o(x - z), \forall z \in Z \}, 
\end{align*}
if $x \in Z$ and $\bar N_Z(x) = \emptyset$ if $x \notin Z$. Then the (limiting) normal cone to $Z$ at $x \in Z$, denoted as $N_Z(x)$, is defined as follows.
\begin{align*}
N_Z(x) \triangleq \{ v \in \bR^n: \exists \{ x_k \}_{k \ge 1}, \mbox{ s.t. } x_k \rightarrow x,\ x_k \in Z, \ v_k \rightarrow v, \ v_k \in \bar N_Z(x_k) \}.
\end{align*}
\end{definition}
These concepts have the following properties:
\be
\item[(i).] (Closedness of $\partial F$) If $d_k \rightarrow d$, $x_k \rightarrow x$ and $d_k \in \partial F(x_k)$, $F(x_k) \rightarrow F(x)$, then $d \in \partial F(x)$.
\item[(ii).] Let $F = F_0 + F_1$. If $F_0$ is finite at $x$ and $F_1$ is smooth in a neighborhood of $x$, then $\partial F = \partial F_0 + \nabla F_1$.
\item[(iii).] If $x^* \ \in \  \mbox{argmin}\ F(x)$, then $x^*$ is a critical point of $F$. 
\item[(iv).] $\partial \bk1_Z(z) = N_Z(z), \forall z \in Z$.
\ee

To simplify the notation, we rewrite \eqref{tractable decomposition} as the
following structured program:
\begin{align} 
\min_{w \in Z_1, y \in Z_2}  \ h(w) + p(y)
\quad \st \ w - y  = 0,
\end{align}
where $Z_1, Z_2$ are defined by \eqref{def-Z1Z2}, $w \triangleq (x^+; x^-; \xi), y \triangleq (y^+;y^-;\zeta), h(w) \triangleq f_Q(x^+-x^-) + \gamma \sum_{i=1}^n (1-\xi_i)$, $f_Q(x^+ - x^-) \triangleq (x^+ - x^-)M(x^+ - x^-) + d^T(x^+ - x^-)$, $p(y) \triangleq g(y^+ - y^-) $. In addition, the perturbed augmented Lagrangian function is rewritten as follows.
$$
\tscrL_{\rho,\alpha}(w,y,\lambda) \triangleq h(w) + p(y) + (1-\rho \alpha)\lambda^T(w - y  - \alpha \lambda) + \frac{\rho}{2} \| w - y \|^2.
$$ 
Let $r_k \triangleq w_k - y_k$ and $\Delta \lambda_{k+1} \triangleq \lambda_{k+1} - \lambda_k$ for all $k \geq 0.$ We define a Lyapunov function $P_\tau^{k}$ for any $\tau > 0$ and $k \geq 1$.
\begin{align}\label{Lyapunov}
P_\tau^{k} \triangleq \tscrL_{\rho,\alpha}(w_k, y_k, \lambda_k) + \frac{(1-\rho\alpha)\alpha}{2} \| \lambda_k \|^2 + \tau \left( \frac{1-\rho\alpha}{2\rho} \right) \| \lambda_k - \lambda_{k-1} \|^2.
\end{align}
We intend to show that the sequence $\{ P^k_{\tau} \}_{k \geq 1}$ is nonincreasing and the following two lemmas are needed.
\begin{lemma} \rm \uvs{Consider the sequence $\{w_k,y_k,\lambda_k\}$ generated by (ADMM$_{cf}^{\mu,\alpha, \rho}$).} \\ 
Then the following holds for any $\nu > 0$, and any $k \geq 1$,
\begin{align}
\label{Subconv-lemma1}
\notag
& \frac{1 - \rho \alpha}{2 \rho} \left( \| \lambda_{k+1} - \lambda_k \|^2 -  \|\lambda_k - \lambda_{k-1} \|^2\right) \\
 & \leq - \left( \alpha - \frac{\nu}{2} \right) \| \lambda_{k+1} - \lambda_k \|^2 + \frac{1}{2\nu}\| w_{k+1} - w_k \|^2.
\end{align}
\end{lemma}
\begin{proof}
Let $G_{k+1} \triangleq \nabla_y p(y_{k+1})$. By (Update-2), for all $y \in Z_2$ and \uvs{$k \geq 0$},
\begin{align} \label{Subconv-Ineq1}
\begin{aligned}
0 & \geq \left( G_{k+1} - (1- \rho \alpha) \lambda_k - \rho r_{k+1} \right)^T ( y_{k+1} - y )  \\
& = \left( G_{k+1} - \lambda_{k+1} \right)^T (y_{k+1} - y ).
\end{aligned}
\end{align}
Consequently, we have that $\forall k \geq 1$,
\begin{align}\label{Subconv-Ineq2}
 \left( G_k -  \lambda_k \right)^T (y_k - y ) \leq 0, \qquad \forall y \in Z_2.
\end{align}
By choosing $y = y_k$ in \eqref{Subconv-Ineq1}, $y = y_{k+1}$ in \eqref{Subconv-Ineq2}, then adding \eqref{Subconv-Ineq1} and \eqref{Subconv-Ineq2}, we have that for any $k \ge 1$,
\begin{align}
\notag
& (G_{k+1} - G_k - \lambda_{k+1} +  \lambda_k )^T(y_{k+1} - y_k) \leq 0, \\
\notag
\implies & \left( G_{k+1} - G_k \right)^T(y_{k+1} - y_k) - (\lambda_{k+1} - \lambda_k)^T(y_{k+1} - y_k) \leq 0,\\
\label{Subconv-Ineq2.5}
\implies  & -(\lambda_{k+1} - \lambda_k)^T(y_{k+1} - y_k) \leq 0,
\end{align}
where the last step follows from convexity of $p(y)$. Recall
$\Delta \lambda_k \triangleq \lambda_k - \lambda_{k-1},
\forall k \geq 1$. Then by adding $\Delta
\lambda_{k+1}^T (w_{k+1} -w_k)$ on both sides,
\eqref{Subconv-Ineq2.5} can be rewritten as follows for $\forall k \geq 1$,
\begin{align}
\notag
& \Delta \lambda_{k+1}^T (w_{k+1} - w_k) \\
\notag
& \overset{\eqref{Subconv-Ineq2.5}}{\geq} \Delta \lambda_{k+1}^T (w_{k+1} - y_{k+1} - w_k + y_k ) \\
\notag 
& = \Delta \lambda_{k+1}^T( r_{k+1} - r_k ) \\
\notag
& =  \Delta \lambda_{k+1}^T ( r_{k+1} - \alpha \lambda_k - r_k + \alpha \lambda_{k-1} ) + \Delta \lambda_{k+1}^T(\alpha \lambda_k - \alpha \lambda_{k-1}  ) \\
\notag
 & = \Delta \lambda_{k+1}^T \left( \frac{\Delta \lambda_{k+1}}{\rho} - \frac{\Delta \lambda_k}{\rho} \right) + \alpha \Delta \lambda_{k+1}^T \Delta \lambda_k \\
 \label{Subconv-Ineq2.6}
& = \frac{1-\rho \alpha}{\rho} \Delta \lambda_{k+1}^T(\Delta \lambda_{k+1} - \Delta \lambda_k) + \alpha \| \Delta \lambda_{k+1} \|^2.
\end{align}
Note that $\Delta \lambda_{k+1}^T(\Delta \lambda_{k+1} - \Delta \lambda_k) = \frac{1}{2}(\| \Delta \lambda_{k+1} \|^2 - \| \Delta \lambda_k \|^2 + \| \Delta \lambda_{k+1} - \Delta \lambda_k \|^2 )$. Let $\Delta w_{k+1} \triangleq w_{k+1} - w_k$. Then by \eqref{Subconv-Ineq2.6},
\begin{align*}
& \frac{1-\rho \alpha}{2\rho} \cdot (\| \Delta \lambda_{k+1} \|^2 - \| \Delta \lambda_k \|^2 + \| \Delta \lambda_{k+1} - \Delta \lambda_k \|^2 ) + \alpha \| \Delta \lambda_{k+1} \|^2 \\
& \le \Delta \lambda_{k+1}^T \vvs{\Delta w_{k+1}} \\
\implies & \frac{1-\rho \alpha}{2\rho} \cdot (\| \Delta \lambda_{k+1} \|^2 - \| \Delta \lambda_k \|^2) \\
& \leq -\alpha \| \Delta \lambda_{k+1} \|^2 + \Delta \lambda_{k+1}^T \Delta w_{k+1}\\
& \leq -\alpha \| \Delta \lambda_{k+1} \|^2
 + \frac{\nu \| \Delta \lambda_{k+1} \|^2 }{2}+ \frac{\| \vvs{\Delta w_{k+1}} \|^2}{2 \nu}  \\
& = \frac{\| \vvs{\Delta w_{k+1}}\|^2}{2 \nu}-\left( \alpha - \frac{\nu}{2} \right) \| \Delta \lambda_{k+1} \|^2 .
\end{align*}
Then the proof is complete. \qed
\end{proof}
\begin{lemma}
Consider $\{w_k,y_k,\lambda_k\}$ generated by ({\rm ADMM}$_{\rm cf}^{\mu,\alpha,\rho}$). Then
\begin{align}
\label{Subconv-lemma2}
\notag
& \left( \tscrL_{\rho,\alpha}(w_{k+1}, y_{k+1}, \lambda_{k+1}) + \frac{(1-\rho\alpha)\alpha}{2} \| \lambda_{k+1} \|^2 \right) \\
\notag
& - \left( \tscrL_{\rho,\alpha}(w_k, y_k, \lambda_k) + \frac{(1-\rho\alpha)\alpha}{2} \| \lambda_k \|^2 \right) \\
& \leq  - \frac{\mu}{2} \| w_{k+1} - w_k \|^2 - \frac{\rho}{2} \| y_{k+1} - y_k \|^2 + \frac{(1-\rho\alpha)(2-\rho\alpha)}{2\rho} \| \Delta \lambda_{k+1} \|^2.
\end{align}
\end{lemma}
\begin{proof}
From (Update-1),
\begin{align}
\notag
& \tscrL_{\rho,\alpha}(w_{k+1}, y_k, \lambda_k) + \frac{\mu}{2}\| w_{k+1} - w_k \|^2 - \tscrL_{\rho,\alpha}(w_k, y_k, \lambda_k) \leq 0 \\
\label{Subconv-Ineq3}
\implies \ & \tscrL_{\rho,\alpha}(w_{k+1}, y_k, \lambda_k) - \tscrL_{\rho,\alpha}(w_k, y_k, \lambda_k) \leq -\frac{\mu}{2} \| w_{k+1} - w_k \|^2.
\end{align}
Also, by the optimality condition of (Update-2), if $\tilde G_{k+1} \triangleq \nabla_y \tscrL_{\rho,\alpha}(w_{k+1},y_{k+1},\lambda_k)$, then $\tilde G_{k+1}^T (y - y_{k+1}) \geq 0, \forall y \in Z_2$. Using this fact and the strong convexity of $\tscrL_{\rho,\alpha}$ in terms of $y$ with constant $\rho$, 
\begin{align}
\label{Subconv-Ineq4}
\notag
& \tscrL_{\rho,\alpha}(w_{k+1}, y_{k+1}, \lambda_k) - \tscrL_{\rho,\alpha}(w_{k+1}, y_k, \lambda_k) \\
& \leq - \tilde G_{k+1}^T(y_k - y_{k+1} ) - \frac{\rho}{2} \| y_{k+1} - y_k \|^2
\leq  - \frac{\rho}{2} \| y_{k+1} - y_k \|^2.
\end{align}
The fact that $ \Delta \lambda_{k+1} = \rho r_{k+1} - \rho \alpha \lambda_k$ and $\lambda_{k+1}^T \Delta \lambda_{k+1} = \frac{1}{2} (\| \lambda_{k+1} \|^2 - \| \lambda_k \|^2 + \| \Delta \lambda_{k+1} \|^2) $ imply:
\begin{align}
\notag
& \quad \tscrL_{\rho,\alpha}(w_{k+1}, y_{k+1}, \lambda_{k+1}) - \tscrL_{\rho,\alpha}(w_{k+1}, y_{k+1}, \lambda_k) \\
\notag
& = (1-\rho\alpha) \lambda_{k+1}^T(r_{k+1} - \alpha \lambda_{k+1}) - (1-\rho\alpha) \lambda_k^T(r_{k+1} - \alpha \lambda_k) \\
\notag
& = (1-\rho\alpha) \lambda_{k+1}^T(r_{k+1} - \alpha \lambda_k - \alpha \Delta \lambda_{k+1} ) - (1-\rho\alpha) \lambda_k^T(r_{k+1} - \alpha \lambda_k) \\
\notag
& = (1-\rho\alpha)(\lambda_{k+1} - \lambda_k )^T(r_{k+1} - \alpha \lambda_k) - (1-\rho\alpha)\lambda_{k+1}^T\alpha \Delta \lambda_{k+1} \\
\notag
& = \frac{1-\rho \alpha}{\rho} \| \Delta \lambda_{k+1} \|^2 - \frac{(1-\rho\alpha)\alpha}{2} (\| \lambda_{k+1} \|^2 - \| \lambda_k \|^2 + \| \Delta \lambda_{k+1} \|^2) \\
\label{Subcon-Ineq5}
& = \frac{(1-\rho\alpha)(2-\rho\alpha)}{2\rho} \| \Delta \lambda_{k+1} \|^2 - \frac{(1-\rho\alpha)\alpha}{2} ( \| \lambda_{k+1} \|^2 - \| \lambda_k \|^2 ).
\end{align}
Finally, \uvs{by  adding} \eqref{Subconv-Ineq3}, \eqref{Subconv-Ineq4} and \eqref{Subcon-Ineq5}, the following holds $\forall k \geq 0$, 
\begin{align*}
& \quad \tscrL_{\rho,\alpha}(w_{k+1}, y_{k+1}, \lambda_{k+1}) - \tscrL_{\rho,\alpha}(w_k, y_k, \lambda_k) \\
& = \tscrL_{\rho,\alpha}(w_{k+1}, y_{k+1}, \lambda_{k+1}) 
 - \tscrL_{\rho,\alpha}(w_{k+1}, y_{k+1}, \lambda_k) + \tscrL_{\rho,\alpha}(w_{k+1}, y_{k+1}, \lambda_k) \\
& - \tscrL_{\rho,\alpha}(w_{k+1}, y_k, \lambda_k) + \tscrL_{\rho,\alpha}(w_{k+1}, y_k, \lambda_k) - \tscrL_{\rho,\alpha}(w_k, y_k, \lambda_k) \\
& \leq -\frac{\mu}{2}\| w_{k+1}-w_k \|^2 - \frac{\rho}{2}\| y_{k+1} - y_k \|^2 \\
&  +\frac{(1-\rho\alpha)(2-\rho\alpha)}{2\rho} \| \Delta \lambda_{k+1} \|^2  - \frac{(1-\rho\alpha)\alpha}{2} ( \| \lambda_{k+1} \|^2 - \| \lambda_k \|^2 ).
\end{align*}
Then the result follows. \qed
\end{proof}

\uvs{We now impose a requirement on $h(w)+p(y)+{\rho \over 2} \|w-y\|^2$ and define several constants to be used later.}
\begin{assumption} \label{lowerbound}
\rm $h(w) + p(y) + \frac{\rho}{2} \| w - y \|^2 \geq \lb $ for all $w \in Z_1, y \in Z_2$.
\end{assumption}
\begin{definition}\label{def-c}
Recall that $\alpha, \tau, \mu, \rho$ are nonnegative parameters from Algorithm~\ref{pb-admm} and the definition \eqref{Lyapunov}. Let $\nu$ and $R$ be nonnegative constants. Then
$c_1(\nu)  \triangleq \frac{\mu}{2} - \frac{\tau}{2 \nu} $, $c_2 \triangleq {\rho\over 2} $, $c_3(\nu) \triangleq \tau \left( \alpha - \frac{\nu}{2} \right) - \frac{(1-\rho\alpha)(2-\rho\alpha)}{2\rho} $, \\
$c_4(R)  \triangleq \frac{(1-\rho\alpha)[(R+1)\rho\alpha - 1]}{2 \rho R}$, $c_5(R) \triangleq \frac{1-\rho \alpha}{2\rho} \left[ \tau - (1-\rho\alpha)R \right].$
\end{definition}

\begin{assumption}\label{Ass: c}
$\exists \nu > 0$, $R > 0$ such that $c_1(\nu), c_3(\nu), c_4(R), c_5(R) > 0$.
\end{assumption}

In the next Lemma, we prove that $\{ P_\tau^k \}_{k \geq 1}$ is a nonincreasing sequence.

\begin{lemma}\label{lm: lowerbounded} \rm
Consider $\{w_k,y_k,\lambda_k\}$ generated by (ADMM$_{\rm cf}^{\mu,\alpha,\rho}$). Then,

\noindent (i). $P_\tau^{k+1} - P_\tau^k \leq - c_1(\nu) \| w_{k+1} - w_k \|^2 - c_2 \| y_{k+1}
- y_k \|^2 - c_3(\nu) \| \lambda_{k+1} - \lambda_k \|^2 $, $\forall k \geq 1$.\\
(ii). Suppose that Assumption~\ref{Ass: c} holds. Then $\{ P_\tau^k \}_{k \geq 1}$ is non-increasing. If Assumption~\ref{lowerbound} also holds, then $P_\tau^k$ is \uvs{bounded from below}. \\
(iii). If Assumption \ref{lowerbound} and \ref{Ass: c} hold, then
$\lim\limits_{k \rightarrow \infty} (w_{k+1} - w_k) = \lim\limits_{k
\rightarrow \infty} (y_{k+1} - y_k) = \lim\limits_{k \rightarrow \infty}
(\lambda_{k+1} - \lambda_k) = 0$. \end{lemma}
\begin{proof}
(i). Take $ \tau \times \eqref{Subconv-lemma1} + \eqref{Subconv-lemma2}$ and the result follows. \\
(ii). When $c_1(\nu), c_2, c_3(\nu) > 0$ for certain $\nu$, we conclude from (i) that $P_\tau^{k+1} \leq P_\tau^k$ for all $k \geq 1$. Further,
\begin{align}
\notag
& \quad (1-\rho \alpha)\lambda_k^T(r_k - \alpha \lambda_k) \\ 
\notag & = (1-\rho \alpha) \lambda_k^T(r_k - \alpha \lambda_{k-1} - \alpha \Delta \lambda_k ) \\
\notag
& = (1-\rho \alpha) \lambda_k^T \left[ \Delta \lambda_k / \rho - \alpha \Delta \lambda_k \right] \\
\notag
& =  \frac{(1-\rho \alpha)^2}{\rho} \lambda_k^T(\lambda_k - \lambda_{k-1}) \\
\label{Subconv_eq1}
& =  \frac{(1-\rho \alpha)^2}{2\rho} ( \| \lambda_k \|^2 - \| \lambda_{k-1} \|^2 + \| \lambda_k - \lambda_{k-1} \|^2 ) \\
\label{Subconv_ineq7}
& \geq [ (1-\rho \alpha)^2 / (2\rho) ]( \| \lambda_k \|^2 - \| \lambda_{k-1} \|^2 ), \qquad k \geq 1.
\end{align}
Then,
\begin{align}
\notag
P_\tau^k & = h(w_k) + p(y_k) + \frac{\rho}{2} \| w_k - y_k \|^2 \\
\notag
& + (1-\rho \alpha)\lambda_k^T(r_k - \alpha \lambda_k) + \frac{(1-\rho\alpha)\alpha}{2} \| \lambda_k \|^2 + \tau \left( \frac{1-\rho\alpha}{2\rho} \right) \| \Delta \lambda_k \|^2 \\
& \geq \bar L + (1-\rho \alpha)\lambda_k^T(r_k - \alpha \lambda_k) + \frac{(1-\rho\alpha)\alpha}{2} \| \lambda_k \|^2 + \tau \left( \frac{1-\rho\alpha}{2\rho} \right) \| \Delta \lambda_k \|^2 \label{bd-Pkk}\\
& \geq \lb + (1-\rho \alpha)\lambda_k^T(r_k - \alpha \lambda_k)  \overset{\eqref{Subconv_ineq7}}{\geq} \lb + \frac{(1-\rho \alpha)^2}{2\rho}( \| \lambda_k \|^2 - \| \lambda_{k-1} \|^2 ).\notag
\end{align}
Then, $\sum_{k=1}^K (P_\tau^k - \lb) \geq \frac{(1-\rho \alpha)^2}{2\rho} \sum_{k=1}^K ( \| \lambda_k \|^2 - \| \lambda_{k-1} \|^2 ) \geq -\frac{(1-\rho \alpha)^2}{2\rho} \| \lambda_0 \|^2$, $\forall K \geq 1$. Since $\{ P_\tau^k - \lb \}_{k \geq 1}$ is a non-increasing sequence and the above inequality holds, $\{ P_\tau^k - \lb \}_{k \geq 1}$ is nonnegative. Thus $\{ P_\tau^k \}_{k \geq 1}$ is bounded from below.
\noindent{(iii).} This may be concluded based on (i) and (ii). \qed 
\end{proof}
\begin{remark} \label{rmk: stopcrit.0}
By Lemma~\ref{lm: lowerbounded}(iii) and the stopping criterion, Algorithm ADMM$_{\rm cf}^{\mu, \alpha, \rho}$ may terminate in finite time.
\end{remark}
The following Lemma provides an inequality related to $\| \lambda_k \|$, which helps in showing boundedness of $\| \lambda_k \|$.
\begin{lemma}\label{lm: bound of multiplier} \rm 
Consider $\{w_k,y_k,\lambda_k\}$ generated by (ADMM$_{cf}^{\mu,\alpha,\rho}$). Suppose Assumption \ref{lowerbound} holds. Then 
$P_\tau^k \geq \lb + c_4(R) \| \lambda_k \|^2 + c_5(R) \| \lambda_k - \lambda_{k-1} \|^2$ for all $R > 0$ and $k \geq 1$. Therefore, if Assumption~\ref{Ass: c} holds, $\| \lambda_k \|^2 \leq \frac{1}{c_4(R)} (P_\tau^k - \lb)$ for all $k \geq 1$.
\end{lemma}
\begin{proof}
We may use the following result for any $R > 0$:
\begin{align}\label{Subconv-Ineq8} \notag
\| \lambda_{k-1} \|^2 & = \| \lambda_{k-1} - \lambda_k  + \lambda_k \|^2 \\
	& \leq (1+R) \| \lambda_{k-1} - \lambda_k \|^2 + ( 1 + 1/R ) \| \lambda_k \|^2.
\end{align}
From the definition of $P_\tau^k$, we have that:
\begin{align*}
& P_\tau^k \\
& \vvs{\overset{\eqref{bd-Pkk},\eqref{Subconv_eq1}}{\geq}} \lb + \frac{(1-\rho\alpha)\alpha}{2}\| \lambda_k \|^2  + \tau \left( \frac{1-\rho \alpha}{2\rho} \right) \| \lambda_k - \lambda_{k-1} \|^2\\ & + \frac{(1-\rho \alpha)^2}{2\rho} ( \| \lambda_k \|^2 - \| \lambda_{k-1} \|^2 + \| \lambda_k - \lambda_{k-1} \|^2 ) \\ 
& \overset{\eqref{Subconv-Ineq8}}{\geq} \lb + \frac{(1-\rho\alpha)\alpha\| \lambda_k \|^2}{2} + \tau \left( \frac{(1-\rho \alpha)\| \lambda_k - \lambda_{k-1} \|^2}{2\rho} \right) \\
& - \frac{(1-\rho \alpha)^2\left( \frac{\| \lambda_k \|^2}{R}  + R \| \lambda_k - \lambda_{k-1} \|^2 \right)}{2\rho} \\
& \geq \lb + \frac{(1-\rho\alpha)[(R+1)\rho\alpha - 1]}{2\rho R} \| \lambda_k \|^2 + \frac{1-\rho \alpha}{2\rho} \left[ \tau - (1-\rho\alpha)R \right] \| \lambda_k - \lambda_{k-1} \|^2.
\end{align*}
Then the result follows from definitions of $c_4(R)$ and $c_5(R)$.\qed
\end{proof}
Boundedness of $\| \lambda_k \|$ and subsequential convergence are proved in the next theorem.
\begin{theorem}\label{thm: Subconv} \rm Suppose $\{(w_k;y_k;\lambda_k)\}_{k \geq 0}$ is \yxIII{generated by (ADMM$_{cf}^{\mu,\alpha,\rho}$)}. Assume that the sequences $\{w_k\}$ and $\{y_k\}$ are bounded. Suppose Assumptions~\ref{lowerbound} and~\ref{Ass: c} hold. Then the sequence $\{ \lambda_k \}_{k\geq 1}$ is bounded and a subsequence of $\{(w_k;y_k;\lambda_k)\}$ converges to $(w^*;y^*;\lambda^*)$ such that
\begin{align}\label{Critical condition}
& 0 \in \partial (h + \bk1_{Z_1})(w^*) + \lambda^*, \quad 0 \in \partial (p + \bk1_{Z_2} )(y^*) - \lambda^*, \quad w^*- y^*  = \alpha \lambda^*.
\end{align}
\end{theorem}
\begin{proof}
Since $c_1(\nu) > 0, c_3(\nu) > 0$, then by Lemma~\ref{lm: lowerbounded}, $\{ P_\tau^k \}$ is a non-increasing sequence. Futhermore, since $c_4(R) > 0, c_5(R) > 0$, Lemma~\ref{lm: bound of multiplier} indicates that 
\begin{align}\label{Bound-Ineq1}
\| \lambda_k \|^2 \leq \frac{1}{c_4(R)} (P_\tau^k - \lb) \leq  \frac{1}{c_4(R)} (P_\tau^1 - \lb) < +\infty.
\end{align}  
Therefore, the sequence $\{ \lambda_k \}$ is bounded, implying that $\{ (w_k; y_k; \lambda_k) \}$ is bounded. Suppose $\{ (w_{n_k}; y_{n_k}; \lambda_{n_k}) \}$ denotes a convergent subsequence of $\{ (w_k; y_k; \lambda_k) \}$ such that {$ (w_{n_k}; y_{n_k}; \lambda_{n_k})$ $\to (w^*; y^*; \lambda^*)
$ as $k \to \infty.$}
Based on the optimality conditions of (Update-1), (Update-2) translated using property of critical points, and the multiplier update, the following hold:
\begin{align}
\label{optcon. of updates}
\begin{array}{l}
0 \in \partial (h + \bk1_{Z_1})(w_{n_k}) + \lambda_{n_k} + \rho(y_{n_k} - y_{n_k-1}) + \mu(w_{n_k} - w_{n_k - 1}),\\
0 \in \partial (p + \bk1_{Z_2} )(y_{n_k}) - \lambda_{n_k},  \quad
 w_{n_k} - y_{n_k} - \alpha\lambda_{n_k - 1}  = ( \lambda_{n_k} - \lambda_{n_k-1} )/\rho.
\end{array}
\end{align}
By Lemma~\ref{lm: lowerbounded}, $w_{n_k} - w_{n_k - 1} \rightarrow 0$, $ y_{n_k} - y_{n_k - 1} \rightarrow 0$, $ \lambda_{n_k} - \lambda_{n_k - 1} \rightarrow 0, k \rightarrow +\infty$, so we also have $\lambda_{n_k - 1} \rightarrow \lambda^*, k \rightarrow +\infty$. Therefore, by taking limits and the closedness of a subdifferential map, we may conclude the result. \qed
\end{proof}

\begin{remark}
(i). Boundedness of $\{ y_k \}$ and $\{ w_k \}$ is a mild assumption. First, boundedness of $\{y_k\}$ can be obtained from adding constraints such as $x^+ \leq ub^+$, $x^- \leq ub^-$ to $Z_2$. If $ub^+$ and $ub^-$ are large enough, $Z_2$ will still include the optimal solution. Second, since $r_k = w_k - y_k = \frac{1}{\rho}\lambda_k - \frac{1-\rho\alpha}{\rho} \lambda_{k-1}, \forall k \geq 1$ and $\{\lambda_k\}$ is bounded, $r_k$ is also a bounded sequence. Thus, boundedness of $\{ w_k \}$ is implied by boundedness of $\{ y_k \}$. \\
(ii). It can be shown that the conditions \eqref{Critical condition} are equivalent to KKT conditions with a feasibility error (See Theorem~\ref{KKTwithError} below). \\
(iii). {Denote $\scrH_{\tau}(w,y,\lambda)$ as
\begin{align*}
& \scrH_{\tau}(w,y,\lambda) \\
& \triangleq \tscrL_{\rho,\alpha}(w,y,\lambda) + \bk1_{Z_1}(w) + \bk1_{Z_2}(y) + \frac{(1-\rho\alpha)\alpha}{2} \| \lambda \|^2 +  \frac{\rho\| w - y - \alpha \lambda \|^2}{2(1-\rho \alpha)/\tau}.
\end{align*}
Then $\scrH_{\tau}(w_k,y_k,\lambda_k) = P_\tau^k, \forall k \geq 1, \tau > 0$. If the assumptions in Theorem~\ref{thm: Subconv} hold, and in addition, $\scrH_{\tau}(w,y,\lambda)$ satisfies the K{\L} property at $(w^*,y^*, \lambda^*)$, then $\{ (w_k,y_k,\lambda_k) \}$ converges to $(w^*,y^*, \lambda^*)$. The proof is similar to \cite[Theorem~3.1]{10.2307/40801236} and is omitted.}  \\
\diff{(iv). The K{\L} property assumption on $\scrH_{\tau}$ indeed holds when $p(y)$ is semialgebraic (See Definition~\ref{def: semiAb} in the Appendix). In fact, $\scrH_{\tau}$ is a sum of semialgebraic functions and is therefore semialgebraic. Then the result follows from the fact that a semialgebraic function satisfies the K{\L} property at every point in its domain \cite{10.2307/40801236}.} \\
\diff{(v).} Although in the context of this paper we focus on problem \eqref{tractable decomposition}, it should be noted that Theorem~\ref{thm: Subconv} may be generalized. Specifically, Theorem~\ref{thm: Subconv} may hold if we apply ADMM$_{\rm cf}^{\mu, \alpha, \rho}$ to tackle a more general class of problem: 
\[
\min f(x) + g(y)  \quad \st \ Ax + By = b, x \in X, y \in Y,
\]
where $g$ is smooth and convex, $Y$ is convex, but $f$ could be nonsmooth and nonconvex, $X$ could be nonconvex, $A, B, b$ are matrices and vectors with appropriate dimensions and do not need to be $I, -I, 0$. The analysis will basically remain the same.
\end{remark}
Note that \eqref{Critical condition} are not the precise conditions
for $(w^*; y^*; \lambda^*)$ to be a critical point of the Lagrangian
$\scrL(w,y,\lambda) \triangleq h(w) + \bk1_{Z_1}(w) + p(y) + \bk1_{Z_2}(y) +
\lambda^T(w - y)$, i.e. $0 \in \partial \Lscr(w^*, y^*, \lambda^*)$. There
exists an infeasibility error $\alpha \lambda^*$ and the following corollary
discusses how to choose the parameters such that this error can be made arbitrarily small. 
 \begin{corollary} \label{Corr: Parameter choosing} \rm Suppose that sequences $\{ w_k \}$ and $\{ y_k \}$ are bounded. In addition,
assume that $\exists \rho_- > 0$ such that $h(w) + p(y) + \frac{\rho_-}{2} \| w
- y \|^2 \geq l$ for all $w \in Z_1$, $y \in Z_2$. Then for any $\epsilon > 0$
such that $\epsilon \leq 1/(4 \rho_- + 2)$, if the parameters in
Algorithm~\ref{pb-admm} satisfy
$
\alpha = \epsilon, \rho = \frac{1}{2\epsilon}, \mu > \frac{2}{\epsilon}, w_0 = y_0,
$
then a subsequence of $\{(w_k; y_k; \lambda_k)\}_{k \geq 1}$ converges to $(w^*; y^*; \lambda^*)$ such that
\begin{align*}
0 \in \partial (h + \bk1_{Z_1})(w^*) + \lambda^*, \ 0 \in \partial (p + \bk1_{Z_2} )(y^*) - \lambda^*, \\
\| w^* - y^* \|^2 \le \alpha^2 \| \lambda^* \|^2 \leq ( 64(h(w_0) + p(y_0)) - 64l + (14 + 5\epsilon) \| \lambda_0 \|^2 ) \epsilon.
\end{align*}
\end{corollary}
\begin{proof}
Let $\nu = \alpha = \epsilon, R = 2$ and $\tau = 2$, then 
\begin{align*}
c_1(\nu) & = \frac{\mu}{2} - \frac{\tau}{2 \nu} = \frac{\mu}{2} - \frac{1}{\epsilon} > 0, \ c_2 = \frac{\rho}{2} = \frac{1}{4\epsilon}, \\
c_3(\nu) & = \tau \left( \alpha - \frac{\nu}{2} \right) - \frac{(1-\rho\alpha)(2-\rho\alpha)}{2\rho} = 2 \left( \epsilon - \frac{\epsilon}{2} \right) - \frac{(1-\frac{1}{2})(2-\frac{1}{2})}{1/\epsilon} = \frac{\epsilon}{4}, \\
c_4(R) & = \frac{(1-\rho\alpha)[(R+1)\rho\alpha - 1]}{2 \rho R} = \frac{(1-1/2)[(2+1)/2 - 1]}{ 2/\epsilon } = \frac{\epsilon}{8}, \\
c_5(R) & = \frac{1-\rho \alpha}{2\rho} \left[ \tau - (1-\rho\alpha)R \right] = \frac{1-1/2}{1/\epsilon} [2 - (1-1/2)\cdot 2] = \frac{\epsilon}{2}.
\end{align*}
Therefore, Assumption~\ref{Ass: c} holds. Since $\rho = \frac{1}{2\epsilon} \geq 2 \rho_- + 1$, we have that 
$$
h(w) + p(y) + \frac{\rho}{2} \| w - y \|^2 \geq h(w) + p(y) + \frac{\rho_-}{2} \| w - y \|^2 \geq l, \forall w \in Z_1, y \in Z_2.
$$
Thus Assumption~\ref{lowerbound} holds. Based on Theorem~\ref{thm: Subconv}, \uvs{it suffices} to show that $ \alpha^2 \| \lambda^* \|^2 \leq ( 64(h(w_0) + p(y_0)) - 64l + (14 + 5\epsilon) \| \lambda_0 \|^2) \epsilon$. By \eqref{Bound-Ineq1} in Theorem~\ref{thm: Subconv}, for $k \geq 1$,
\begin{align}
\notag
 \| \lambda_k \|^2 & \leq \frac{1}{c_4(R)}( P_\tau^1 - l ) =  \left( \frac{2}{1-\rho\alpha} \right) \left( \frac{\rho R}{(R+1)\rho\alpha-1} \right) ( P_\tau^1 - l )  \\
\label{Bound-Ineq2}
\implies  \alpha^2 \| \lambda_k \|^2 & \leq \left( \frac{2}{1-\rho\alpha} \right) \left( \frac{\alpha R}{R+1-1/(\rho\alpha)} \right) ( P_\tau^1 - l ).
\end{align}
\vvs{Since $P_\tau^1 = \tscrL_{\rho,\alpha}(w_1, y_1, \lambda_1) + \frac{(1-\rho\alpha)\alpha}{2} \| \lambda_1 \|^2 + \tau \left( \frac{1-\rho\alpha}{2\rho} \right) \| \lambda_1 - \lambda_0\|^2$, }
\begin{align}
\notag
 P_\tau^1 
& \overset{\eqref{Subconv-lemma2}}{\leq}  \tscrL_{\rho,\alpha}(w_0, y_0, \lambda_0) + \frac{(1-\rho\alpha)\alpha}{2} \| \lambda_0 \|^2  + \tau \left( \frac{1-\rho\alpha}{2\rho} \right) \| \lambda_1 - \lambda_0\|^2 \\
\notag
& \quad - \frac{\mu}{2} \| w_1 - w_0 \|^2 - \frac{\rho}{2} \| y_1 - y_0 \|^2 + \frac{(1-\rho\alpha)(2-\rho\alpha)}{2\rho}\| \lambda_1 -\lambda_0 \|^2 \\
\notag
& \leq h(w_0) + p(y_0) + \frac{\rho}{2}\| r_0 \|^2 + (1-\rho\alpha) \lambda_0^T(r_0 - \alpha \lambda_0) \\
& \notag + \frac{(1-\rho\alpha)\alpha}{2}\| \lambda_0 \|^2 + \frac{(1-\rho\alpha)(2+\tau - \rho\alpha)}{2\rho} \| \lambda_1 - \lambda_0 \|^2 \\
\notag
& = h(w_0) + p(y_0) + 0 - (1-1/2) \epsilon \| \lambda_0 \|^2 + \frac{(1-1/2)\epsilon}{2} \| \lambda_0 \|^2 \\
\notag
& + \frac{(1-1/2)(2+2 - 1/2)}{1/\epsilon} \| (1-1/2)\lambda_0  + \rho (w_1 - y_1) - \lambda_0 \|^2  \\
\notag
& \leq h(w_0) + p(y_0) - \frac{\epsilon}{4} \| \lambda_0 \|^2 + \frac{7 \epsilon}{4} \cdot \left( \frac{1}{2} \| \lambda_0 \|^2 + 2 \rho^2 \| w_1 - y_1 \|^2 \right) \\
\label{Bound-Ineq2.5}
& \leq h(w_0) + p(y_0) + \frac{5 \epsilon}{8} \| \lambda_0 \|^2 + \frac{7}{8 \epsilon} \| w_1 - y_1 \|^2.
\end{align}
Adding \eqref{Subconv-Ineq3} and \eqref{Subconv-Ineq4} and letting $k = 0$, we obtain the following.
\begin{align*}
\tilde{L}_{\rho,\alpha}(w_1,y_1,\lambda_0) - \tilde{L}_{\rho,\alpha}(w_0,y_0,\lambda_0) \leq - \frac{\mu}{2} \| w_1 - w_0 \|^2 - \frac{\rho}{2} \| y_1 - y_0 \|^2,
\end{align*}
which indicates that
\begin{align}
\notag
& h(w_1) + p(y_1) + (1-\rho \alpha)\lambda_0^T(w_1 - y_1 - \alpha \lambda_0) + \frac{\rho}{2} \| w_1 - y_1 \|^2 \\
\notag
& \leq h(w_0) + p(y_0) + (1-\rho \alpha)\lambda_0^T(w_0 - y_0 - \alpha \lambda_0) + \frac{\rho}{2} \| w_0 - y_0 \|^2 \\
\notag
& = h(w_0) + p(y_0) - (1-\rho \alpha) \alpha \| \lambda_0 \|^2 \\
\notag
\implies & (1-\rho \alpha) \lambda_0^T(w_1 - y_1) + \frac{\rho - \rho_- }{2} \| w_1 - y_1 \|^2 \\
\notag
& \leq  h(w_0) + p(y_0) - h(w_1) - p(y_1) - \frac{\rho_-}{2} \| w_1 - y_1 \|^2 \\
\notag
\implies & \frac{\rho - \rho_-}{2} \| w_1 - y_1 \|^2 \leq  h(w_0) + p(y_0) - l - (1-\rho \alpha) \lambda_0^T(w_1 - y_1) \\
\notag
& \leq h(w_0) + p(y_0) - l + (1-\rho \alpha) \| \lambda_0 \|^2/2 + (1-\rho \alpha) \| w_1 - y_1 \|^2/2 \\
\notag
\implies & \frac{\rho - \rho_- - (1-\rho \alpha)}{2} \| w_1 - y_1 \|^2 \leq  h(w_0) + p(y_0) - l + (1-\rho \alpha) \| \lambda_0 \|^2/2 \\
\notag
\implies & \frac{ 1/\epsilon - 2\rho_- - 1}{4} \| w_1 - y_1 \|^2 \leq  h(w_0) + p(y_0) - l + \| \lambda_0 \|^2/4. \\
\label{Bound-Ineq2.6}
\implies &  \| w_1 - y_1 \|^2 \leq  8\epsilon( h(w_0) + p(y_0) - l + \| \lambda_0 \|^2/4),
\end{align}
where the last inequality holds because $\epsilon \leq 1/(4 \rho_- + 2)$. By \eqref{Bound-Ineq2.5} and \eqref{Bound-Ineq2.6},
\begin{align}\label{Bound-Ineq3}
P_\tau^1 \leq 8( h(w_0) + p(y_0) ) - 7l + \left( \frac{7}{4} + \frac{5\epsilon}{8} \right) \| \lambda_0 \|^2
\end{align}
By combining \eqref{Bound-Ineq2} and \eqref{Bound-Ineq3}, we have for any $k \geq 1$,
\begin{align}
\notag
& \alpha^2 \| \lambda_k \|^2 \\
\notag
& \leq \frac{2}{1-1/2}\cdot\frac{2}{2+1-2} \left( 8( h(w_0) + p(y_0) ) - 7l + \left( \frac{7}{4} + \frac{5\epsilon}{8} \right) \| \lambda_0 \|^2 - l \right) \epsilon \\
\label{Bound for multiplier}
& = ( 64(h(w_0) + p(y_0)) - 64l + (14 + 5\epsilon) \| \lambda_0 \|^2) \epsilon.
\end{align}
This implies that $\alpha^2 \| \lambda^* \|^2 \leq ( 64(h(w_0) + p(y_0)) - 64l + (14 + 5\epsilon) \| \lambda_0 \|^2 ) \epsilon. $ \qed
\end{proof}

\begin{remark} \label{rm: stopcrit.}
Based on the optimality conditions of (Update-1), (Update-2), and the multiplier update, the following holds for any $k \geq 0$:
\begin{align*}
\begin{array}{l}
0 \in \partial (h + \bk1_{Z_1})(w_{k+1}) + \lambda_{k+1} + \rho( y_{k+1} - y_k ) + \mu (w_{k+1} - w_k ), \\
0 \in \partial (p + \bk1_{Z_2} )( y_{k+1} ) - \lambda_{k+1},  \ 
 w_{k+1} - y_{k+1} - \alpha \lambda_k  = ( \lambda_{k+1} - \lambda_k )/ \rho.
\end{array}
\end{align*}
According to \eqref{Bound for multiplier}, if we choose the parameters as in Corollary~\ref{Corr: Parameter choosing}, the stopping criteria in Algorithm~\ref{pb-admm} indicates that: 
\begin{align*}
& {\rm dist}(0, \partial( h + \bk1_{Z_1} ) (w_{k+1}) + \lambda_{k+1} ) < \epsilon_0, \ 0 \in \partial( p + \bk1_{Z_2} )( y_{k+1} ) - \lambda_{k+1}, \\
& \| w_{k+1} - y_{k+1} \| < \sqrt{ ( 64(h(w_0) + p(y_0)) - 64l + (14 + 5\epsilon) \| \lambda_0 \|^2) \epsilon } + \epsilon_0.
\end{align*}
\end{remark}

Finally we will show that the conditions \eqref{Critical condition} in Theorem~\ref{thm: Subconv} are equivalent to KKT conditions of \eqref{tractable decomposition} with a feasibility error $\alpha \lambda^*$. 

Denote $w^* \triangleq (x_+^*;x_-^*;\xi^*), w \triangleq (x^+;x^-;\xi)$, $y^* \triangleq (y_1^*;y_2^*;y_3^*), y \triangleq (y_1;y_2;y_3)$. 
By Definition~\ref{first order KKT}, $(w^*;y^*)$ satisfies first-order KKT conditions of \eqref{tractable decomposition} if there exist $\mu \in \bR$, $\beta_1, \beta_2, \beta_3, \beta_4 \in \bR^n$, $ \pi \in \bR^m $ such that
\begin{subequations} \label{KKT-tf}
\begin{align}
\label{KKT-tf-1}
\pmat{\nabla f_Q (x^*_+ - x^*_-) + \nabla g(y_1^* - y_2^*) \\ -\nabla f_Q (x^*_+ - x^*_-) - \nabla g(y_1^* - y_2^*) \\ - \gamma e} + \pmat{ \mu \xi^*-\beta_1 - A^T \pi \\  \mu \xi^* -\beta_2 + A^T \pi  \\ \mu (x^*_+ + x^*_-) + \beta_4 - \beta_3 } = 0,
& \\
\label{KKT-tf-2}
0 \leq \beta_1 \perp y_1^* \geq 0, \\
\label{KKT-tf-3}
0 \leq \beta_2 \perp y_2^* \geq 0, \\
\label{KKT-tf-4}
0 \leq \beta_3 \perp y_3^* \geq 0, \\
\label{KKT-tf-5}
0 \leq \beta_4 \perp e - y_3^* \geq 0, \\
\label{KKT-tf-6}
0 \leq \pi \perp A(y_1^* - y_2^*) - b \geq 0, \\
\label{KKT-tf-7}
( x^*_+ + x^*_- )^T \xi^* = 0, \\
\label{KKT-tf-8}
w^* - y^* = 0 .
\end{align}
\end{subequations}
It can be easily seen that \eqref{KKT-tf} is equivalent to \eqref{KKT-F} by merging $w^*$ and $y^*$. Recall that according to discussions in Section~\ref{sec:MPCC}, point satisfying \eqref{KKT-F} is exactly the local minimum of \eqref{mpcc}, thus the local minimum of $\ell_0$-minimization \eqref{L0min}.

\begin{theorem}\label{KKTwithError} \rm
Suppose that $(w^*; y^*; \lambda^*)$ satisfies \eqref{Critical condition}, and recall that
$h(w) = f_Q(x^+ - x^-) + \gamma e^T(e - \xi)$, $p(y) = g(y_1 - y_2)$ are smooth functions. 
Assume that vector $(\xi^*; \xi^* ; x_+^*+x_-^*) \neq 0$. Then $\exists \mu \in \bR, \beta_1, \beta_2, \beta_3, \beta_4 \in \bR^n, \pi \in \bR^m $ such that \eqref{KKT-tf-1} - \eqref{KKT-tf-7} hold and $w^* - y^* = \alpha \lambda^*$.
\end{theorem}
\begin{proof}
We know that $\partial \bk1_{Z_1}(w) = \sN_{Z_1}(w)$, $\partial \bk1_{Z_2}(y) = \sN_{Z_2}(y)$. Due to \eqref{Critical condition} and the smoothness of function $h$, $0 \in   \partial (h  + \bk1_{Z_1}) (w^*) + \lambda^* \Rightarrow 0 \in \nabla_w h(w^*) + \lambda^* + \partial \bk1_{Z_1}(w^*) \Rightarrow - \nabla_w h(w^*) - \lambda^* \in \sN_{Z_1}(w^*)$. Recall $Z_1 = \{(x^+;x^-;\xi) \in \bR^{3n} \mid \xi^T(x^+ + x^-) = 0 \}$. Then by \diff{Lemma~\ref{NormalConeCharacterization} in the Appendix} and the assumption $(\xi^*; \xi^*;  x_+^*+x_-^*) \neq 0$, we have $\sN_{Z_1}(w^*) = \{ \mu(\xi^*; \xi^*;  x_+^*+x_-^*) \mid \mu \in \bR \}$. Therefore, $\exists \mu \in \bR$ s.t. 
\begin{align} \label{w-update: KKTwithError}
\nabla_w h(w^*) + \lambda^*  + \mu(\xi^*; \xi^*;  x_+^*+x_-^*) = 0.
\end{align}
On the other hand, \eqref{Critical condition} and smoothness of function $p$ imply $0 \in \partial (p + \bk1_{Z_2})(y^*) - \lambda^* \Rightarrow 0 \in \nabla_y p(y^*) - \lambda^* + \partial \bk1_{Z_2}(y^*) \Rightarrow - \nabla_y p(y^*) + \lambda^* \in \sN_{Z_2}(y^*)$. Since $Z_2$ is a convex set, $\sN_{Z_2}(y^*) = \{ v \mid v^T(y - y^*) \leq 0, \forall y \in Z_2 \}$. Therefore, $(  \nabla_y p(y^*) - \lambda^* )^T(y - y^*) \geq 0, \forall y \in Z_2$. This indicates that $y^*$ is the optimal solution of the linear program: $\min_{y \in Z_2} \{ (\nabla_y p(y^*) - \lambda^* )^Ty \}$. Thus the KKT conditions are satisfied at $y^*$, i.e. $\exists \beta_1, \beta_2, \beta_3, \beta_4 \in \bR^n, \pi \in \bR^m$ s.t. 
\begin{align}\label{y-update: KKTwithError}
\begin{aligned}
\nabla_y p(y^*) - \lambda^* + (-\beta_1 - A^T\pi ; -\beta_2 + A^T \pi ; \beta_4 - \beta_3 ) = 0, \\
0 \leq \beta_1 \perp y_1^* \geq 0, 0 \leq \beta_2 \perp y_2^* \geq 0, 0 \leq \beta_3 \perp y_3^* \geq 0,\\
0 \leq \beta_4 \perp e - y_3^* \geq 0, 0 \leq \pi \perp A(y_1^* - y_2^*) - b \geq 0.
\end{aligned}
\end{align}
\vvs{From} \eqref{w-update: KKTwithError} and \eqref{y-update: KKTwithError}, utilizing the def. of $h$ and $p$, and adding the feasibility constraints $ (x^*_+ + x^*_-)^T \xi^* = 0$, $w^* - y^* = \alpha \lambda^*$, we obtain the perturbed KKT conditions. \qed
\end{proof}

\section{{Preliminary numerics}} \label{sec:num}
In Section~\ref{Subsec: test problem}, we describe the test problem of interest while in Section~\ref{subsec:tract-subprob}, we study the impact of tractability by comparing tractable ADMM frameworks to their standard counterpart. In Section~\ref{subsec: admm_vs_others}, performance of (ADMM$_{\rm cf}$) is examined by comparing it with other methods.\footnote{All experiments are conducted on Matlab and the code is uploaded to \texttt{https://github.com/yue-xie/l0-minimization}.}

	\subsection{Least squares regression with $\ell_0$-norm}\label{Subsec: test problem}

Suppose $f_Q(x) \triangleq \|Cx - d\|^2$, $C \in \bR^{p \times n}, g(x) \equiv 0$, and there is no linear constraint $Ax \geq b$ in \eqref{L0min}, leading 
	to the following $\ell_0$-regularized least-squares regression:
	\begin{align}\label{LSR}\tag{$\ell_0$-LSR}
	\min \|Cx - d\|^2 + \gamma \| x\|_0.
	\end{align}
This special case finds application in signal recovery and regression
problems. The rows of $C$ are generated from a
multivariate normal $N(0,I_n)$ \us{while the true}
coefficients $x^{\rm true}$ \us{are} created as \us{follows}: 
	(1) Generate $x^{\rm true}_i$ for $i =1,\hdots,n$ from uniform
	distribution $U(-60,60)$.
	(2) If $\abs{x^{\rm true}_i} \geq \frac{60 \kappa}{n}$ then $x^{\rm
		true}_i \leftarrow 0$ for $i = 1, \hdots, n$.
 Then $x^{\rm true}$ is approximately \yx{$\kappa$-sparse (or
			$\|x^{\rm true}\|_0 \approx \kappa$)}. The observation vector $d = Cx^{\rm true} + \epsilon$, where $\epsilon_i \sim N(0, \sigma^2)$ and $\sigma^2 = \| x^{\rm
		true} \|^2 / 10$.
	
\subsection{Impact of tractable subproblems}\label{subsec:tract-subprob}

In this subsection, we will compare the ADMM based algorithms (Algorithms~\ref{pb-admm}-\ref{Alg: intractable}) proposed in Section~\ref{sec: ADMM} to resolve \eqref{LSR}.

\noindent{\bf Algorithm descriptions and settings.} We start all algorithms from an initial point $w_0 = y_0 = (e; \textbf{0}_n; \textbf{0}_n)$, $\lambda_0 = {\bf 0}_{3n}$. The maximum runtime allowed is 200s. Experiments are run on CPU of 1.3GHz Intel Core i5 with 8GB memory. Other settings are as follows.\\
\noindent{\bf (ADMM$_{\rm cf}$)}: Please refer to Algorithm~\ref{admm} for details. Specifically, we use the following rule for updating $\rho_k$:\footnote{This update rule for $\rho_k$ is related to the convergence property of (ADMM$_{\rm cf}$). Please refer to \diff{Section~\ref{subsec: Anal.ADMMcf}} for details.}
\begin{align*}
& \mbox{If }  (\rho_k - \delta) \| y_{k+1} - y_k \| < \sqrt{2} \| \lambda_{k+1} - \lambda_k \| \mbox{ and } \rho_k \leq \rho_{\max},
\mbox{ then } \rho_{k+1} := \delta_{\rho} \rho_k; \\
& \mbox{else } \rho_{k+1}:= \rho_k.
\end{align*}
In addition, $\rho_0 = 1$, $\rho_{\max} = 2000$, $\delta = \rho_0 / 2$, $\delta_\rho = 1.01$, $\epsilon = 10^{-4}$;\\
\noindent{\bf (ADMM$_{\rm cf}^{\mu,\alpha,\rho}$)}: Algorithm~\ref{pb-admm}. $\alpha = 10^{-3}$, $\rho = 1/(2\alpha)$, $\mu = 3/\alpha$, $\epsilon_0 = 10^{-2}$.\\
\noindent{\bf (ADMM$_0$)}: Algorithm~\ref{Alg: intractable} with Update-1
solved by \texttt{Baron}. The maximum runtime for \texttt{Baron} is set to
200s. Note that we do not fix the penalty parameter at a value suggested by
theory, which more often than not is impractical and involves problem parameter
estimation. Instead, we update it adaptively. The update rule for $\rho_k$,
inspired by augmented Lagrangian schemes \cite{Birgin2010}, is as follows:
\begin{align*}
& \mbox{If } k = 0 \mbox{ or } ( h_{k+1} \geq 10^{-2} \mbox{ and } h_{k+1} > 0.9 h_k \mbox{ and } \rho_k < 500 ), \rho_{k+1} = 1.01 \rho_k; \\
& \mbox{otherwise } \rho_{k+1} = \rho_k;
\end{align*}
where $h_k = \| w_k - y_k \|$ for all $k \geq 0$. $\rho_0 = 1$, $\epsilon = 10^{-4}$. \\
\textit{Stopping criteria.} The stopping criteria for (ADMM$_{\rm cf}$) and (ADMM$_0$) guarantee that the KKT residual is below $\epsilon$ if terminated within time limit. For (ADMM$_{\rm cf}^{\mu,\alpha,\rho}$),  it is guaranteed that the KKT residual is below $\epsilon_0 + \mathcal{O}(\sqrt{\epsilon})$ (See Remark~\ref{rm: stopcrit.}). Thus, stopping criteria of the three algorithms are related to the optimality conditions. 

\noindent{\bf Metric.} In Table~\ref{tab:hiddenconvex_vs_nonconvex},
KKT residual for (ADMM$_{\rm cf}$), (ADMM$_{\rm cf}^{\mu,\alpha,\rho}$) and (ADMM$_0$) are $\max\{  \rho_{K-1} \| y_K - y_{K-1}  \|, \| (x_K^+;x_K^-;\xi_K) - y_K \| \}$, $\max\{ \| \rho ( y_K - y_{K-1} ) +  \mu (x_K - x_{K-1}) \|,  \| w_K - y_K \| \}$ and $\max\{ \rho_{K-1} \| y_K - y_{K-1}  \|, \| w_K - y_K \| \}$, respectively. $K$ is the last iteration. 

\noindent{\bf Results.} In Table~\ref{tab:hiddenconvex_vs_nonconvex}, we
provide a comparison of (ADMM$_{\rm cf}$), (ADMM$_{\rm
cf}^{\mu,\alpha,\rho}$) and (ADMM$_0$) to address \eqref{LSR}. Note
that (ADMM$_{\rm cf}$), (ADMM$_{\rm cf}^{\mu,\alpha,\rho}$) are designed for
formulation \eqref{tractable decomposition}, which renders tractable
subproblems, while (ADMM$_0$) is for formulation \eqref{intractable decomp},
which requires global resolution of an MPCC as the subproblem. Therefore, even
though (ADMM$_0$) can be efficient when the dimension is relatively low, but
becomes less so when $n$ is larger. This is because the subproblem solver does not scale well with problem size and requires significant time for larger dimensions. This is supported by the drastically reduced number of outer loop iterations and large
KKT residual upon termination when $n$ is larger. Meanwhile, (ADMM$_{\rm cf}$) and (ADMM$_{\rm cf}^{\mu,\alpha,\rho}$) appear to be scale far better with $n$ due to tractability
of the subproblem. Furthermore, it can be seen that
(ADMM$_{\rm cf}$) is far more efficient than the other two methods. It spends far less time to find solutions with lower KKT residual. In fact, it also provides better objective function value than the other two methods as observed during the experiment. Further exploration of (ADMM$_{\rm cf}$) through comparison with other $\ell_0$-minimization solvers is presented in the next subsection.

	\begin{table}[htbp]
	\begin{center}
{	
	\begin{tabular}{c|ccc|ccc|ccc}
\hline	
($n$, $\| x^{\rm tr.} \|_0$, $\gamma$) & \multicolumn{3}{c}{(ADMM$_{\rm cf}$)} &  \multicolumn{3}{|c|}{(ADMM$_{\rm cf}^{\mu,\alpha,\rho}$)} &  \multicolumn{3}{c}{(ADMM$_0$)} \\
\hline
& t(s) & res. & iter. & t(s) & res. & iter. & t(s) & res. & iter.\\
\hline
(20, 1, 1) & 0.38 & 9.4e-5 & 587  & 2.2e0 & 1.0e-2 & 4.90e3 & 201 & 3.0e-3 & 135 \\
(20,  1,  10)  & 0.46 & 7.9e-5 & 515 & 2.0e2 & 3.5e-1 & 2.05e5 & 200 & 1.0e-2 & 158 \\
(20, 4,  1) & 0.26 & 7.2e-5 & 214 & 2.0e2 & 7.6e-2 & 2.67e5 & 200 & 8.7e-3 & 197 \\
(20,  4,  10) & 0.70 & 9.9e-5 & 798 & 2.0e2 & 1.8e-1 & 2.18e5 & 200 & 2.1e-2 & 150 \\
(50,  10,  1) & 1.10 & 1.0e-4 & 781 & 6.3e-1 & 9.9e-3 & 551 & 209 & 7.2e0 & 2 \\
(50,  10,  10) & 2.10 & 9.9e-5 & 1093 & 2.0e0 & 1.6e-1 & 1.18e5 & 201 & 2.5e-1 & 101 \\
(50,  18,  1) & 0.62 & 7.7e-5 & 283 & 2.0e0 & 1.8e-1 & 1.42e5 & 215 & 1.3e1 & 2 \\
(50,  18, 10) & 4.20 & 9.9e-5 & 1802 & 2.0e0 & 6.9e-1 & 1.29e5 & 201 & 4.8e-1 & 47 \\
(100,  6,  1) & 0.58 & 9.3e-5 & 617 & 2.5e-1 & 9.6e-3 & 200 & 209 & 9.8e0 & 2 \\
(100,  6,  10) & 0.66 & 9.8e-5 & 593 & 2.0e2 & 1.4e-1 & 2.10e5 & 399 & 4.8e0 & 14 \\
(100,  19,  1) & 0.76 & 9.2e-5 & 483 & 4.0e-1 & 9.9e-3 & 376 & 211 & 1.0e1 & 2 \\
(100,  19, 10) & 1.80 & 9.0e-5 & 535 & 2.0e2 & 1.3e-1 & 1.72e5 & 289 & 5.5e0 & 5 \\
\hline
	\end{tabular}
}
	\caption{ Comparison of methods on \eqref{LSR}, $p = 10$
 } 
 \label{tab:hiddenconvex_vs_nonconvex}
	\end{center}
	\end{table}

	\subsection{Comparison among (ADMM$_{\rm cf}$) and other methods} \label{subsec: admm_vs_others}
	In this set of experiments, we test (ADMM$_{\rm cf}$) on \eqref{LSR} with higher dimensions ($p =
	256, n = 1024$) and compare it with other known
	methods directly addressing $\ell_0$-minimization: iterative hard
	thresholding (IHT) and iterative hard thresholding with warm start
	(IHTWS) \cite{Blumensath2008}.  We again test the schemes on \eqref{LSR} and choose almost the same settings as in Section \ref{Subsec: test problem}, the only difference being that $\epsilon \in
	\bR^p, \epsilon_i \sim N(0, \sigma^2), i.i.d., \sigma^2 =
	\frac{(x^{\rm true})^TI_nx^{\rm true}}{\rm SNR}$, where SNR refers
	to the  signal-to-noise ratio. All experiments are conducted on CPU of 3.4GHz Intel Core i7 with 16GB memory.

\noindent{\bf Algorithm descriptions and settings.} 

\noindent{\bf (IHT) and (IHTWS)}: (IHT) is implemented with 50 initial points (including the origin and points drawn from normal distribution $N(0,I_n)$), and the best solution is chosen. (IHTWS) is warm-started from a point computed by matching pursuit. The termination condition for both (IHT) and (IHTWS) is $\|x_{k+1} - x_k\| < 1\times 10^{-6}$. 

\noindent{\bf (ADMM$_{\rm cf}$)}:
Implementation of (ADMM$_{\rm cf}$) is almost the same with the last experiment in Section~\ref{subsec:tract-subprob}: Initial point is selected as $ y_0 = ({e}_{n},{\bf 0}_{n},{\bf 0}_{n}) $, ${\lambda_0 = \bold{0}_{3n}}$, and the parameters are chosen as $\rho_0 = \gamma$, $\epsilon = 10^{-4}$, $\delta_\rho = 1.01$, $\delta = \rho_0/2$, $\rho_{\max} = 2000$, ${\rm time}_{\max} = 300$ for all cases. 

\noindent{\bf Metric.} In Table~\ref{tab:ADMMcf vs. others}, ${\rm RDF} \triangleq \frac{f_{\rm method} - f_{\rm ADMM_{\rm cf}}}{f_{\rm ADMM_{\rm cf}}}$, where $f_{\rm ADMM_{\rm cf}}$ is calculated as follows. Suppose Algorithm~\ref{admm} terminates when $k = T$. Let $(\bar x^+; \bar x^-; \bar \xi) = y_{T+1}$. Then the solution given by (ADMM$_{\rm cf}$) is $\bar x^+ - \bar x^-$ and
\begin{align*}
& f_{\rm ADMM_{\rm cf}} = 
\begin{cases}
\| C(\bar x^+ - \bar x^-)-d \|^2 + \gamma (n - e^T \bar \xi), \\
\hspace{+.5in} \mbox{if } \max( \| w_{T+1} - y_{T+1} \|,\rho_T\|y_{T+1}-y_T\|) \leq \epsilon \\
\| C(\bar x^+ - \bar x^-)-d \|^2 + \gamma \| \bar x^+ - \bar x^- \|_0, \\
\hspace{+.5in} \mbox{if } \max( \| w_{T+1} - y_{T+1} \|, \rho_T\|y_{T+1}-y_T\|) > \epsilon 
\end{cases}
\end{align*}

	\begin{table}[htbp]
	{
	\begin{center}
	\begin{tabular}{c|cc|ccc|ccc}
	\hline
     & \multicolumn{2}{c|}{ RDF } & \multicolumn{3}{c|}{time(s)} & \multicolumn{3}{c}{$ \| x^{\rm sol.} \|_0 $} \\
	\hline
	(SNR, $\| x^{\rm tr.} \|_0 $, $\gamma$) & (I) & (W)  &  (I) & (W) &
	(A) & (I) & (W) & (A) \\
	\hline
	(5, 16, 0.10) & 1.92 & -0.02 & 1.2 & 1.2 & 95.4 & 624 & 194 & 98 \\ 
\hline 
(5, 16, 1.00) & 0.83 & -0.03 & 1.3 & 0.3 & 17.2 & 135 & 77 & 62 \\ 
\hline 
(5, 16, 10.00) & 0.58 & -0.05 & 1.5 & 0.1 & 41.7 & 7 & 15 & 12 \\ 
\hline 
(5, 16, 50.00) & 0.93 & -0.10 & 2.0 & 0.1 & 42.5 & 0 & 10 & 7 \\ 
\hline 
(5, 87, 0.10) & 2.93 & 0.62 & 1.2 & 2.2 & 102.1 & 981 & 404 & 180 \\ 
\hline 
(5, 87, 1.00) & 3.37 & 0.59 & 1.3 & 1.7 & 19.5 & 893 & 326 & 161 \\ 
\hline 
(5, 87, 10.00) & 2.94 & 0.46 & 1.5 & 4.5 & 75.8 & 639 & 230 & 133 \\ 
\hline 
(5, 87, 50.00) & 0.82 & 0.33 & 20.1 & 1.0 & 92.4 & 261 & 168 & 103 \\ 
\hline 
(5, 253, 0.10) & 2.75 & 0.80 & 1.4 & 2.5 & 236.3 & 1006 & 483 & 268 \\ 
\hline 
(5, 253, 1.00) & 2.69 & 0.59 & 1.5 & 2.0 & 234.2 & 990 & 425 & 268 \\ 
\hline 
(5, 253, 10.00) & 2.49 & 0.34 & 1.6 & 1.8 & 300.0 & 932 & 357 & 267 \\ 
\hline 
(5, 253, 50.00) & 2.15 & 0.11 & 1.9 & 3.1 & 63.7 & 831 & 293 & 264 \\ 
\hline 
(10, 20, 0.10) & 1.65 & -0.18 & 1.3 & 0.5 & 118.7 & 529 & 128 & 68 \\ 
\hline 
(10, 20, 1.00) & 0.29 & -0.18 & 1.5 & 0.2 & 15.4 & 76 & 53 & 22 \\ 
\hline 
(10, 20, 10.00) & 2.09 & -0.03 & 1.6 & 0.1 & 43.2 & 4 & 16 & 14 \\ 
\hline 
(10, 20, 50.00) & 0.34 & -0.18 & 2.1 & 0.1 & 58.7 & 0 & 10 & 4 \\ 
\hline 
(10, 82, 0.10) & 2.82 & 0.42 & 1.3 & 1.7 & 96.1 & 978 & 365 & 164 \\ 
\hline 
(10, 82, 1.00) & 3.83 & 0.56 & 1.4 & 4.3 & 20.9 & 881 & 285 & 150 \\ 
\hline 
(10, 82, 10.00) & 2.68 & 0.35 & 1.7 & 1.8 & 73.8 & 599 & 209 & 117 \\ 
\hline 
(10, 82, 50.00) & 0.74 & 0.23 & 3.8 & 0.7 & 72.8 & 212 & 138 & 93 \\ 
\hline 
(10, 249, 0.10) & 2.32 & 0.66 & 1.4 & 2.6 & 12.4 & 1010 & 506 & 304 \\ 
\hline 
(10, 249, 1.00) & 2.26 & 0.47 & 1.4 & 2.1 & 8.7 & 990 & 446 & 304 \\ 
\hline 
(10, 249, 10.00) & 2.11 & 0.25 & 1.5 & 1.8 & 6.1 & 945 & 381 & 304 \\ 
\hline 
(10, 249, 50.00) & 1.91 & 0.05 & 1.8 & 2.4 & 6.6 & 877 & 316 & 301 \\ 
\hline 
	\end{tabular}
	\caption{Comparison of methods on \eqref{LSR} ($p \times n$ = 256 $\times$ 1024 ) }\label{tab:ADMMcf vs. others}
	(I): (IHT) (W): (IHTWS) (A): (ADMM$_{\rm cf}$)
	\end{center}
	}
	\end{table}
	
\noindent{\bf Results.} From Table \ref{tab:ADMMcf vs. others}, we conclude the following:
	\begin{enumerate}
	\item[(1)] Although ($\rm{ADMM_{\rm cf}}$) takes more time, it
	generally produces solutions that are  superior to (IHT) in objective
	function value and provides better values than (IHTWS) in most cases. Note that ($\rm{ADMM_{\rm cf}}$) is cold started. 
	\item[(2)] ($\rm{ADMM_{\rm cf}}$) generally produces sparser solution than (IHTWS) and (IHT), which indicates that ($\rm{ADMM_{\rm cf}}$) scheme is potentially more favorable from a statistical standpoint.
	\end{enumerate}

	\section{Concluding remarks and future work} \label{sec:con} We consider a full complementarity reformulation of a general class of
$\ell_0$-norm minimization problems. The focus of this paper lies on the characterization and efficient computation of KKT points for this formulation. In particular, we show that a suitable (Guignard) constraint qualification holds at every feasible point. Moreover, when $f$ is a convex function, a point satisfies the first-order KKT conditions if and only if it is a local minimizer. Next, two tractable ADMM schemes are presented for resolution. In these schemes, a hidden convexity property is leveraged to allow for tractable resolution of ADMM subproblems. For the perturbed proximal ADMM framework, subsequential convergence to KKT points with arbitrarily small error under mild assumptions can be shown. Preliminary empirical studies show the significance of having tractable subproblems in ADMM schemes and that the tractable ADMM framework compares well with its competitors. In future work, we may consider characterization and computation of KKT points of problems complicated by cardinality constraints \eqref{Msparse} and affine sparsity constraints \cite{dong17structural}. \\

\noindent {\bf Acknowledgements.}
The authors would like to acknowledge an early discussion with Dr. Ankur
Kulkarni of IIT, Mumbai,  as well as the inspiration
provided by Dr. J. S. Pang during his visit to Penn. State University, and suggestion by Dr. Mingyi Hong in INFORMS 2018, Denver. 




\bibliographystyle{siam}

\begin{thebibliography}{10}

\bibitem{10.2307/40801236}
{\sc H.~Attouch, J.~Bolte, P.~Redont, and A.~Soubeyran}, {\em Proximal
  alternating minimization and projection methods for nonconvex problems: an
  approach based on the {K}urdyka-{\L}ojasiewicz inequality}, Mathematics of
  Operations Research, 35 (2010), pp.~438--457.

\bibitem{bach2012optimization}
{\sc F.~Bach, R.~Jenatton, J.~Mairal, and G.~Obozinski}, {\em Optimization with
  sparsity-inducing penalties}, Foundations and Trends{\textregistered} in
  Machine Learning, 4 (2012), pp.~1--106.

\bibitem{ben2001lectures}
{\sc A.~Ben-Tal and A.~Nemirovski}, {\em Lectures on modern convex
  optimization: analysis, algorithms, and engineering applications}, vol.~2,
  SIAM, 2001, ch.~5. Computational tractability of convex programs.

\bibitem{Ben-Tal1996}
{\sc A.~Ben-Tal and M.~Teboulle}, {\em Hidden convexity in some nonconvex
  quadratically constrained quadratic programming}, Mathematical Programming,
  72 (1996), pp.~51--63.

\bibitem{bertsimas2016best}
{\sc D.~Bertsimas, A.~King, and R.~Mazumder}, {\em Best subset selection via a
  modern optimization lens}, The Annals of Statistics, 44 (2016), pp.~813--852.

\bibitem{bertsimas2009algorithm}
{\sc D.~Bertsimas and R.~Shioda}, {\em Algorithm for cardinality-constrained
  quadratic optimization}, Computational Optimization and Applications, 43
  (2009), pp.~1--22.

\bibitem{Birgin2010}
{\sc E.~G. Birgin, C.~A. Floudas, and J.~M. Mart{\'i}nez}, {\em Global
  minimization using an augmented {L}agrangian method with variable lower-level
  constraints}, Mathematical Programming, 125 (2010), pp.~139--162.

\bibitem{Blumensath2008}
{\sc T.~Blumensath and M.~E. Davies}, {\em Iterative thresholding for sparse
  approximations}, Journal of Fourier Analysis and Applications, 14 (2008),
  pp.~629--654.

\bibitem{doi:10.1137/18M1190689}
{\sc R.~Bo{\c t}, E.~Csetnek, and D.~Nguyen}, {\em A proximal minimization
  algorithm for structured nonconvex and nonsmooth problems}, SIAM Journal on
  Optimization, 29 (2019), pp.~1300--1328.

\bibitem{burdakov2016mathematical}
{\sc O.~P. Burdakov, C.~Kanzow, and A.~Schwartz}, {\em Mathematical programs
  with cardinality constraints: reformulation by complementarity-type
  conditions and a regularization method}, SIAM Journal on Optimization, 26
  (2016), pp.~397--425.

\bibitem{BurkeLagrange}
{\sc J.~Burke}, {\em Fundamentals of optimization, {Chapter 5, Langrange
  multipliers}}.
\newblock Course Notes, AMath/Math 515, University of Washington.

\bibitem{Burke2012note}
\leavevmode\vrule height 2pt depth -1.6pt width 23pt, {\em Numerical
  optimization}.
\newblock Course Notes, AMath/Math 516, University of Washington, Spring Term
  2012.

\bibitem{candes2008introduction}
{\sc E.~J. Cand{\`e}s and M.~B. Wakin}, {\em An introduction to compressive
  sampling}, IEEE signal processing magazine, 25 (2008), pp.~21--30.

\bibitem{dong17structural}
{\sc H.~Dong, M.~Ahn, and J.-S. Pang}, {\em Structural properties of affine
  sparsity constraints}, Mathematical Programming, 176 (2019), pp.~95--135.

\bibitem{donoho2006compressed}
{\sc D.~L. Donoho}, {\em Compressed sensing}, IEEE Transactions on information
  theory, 52 (2006), pp.~1289--1306.

\bibitem{facchinei2007finite}
{\sc F.~Facchinei and J.-S. Pang}, {\em Finite-dimensional variational
  inequalities and complementarity problems Volumn I}, Springer Science \&
  Business Media, 2007.

\bibitem{fan2001variable}
{\sc J.~Fan and R.~Li}, {\em Variable selection via nonconcave penalized
  likelihood and its oracle properties}, Journal of the American Statistical
  Association, 96 (2001), pp.~1348--1360.

\bibitem{Fang2019}
{\sc E.~X. Fang, H.~Liu, and M.~Wang}, {\em Blessing of massive scale: spatial
  graphical model estimation with a total cardinality constraint approach},
  Mathematical Programming, 176 (2019), pp.~175--205.

\bibitem{feng2018complementarity}
{\sc M.~Feng, J.~E. Mitchell, J.-S. Pang, X.~Shen, and A.~W{\"a}chter}, {\em
  Complementarity formulations of $\ell_0$-norm optimization problems}, Pacific
  Journal of Optimization, 14 (2018).

\bibitem{fung2011equivalence}
{\sc G.~Fung and O.~Mangasarian}, {\em Equivalence of minimal $\ell_0$ and
  $\ell_p$ norm solutions of linear equalities, inequalities and linear
  programs for sufficiently small p}, Journal of Optimization Theory and
  Applications, 151 (2011), pp.~1--10.

\bibitem{ge2011note}
{\sc D.~Ge, X.~Jiang, and Y.~Ye}, {\em A note on the complexity of ${L}_p$
  minimization}, Mathematical Programming, 129 (2011), pp.~285--299.

\bibitem{2017arXiv170201850G}
{\sc M.~L. Gon{\c{c}}alves, J.~G. Melo, and R.~D. Monteiro}, {\em Convergence
  rate bounds for a proximal {ADMM} with over-relaxation stepsize parameter for
  solving nonconvex linearly constrained problems}, arXiv preprint
  arXiv:1702.01850,  (2017).

\bibitem{Hajinezhad2019}
{\sc D.~Hajinezhad and M.~Hong}, {\em Perturbed proximal primal--dual algorithm
  for nonconvex nonsmooth optimization}, Mathematical Programming, 176 (2019),
  pp.~207--245.

\bibitem{DBLP:journals/siamjo/HongLR16}
{\sc M.~Hong, Z.~Luo, and M.~Razaviyayn}, {\em Convergence analysis of
  alternating direction method of multipliers for a family of nonconvex
  problems}, {SIAM} Journal on Optimization, 26 (2016), pp.~337--364.

\bibitem{DBLP:journals/corr/JiangLMZ16}
{\sc B.~Jiang, T.~Lin, S.~Ma, and S.~Zhang}, {\em Structured nonconvex and
  nonsmooth optimization: algorithms and iteration complexity analysis},
  Computational Optimization and Applications, 72 (2019), pp.~115--157.

\bibitem{liu2016global}
{\sc H.~Liu, T.~Yao, and R.~Li}, {\em Global solutions to folded concave
  penalized nonconvex learning}, Annals of statistics, 44 (2016), p.~629.

\bibitem{DBLP:journals/corr/LiuSG17}
{\sc Q.~Liu, X.~Shen, and Y.~Gu}, {\em Linearized {ADMM} for non-convex
  non-smooth optimization with convergence analysis}, CoRR, abs/1705.02502
  (2017).

\bibitem{luo96mathematical}
{\sc Z.-Q. Luo, J.-S. Pang, and D.~Ralph}, {\em Mathematical programs with
  equilibrium constraints}, Cambridge University Press, Cambridge, 1996.

\bibitem{rockafellar2009variational}
{\sc R.~T. Rockafellar and R.~J.-B. Wets}, {\em Variational analysis},
  vol.~317, Springer Science \& Business Media, 2009.

\bibitem{tibshirani1996regression}
{\sc R.~Tibshirani}, {\em Regression shrinkage and selection via the {L}asso},
  Journal of the Royal Statistical Society. Series B (Methodological),  (1996),
  pp.~267--288.

\bibitem{2015arXiv150503063W}
{\sc F.~Wang, W.~Cao, and Z.~Xu}, {\em Convergence of multi-block {B}regman
  {ADMM} for nonconvex composite problems}, Science China Information Sciences,
  61 (2018), p.~122101.

\bibitem{DBLP:journals/corr/WangZ17h}
{\sc J.~Wang and L.~Zhao}, {\em Nonconvex generalizations of {ADMM} for
  nonlinear equality constrained problems}, CoRR, abs/1705.03412 (2017).

\bibitem{DBLP:journals/corr/WangYZ15}
{\sc Y.~Wang, W.~Yin, and J.~Zeng}, {\em Global convergence of {ADMM} in
  nonconvex nonsmooth optimization}, Journal of Scientific Computing,  (2018).

\bibitem{DBLP:journals/corr/0002DFSG16}
{\sc Z.~Xu, S.~De, M.~A.~T. Figueiredo, C.~Studer, and T.~Goldstein}, {\em An
  empirical study of {ADMM} for nonconvex problems}, CoRR, abs/1612.03349
  (2016).

\bibitem{DBLP:journals/siamis/YangPC17}
{\sc L.~Yang, T.~K. Pong, and X.~Chen}, {\em Alternating direction method of
  multipliers for a class of nonconvex and nonsmooth problems with applications
  to background/foreground extraction}, {SIAM} Journal on Imaging Sciences, 10
  (2017), pp.~74--110.

\bibitem{zhang10nearly}
{\sc C.-H. Zhang}, {\em Nearly unbiased variable selection under minimax
  concave penalty}, The Annals of Statistics, 38 (2010), pp.~894--942.

\bibitem{10.2307/41714785}
{\sc C.-H. Zhang and T.~Zhang}, {\em A general theory of concave regularization
  for high-dimensional sparse estimation problems}, Statistical Science, 27
  (2012), pp.~576--593.

\bibitem{zhang10analysis}
{\sc T.~Zhang}, {\em Analysis of multi-stage convex relaxation for sparse
  regularization}, Journal of Machine Learning Research, 11 (2010),
  pp.~1081--1107.

\end{thebibliography}
\def\cprime{$'$}

\section{Appendix}

\subsection{Hidden Convexity}\label{subsec: hiddenconvex}
Consider a QCQP defined as follows:
	\begin{align}
	\label{QCQP}
	\min  & \left\{ x^THx + h^Tx \mid  \uvs{\ell} \leq x^TQx \leq r \right\},
	\end{align}
	where $x\in \Real^n, H \in \Real^{n \times n}, Q \in \Real^{n \times
		n}, h \in \Real^n, \uvs{\ell} \in \Real, r \in \Real.$ Suppose that \eqref{QCQP} is feasible and the
		two matrices $H$ and $Q$ can be simultaneously diagonalized.
		{Recall from~\cite{Ben-Tal1996} that $H$ and $Q$ can be simultaneously
			diagonalized if there exists a nonsingular matrix} $P \in \Real^{n \times n}$ such that 
	\begin{align*}
	P^THP = D \triangleq \diag(d_1,\hdots,d_n) \mbox{ and } 
	P^TQP = S \triangleq \diag(s_1,\hdots,s_n).
	\end{align*}
	Let $d \triangleq (d_1; \hdots; d_n), s \triangleq (s_1; \hdots; s_n) $.
	Then, by {utilizing a transformation $x = Py$ and  $c = P^Th$},
	\eqref{QCQP} can be written as follows:
	\begin{align}
	\label{diagonalQCQP}
	\min \left\{ y^T D y + c^Ty \mid  \uvs{\ell} \leq y^T S y \leq r \right\}, 
	\end{align}
	Further, consider the following convex program:
	\begin{align}
	\label{convexreform}
	\min_{w_1,\hdots,w_n} \left\{ \sum_{i=1}^n (d_i w_i - |c_i| \sqrt{w_i}) \mid \ell \leq \sum_{i=1}^n s_iw_i \leq r, \  w \geq 0\right\},
	\end{align}
	which is defined using $c,d,\uvs{\ell},$ and $r$. In fact, problems \eqref{diagonalQCQP} and
	\eqref{convexreform} are equivalent in the following sense: (i) if
	one is  {unbounded}, so is the other; (ii) if both are finite, the
	infimum of \eqref{diagonalQCQP} is equal to the the infimum of
	\eqref{convexreform}; (iii) {the optimal solution $y^*$ of
		\eqref{diagonalQCQP} can be constructed from {the} optimal solution $w^*$ of \eqref{convexreform} through the following equations:}
	\begin{align}
	y^*_j = -\sgn(c_j)\sqrt{w^*_j}, \quad j = 1,\hdots, n
	\end{align}
	where $\sign(c_j) = 1$ if $c_j \geq 0$ and $\sign(c_j) = -1$ if $c_j
	< 0$. Through the above arguments, a global minimizer of the
	solution to nonconvex program
	\eqref{diagonalQCQP} may be  obtained by the solution of a suitably defined convex program \eqref{convexreform}. Thus \eqref{diagonalQCQP} may have \textit{hidden convexity}. This property of QCQP is first discovered by Ben-Tal and Teboulle in 1996 \cite{Ben-Tal1996}.
\subsection{Proofs}
\begin{lemma}[{\bf Tightness of relaxation}] \label{Lm:TightRelax}
	\rm Consider the problem \eqref{mpcc} and suppose a global
	minimizer to this problem is denoted by $(x^{\pm},\xi).$ Let 
	$(\tilde x,\tilde x^{\pm},\tilde \xi)$ be defined as follows:
		\begin{align} \label{def-xtilde}
				\tilde \xi & \triangleq \xi, 
		\tilde x_i^+  \triangleq 
		\begin{cases} x_i^+-x_i^-, & \mbox{ if } x_i^+ \geq
		x_i^- \\
						 			   0, & \mbox{ otherwise}
							\end{cases}, 
						\tilde x_i^-  \triangleq 
						\begin{cases} 0, & \mbox{ if } x_i^+ \geq x_i^- \\
						 			  x_i^--x_i^+. & \mbox{ otherwise}
							\end{cases}, \\
							\notag
							& \forall i = 1,\hdots,n.
				\tilde x \triangleq \tilde x^+-\tilde x^-. 
	\end{align}
	Then $(\tilde x,\tilde x^{\pm},\tilde \xi)$ is a  
	a global minimizer
	of \eqref{mpcc-unrel}.
	\end{lemma}
	\begin{proof}
	Consider a solution $(x^{\pm},\xi)$ to \eqref{mpcc}. We
	first prove that the constructed solution $(\tilde x, \tilde x^{\pm},
			\tilde \xi)$ is feasible with respect to \eqref{mpcc-unrel}
		and then prove that it is optimal.
	
	\noindent {\bf Feasibility of $(\tilde x, \tilde x^{\pm}, \tilde \xi)$.} 
	By definition \eqref{def-xtilde}, we have that $\tilde x^{\pm} \geq 0$
	and  \\$\min(\tilde
			x_i^+,\tilde x_i^-) = 0$ for $i = 1, \hdots, n$. Consequently,
	   $\tilde x^+ \perp \tilde x^-$. Furthermore,  $\tilde x = \tilde x^+-\tilde x^- =
	   x^+ - x^-$ and $\tilde \xi = \xi$. This implies that $A \tilde x = A \tilde
	   x^+ - A\tilde x^- = A x^+ - Ax^- \geq b.$  Finally, it suffices to
	   show that $(\tilde x^++\tilde x^-)^T \tilde \xi = 0.$ By the
	   feasibility of $(x^{\pm},\xi)$ with respect to
	\eqref{mpcc}, we have that
	\begin{align}
	\notag	0 & = \sum_{i=1}^n  (x_i^++x_i^-)\xi_i  = \sum_{i \in \Iscr_+}
				(\underbrace{x_i^+-x_i^-}_{\triangleq \tilde x_i^+} + 2x_i^-)\xi_i +
				\sum_{i \in \Iscr_-}  (2x_i^++\underbrace{x_i^--x_i^+}_{\triangleq \tilde
						x_i^-})\xi_i\\
		\label{eq1}	& = \sum_{i \in \Iscr_+}
				(\tilde x_i^+ + \underbrace{\tilde x_i^-}_{=0} + 2x_i^-)\xi_i +
				\sum_{i \in \Iscr_-}  ( \underbrace{\tilde x_i^+}_{=0}+\tilde
						x_i^- + 2x_i^+ )\xi_i,
	\end{align}
	where $\Iscr_+ \triangleq  \{ i : x_i^+ \geq x_i^- \}$ and $\Iscr_- \triangleq \{ i : x_i^+ < x_i^- \}$. Since, \eqref{eq1}
	can be expressed as follows:
	\begin{align} \notag 
	0 &=  \sum_{i \in \Iscr_+}
				(\tilde x_i^+ + \tilde x_i^- + 2x_i^-)\xi_i +
				\sum_{i \in \Iscr_-}  (\tilde x_i^++\tilde
						x_i^- + 2x_i^+)\xi_i \\
			&	= (\tilde x^+ + \tilde x^-)^T \xi + 
	\sum_{i \in \Iscr_+} 2x_i^- \xi_i + \sum_{i \in \Iscr_-} 2x_i^+ \xi_i,
	\end{align}
	 and $x^{\pm}, \tilde x^{\pm}, \xi \geq 0$ implying that  
	\begin{align*}  
	&	  (\tilde x^+ + \tilde x^-)^T \xi  \geq 0,   
	\sum_{i \in \Iscr_+} 2x_i^- \xi_i \geq 0,  \sum_{i \in \Iscr_-}
	2x_i^- \xi_i \geq 0 \\
 &		\implies  
	 (\tilde x^+ + \tilde x^-)^T \xi  = 0,  
	 \overset{\xi = \tilde{\xi}}{\implies}
	 (\tilde x^+ + \tilde x^-)^T \tilde \xi  = 0.
	\end{align*}
	
	\noindent {\bf Optimality of $(\tilde x, \tilde x^{\pm}, \tilde \xi)$.} We observe
	that the $f(\tilde x) + \gamma \sum_{i=1}^n (1-\tilde \xi_i) = f(x^+-x^-) +
	\gamma \sum_{i=1}^n (1-\xi_i)$ for $\tilde x = x^+-x^-$ and $\tilde \xi = \xi$. But
	since \eqref{mpcc} is a relaxation of \eqref{mpcc-unrel}, the
	optimality of $(\tilde x, \tilde x^{\pm}, \tilde \xi)$ follows from
	\uvs{feasibility of} $(\tilde x, \tilde x^{\pm}, \tilde \xi)$ with
	respect to the tightened optimization problem with an identical
	objective value. \qed
\end{proof}

\begin{lemma}\label{lm: SimulDiag}
\rm Given $H \triangleq \pmat{M + \frac{\rho + \mu }{2}I_n & -M & \\ -M & M + \frac{\rho + \mu}{2}I_n & \\ & &\frac{\rho + \mu}{2}I_n}, \tilde{Q}  \triangleq \pmat{&&I_n\\&&I_n\\I_n&I_n&}$, and $G \triangleq \pmat{ \frac{1}{2} I_n & \frac{\sqrt{2}}{2} I_n &
		\frac{1}{2}I_n \\  \frac{1}{2}I_n & -\frac{\sqrt{2}}{2} I_n &
			\frac{1}{2}I_n \\ -\frac{\sqrt{2}}{2} I_n & &
			\frac{\sqrt{2}}{2} I_n} \pmat{I_n&&\\&V&\\&&I_n}$, where $V$ is orthogonal such that $S \triangleq V^TMV$ is diagonal. Then G is an orthogonal matrix, and $H, \tilde{Q}$ can be simultaneously diagonalized through $G$.
\end{lemma}
\begin{proof}
\begin{align*}
& G^TG \\
& = \pmat{I_n&&\\&V^T&\\&&I_n} \pmat{ \frac{1}{2} I_n & \frac{1}{2} I_n &
		-\frac{\sqrt{2}}{2}I_n \\  \frac{\sqrt{2}}{2}I_n & -\frac{\sqrt{2}}{2} I_n &
			 \\  \frac{1}{2} I_n & \frac{1}{2} I_n &
			\frac{\sqrt{2}}{2} I_n} \pmat{ \frac{1}{2} I_n & \frac{\sqrt{2}}{2} I_n &
		\frac{1}{2}I_n \\  \frac{1}{2}I_n & -\frac{\sqrt{2}}{2} I_n &
			\frac{1}{2}I_n \\ -\frac{\sqrt{2}}{2} I_n & &
			\frac{\sqrt{2}}{2} I_n} \pmat{I_n&&\\&V&\\&&I_n} \\
& = \pmat{I_n&&\\&V^T&\\&&I_n} \pmat{I_n&&\\&I_n&\\&&I_n} \pmat{I_n&&\\&V&\\&&I_n} = \pmat{I_n&&\\&V^TV&\\&&I_n} = I_{3n}
\end{align*}
Therefore, $G$ is orthogonal. Note that
{\footnotesize
\begin{align*}
& \pmat{ \frac{1}{2} I_n & \frac{1}{2} I_n &
		-\frac{\sqrt{2}}{2}I_n \\  \frac{\sqrt{2}}{2}I_n & -\frac{\sqrt{2}}{2} I_n &
			 \\  \frac{1}{2} I_n & \frac{1}{2} I_n &
			\frac{\sqrt{2}}{2} I_n}
\pmat{M + \frac{\rho + \mu}{2}I_n & -M & \\ -M & M + \frac{\rho + \mu}{2}I_n & \\ & &\frac{\rho + \mu}{2}I_n} 
\pmat{ \frac{1}{2} I_n & \frac{\sqrt{2}}{2} I_n &
		\frac{1}{2}I_n \\  \frac{1}{2}I_n & -\frac{\sqrt{2}}{2} I_n &
			\frac{1}{2}I_n \\ -\frac{\sqrt{2}}{2} I_n & &
			\frac{\sqrt{2}}{2} I_n} \\
& = \pmat{ \frac{1}{2} I_n & \frac{1}{2} I_n &
		-\frac{\sqrt{2}}{2}I_n \\  \frac{\sqrt{2}}{2}I_n & -\frac{\sqrt{2}}{2} I_n &
			 \\  \frac{1}{2} I_n & \frac{1}{2} I_n &
			\frac{\sqrt{2}}{2} I_n} \pmat{ \frac{\rho + \mu}{4} I_n & \sqrt{2}M + \frac{\sqrt{2}}{4} (\rho + \mu) I_n & \frac{\rho + \mu}{4}  I_n \\ \frac{\rho + \mu}{4}I_n & -\sqrt{2} M - \frac{\sqrt{2}}{4} (\rho + \mu)  I_n & \frac{\rho + \mu}{4}I_n \\ - \frac{\sqrt{2}}{4} (\rho + \mu) I_n &  & \frac{\sqrt{2}}{4} (\rho + \mu) I_n }  \\
& = \pmat{\frac{\rho + \mu}{2}I_n &&\\ & 2M + \frac{\rho + \mu}{2}I_n & \\ && \frac{\rho + \mu}{2}I_n}.
\end{align*}
}
This fact implies:
\begin{align*} 
& G^THG \\
& = \pmat{I_n&&\\&V^T&\\&&I_n} \pmat{\frac{\rho + \mu}{2}I_n &&\\ & 2M + \frac{\rho + \mu}{2}I_n & \\ && \frac{\rho + \mu}{2}I_n} \pmat{I_n&&\\&V&\\&&I_n} \\
& =  \pmat{\frac{\rho + \mu}{2}I_n &&\\ & 2 V^T M V + \frac{\rho + \mu}{2} V^TV & \\ && \frac{\rho + \mu}{2}I_n} = \pmat{\frac{\rho + \mu}{2}I_n &&\\ & 2 S + \frac{\rho + \mu}{2} I_n & \\ && \frac{\rho + \mu}{2}I_n}
\end{align*}
Meanwhile,
\begin{align*}
& G^T\tilde{Q}G \\
& = \pmat{I_n&&\\&V^T&\\&&I_n} \pmat{ \frac{1}{2} I_n & \frac{1}{2} I_n &
		-\frac{\sqrt{2}}{2}I_n \\  \frac{\sqrt{2}}{2}I_n & -\frac{\sqrt{2}}{2} I_n &
			 \\  \frac{1}{2} I_n & \frac{1}{2} I_n &
			\frac{\sqrt{2}}{2} I_n} \pmat{&&I_n\\&&I_n\\I_n&I_n&} \cdot \\ & \quad \pmat{ \frac{1}{2} I_n & \frac{\sqrt{2}}{2} I_n &
		\frac{1}{2}I_n \\  \frac{1}{2}I_n & -\frac{\sqrt{2}}{2} I_n &
			\frac{1}{2}I_n \\ -\frac{\sqrt{2}}{2} I_n & &
			\frac{\sqrt{2}}{2} I_n} \pmat{I_n&&\\&V&\\&&I_n} \\
& = \pmat{I_n&&\\&V^T&\\&&I_n} \pmat{ \frac{1}{2} I_n & \frac{1}{2} I_n &
		-\frac{\sqrt{2}}{2}I_n \\  \frac{\sqrt{2}}{2}I_n & -\frac{\sqrt{2}}{2} I_n &
			 \\  \frac{1}{2} I_n & \frac{1}{2} I_n &
			\frac{\sqrt{2}}{2} I_n} \pmat{-\frac{\sqrt{2}}{2}I_n & 0 & \frac{\sqrt{2}}{2}I_n \\ -\frac{\sqrt{2}}{2}I_n & 0 & \frac{\sqrt{2}}{2}I_n \\ I_n & 0 &I_n} \pmat{I_n&&\\&V&\\&&I_n} \\
& = \pmat{I_n&&\\&V^T&\\&&I_n} \pmat{-\sqrt{2}I_n &  & \\  & 0 & \\  &  & \sqrt{2} I_n} \pmat{I_n&&\\&V&\\&&I_n} = \pmat{-\sqrt{2}I_n &  & \\  & 0 & \\  &  & \sqrt{2} I_n}\\
\end{align*}
Thus, $H$ and $\tilde Q$ can be simultaneously diagonalized through $G$. \qed
\end{proof}

\subsection{Miscellaneous}

\begin{definition}[{\bf Kurdyka-{\L}ojasiewicz (K{\L})
	property}]\label{KL}\rm 
	 A proper lower\\semi-continuous function $\Lscr: \bR^{N} \rightarrow \bR \cup \{+\infty\}$ has {the} K{\L} property at $\bar
	x \in \mbox{dom}(\partial \Lscr)$, if there exists $\eta \in (0,+\infty)$, a neighborhood $U$ of $\bar x$, and a continuous concave function
	$\phi: [0,\eta) \rightarrow \bR_+$ such that the following hold:  (i) $ \phi(0)= 0$, and $\phi$ is continuously differentiable on $(0,\eta)$. For all $s \in (0,\eta)$, $\phi'(s) > 0$; (ii) For all $x$ in $U \cap \{ x \in \bR^{N}: \Lscr(\bar x) < \Lscr(x) <
	\Lscr(\bar x) + \eta ]$, the Kurdyka-{\L}ojasiewicz (K{\L}) inequality
	holds:
$
\phi'(\Lscr(x) - \Lscr(\bar x)) {\rm dist} (0,\partial \Lscr(x)) \geq 1. 
$
\end{definition}

\begin{definition}[Semialgebraic function]\label{def: semiAb}
Recall that a semialgebraic set $S \subseteq \bR^n$ is defined as:
\begin{align*}
S \triangleq \{ x \in \bR^n: p_i(x) = 0, q_i(x) < 0, i = 1,\hdots,m \},
\end{align*}
where $p_i$ and $q_i$ are real polynomial functions. A function $F: \bR^n \rightarrow \bR \cup \{ +\infty \}$ is a semialgebraic function if and only if its graph $\{ (x;y) \in \bR^n \times \bR: y = F(x) \}$ is a semialgebraic subset in $\bR^{n+1}$. 
\end{definition}

\begin{remark}
Semialgebraic function has the following properties: (i). It satisfies K{\L} property with $\phi(s) = cs^{1-\theta}$ for some $\theta \in [0,1) \cap \mathbb{Q}$ and $c > 0$. (ii). Finite sums and products of semialgebraic functions are semialgebraic. See \cite[Section 4.3]{10.2307/40801236} for more details.
\end{remark}

\begin{lemma}[\bf Theorem 10~\cite{BurkeLagrange}]\label{NormalConeCharacterization}
\rm In $\bR^{n_1}$, let $C = \{ x \in X \mid F(x) \in D \}$,  for closed convex sets $X \subset \bR^{n_1}, D \subset \bR^{n_2}$, and a $\mathcal{C}^1$ mapping $F: \bR^{n_1} \rightarrow \bR^{n_2}$, written componentwise as $F(x) = (f_1(x); \hdots ; f_{n_2}(x))$. Suppose the following constraint qualification is satisfied at a point $\bar x \in C$:
\begin{align*}
\sum_{j=1}^{n_2} y_j \nabla f_j(\bar x) 
 + z = 0, y = (y_1; \hdots; y_{n_2}) \in \sN_D(F(\bar x)),  z \in \sN_X(\bar x)  \\ \implies y = {\bf 0}, z = 0.
\end{align*}
Then the normal cone $\sN_C(\bar x)$ consists of all vectors $v$ of the form
\begin{align*}
v = y_1 \nabla f_1(\bar x) + \hdots + y_{n_2} \nabla f_{n_2}(\bar x) + z \mbox{ with } y = (y_1;\hdots;y_{n_2}) \in \sN_D(F(\bar x)),\\
z \in \sN_X(\bar x).
\end{align*}
Note: When $X = \bR^{n_1}$, the normal cone $\sN_X(\bar x) = \{0\}$, so the $z$ terms here drop out. When $D$ is a singleton, $\sN_D(F(\bar x)) = \bR^{n_2}$.
\end{lemma}

\subsection{Convergence analysis of ADMM$_{\rm cf}$}\label{subsec: Anal.ADMMcf}
	\uvs{We now \vvs{analyze convergence} of (ADMM$_{\rm cf}$)  under the following update rule of $\rho_k$:}
\begin{align*}
& \mbox{If }  (\rho_k - \delta) \| y_{k+1} - y_k \| < \sqrt{2} \| \lambda_{k+1} - \lambda_k \| \mbox{ and } \rho_k \leq \rho_{\max},
\mbox{ then } \rho_{k+1} := \delta_{\rho} \rho_k; \\
& \mbox{else } \rho_{k+1}:= \rho_k.
\end{align*}
Further, we make the following assumption. 
\begin{assumption}\label{upperbound for rho}
The penalty parameter sequence $\{\rho_k\}$ in (ADMM$_{cf}$) \yxII{never exceeds the prescribed upper bound}, \ie $\rho_k \leq \rho_{\max}$, $\forall k \geq 0$.
\end{assumption}
	\begin{proposition} [{\bf Limit points of (ADMM$_{\rm cf}$) are KKT Points}]\label{limit KKT point}\rm
	{Consider problem \eqref{mpcc} with {$f_Q(x) = x^TMx + d^Tx$. Suppose
	(ADMM$_{\rm cf}$) {generates a sequence} $\{w_k
		\triangleq (x_k^+;x_k^-;\xi_k), y_k, {\lambda}_k\}$. \yxII{Assume that this sequence converges to a limit point denoted by $(\bar w, \bar y, {\bar \lambda})$.} Then, we may claim the following}: 
(a) The point $\bar w = (\bar x^+; \bar x^-; \bar \xi)$ satisfies first-order KKT conditions of \eqref{mpcc}.
(b) {If $f$ is convex,} then $\bar w$ is a local minimum of \eqref{mpcc}.
}
	\end{proposition}   
	\begin{proof}
	{\bf (a)}. 
\yxII{By the update rule of (ADMM$_{\rm cf}$), $\exists K > 0, s.t. \  \rho_k \equiv \rho$, $\forall k \geq K$.} Consequently, suppose $\rho_k \equiv \rho$ for all
$k$ without loss of generality. \uvs{(Otherwise, we \uvs{may} initialize using $w_K, y_K, \lambda_K, \rho_K$)} By (Update-1) at iteration $k+1$, we have
that
	$$w_{k+1} \in \mbox{arg}\hspace{-0.28in}\min_{(x^++x^-)^T\xi = 0}
	\left[f_Q(x^+ - x^-) - \gamma e^T\xi +
	\frac{\rho}{2}\|w - y_k+\lambda_k/ \rho \|^2\right].$$ Let $z_{k+1} \triangleq
	G^Tw_{k+1}$ and \uvs{$q_k \triangleq G^T \left( \pmat{d; -d;-\gamma e} + \lambda_k - \rho y_k \right),$}
	 where $G$ is defined in \eqref{G}. Denote $z_{k+1} \triangleq (z_{k+1,1};z_{k+1,2};z_{k+1,3}), z_{k+1,1}, z_{k+1,2}, z_{k+1,3} \in \bR^n$. From \eqref{closedform0}, we have $\exists u_k, v_k \in \bR^n$ such that
	 \begin{align} \notag
	z_{k+1,1} & = \begin{cases} 
			\frac{-(\|q_{k,1}\|+\|q_{k,3}\|)q_{k,1}}{2\rho\|q_{k,1}\|},
			& \hspace{-0.1in} \|q_{k,1}\| > 0 \\
			\frac{\|q_{k,3}\|}{2\rho}u_k,\| u_k \|=1, & \hspace{-0.1in} \|q_{k,1}\| =  0
			\end{cases}, \\
			\notag
	z_{k+1,3} & = \begin{cases} 
			\frac{-(\|q_{k,1}\|+\|q_{k,3}\|)q_{k,3}}{2\rho\|q_{k,3}\|},
			& \hspace{-0.1in}\|q_{k,3}\| > 0 \\
			\frac{\|q_{k,1}\|}{2\rho}v_k,\|v_k\|=1, & \hspace{-0.1in}\|q_{k,3}\| =  0
			\end{cases}, \\
			 \label{closedform}
	 (z_{k+1,2})_i  & = -(q_{k,2})_i/(\rho + 4s_i), \forall i = 1,\hdots,n,
	\end{align}
	where $q_k \triangleq (q_{k,1}; q_{k,2}; q_{k,3}), q_{k,1}, q_{k,2}, q_{k,3} \in \bR^n$. Since $w_{k+1} \rightarrow \bar w$, $y_k \rightarrow \bar y$, $\lambda_k
	 \rightarrow \bar \lambda$ as $k \to \infty$, $z_{k+1} \rightarrow \bar z = G^T
	 \bar w$ and as $k \rightarrow +\infty$, $q_k \rightarrow \bar q \triangleq G^T \left( \pmat{d; -d;-\gamma e} + \bar \lambda - \rho \bar y \right).$
We proceed to show that $\bar z$ and $\bar
	 q$ also {satisfy} the following: $\exists \bar u, \bar v \in 
	\bR^n$ \uvs{such that } 
	 \begin{align}
\label{limits property 1-Prop: limitKKT}
	\bar z_1 & = \begin{cases} 
			\frac{-(\| \bar q_1 \|+\| \bar q_3 \|) \bar q_1 }{2\rho\|\bar q_1 \|},
			& \hspace{-0.1in} \| \bar q_1 \| > 0 \\
			\frac{\| \bar q_3 \|}{2\rho} \bar u,\| \bar u \|=1, & \hspace{-0.1in} \| \bar q_1 \| =  0
			\end{cases}; 
	\bar z_3  = \begin{cases} 
			\frac{-(\| \bar q_1 \|+\| \bar q_3 \|) \bar q_3 }{2\rho\| \bar q_3 \|},
			& \| \bar q_3 \| > 0 \\
			\frac{\| \bar q_1 \|}{2\rho} \bar v,\| \bar v \|=1, & \| \bar q_3 \| =  0
			\end{cases}
\end{align}
\vspace{-0.2in}
\begin{align}
\label{limits property 3-Prop: limitKKT}
	(\bar z_2)_i &= -(\bar q_2 )_i/(\rho + 4s_i), \quad \forall i = 1,\hdots,n,
	\end{align}	
where $\bar z = (\bar z_1; \bar z_2; \bar z_3)$ and $\bar q = (\bar q_1; \bar
q_2; \bar q_3)$. \yxI{First, we prove that $\bar z_1$ and $\bar q$ satisfy \eqref{limits property 1-Prop: limitKKT}.}

\noindent (i) \uvs{Case 1. $\| \bar q_1
\| > 0$.} Then $\exists K$ s.t. $\forall k \geq K, \| q_{k,1} \| > 0$,
$$z_{k+1,1} = \frac{-(\|q_{k,1}\|+\|q_{k,3}\|)q_{k,1}}{2\rho\|q_{k,1}\|}.$$
Therefore, $$\bar z_1 = \lim_{k \rightarrow +\infty} z_{k+1,1} = \lim_{k
\rightarrow +\infty} \frac{-(\|q_{k,1}\|+\|q_{k,3}\|)q_{k,1}}{2\rho\|q_{k,1}\|}
= \frac{-(\| \bar q_1 \|+\| \bar q_3 \|) \bar q_1 }{2\rho\|\bar q_1 \|}.$$

\noindent (ii) Case 2.  
$\| \bar q_1 \| = 0$. Then 
\begin{align*}
& \| \bar z_1 \| = \lim_{k \rightarrow +\infty} \|
z_{k+1,1} \| = \| \bar q_3 \|/(2\rho) \implies \exists \bar u, \| \bar u \|
= 1, \ s.t.,\ \bar z_1 = \frac{\| \bar q_3 \|}{2\rho} \bar u.
\end{align*}
Note that expression of $\bar z_3$  can be proven similarly and \uvs{we omit proof of 
\eqref{limits property 3-Prop: limitKKT}.}  Therefore, 
{
	 \begin{align}\bar w = G \bar z \in
	 \mbox{arg}\hspace{-0.28in}\min_{(x^++x^-)^T\xi = 0} \left[ f_Q(x^+ - x^-) - \gamma e^T\xi +
	 \frac{\rho}{2}\|(x^+;x^-;\xi)- \bar y + \bar \lambda / \rho \|^2 \right].\label{upd1}
	 \end{align}
	In particular, it follows
	 that $(\bar x^+ + \bar x^-)^T \bar \xi = 0$.} Next, we consider
	whether such a limit point satisfies the first-order KKT conditions
		of \eqref{upd1} by examining two cases:

\noindent(i) Suppose $( \bar \xi;  \bar \xi ; \bar x^+ + \bar x^-)
		\neq 0$. Then the linear independence constraint qualification
		(LICQ) holds at $(\bar x^+, \bar x^-, \bar \xi)$. Consequently, there exists a scalar $\mu$ such that
	\begin{align}\label{KKT1}
	\pmat{\nabla f_Q(\bar x^+ - \bar x^-) \\ -\nabla f_Q(\bar x^+ - \bar
			x^-) \\ -\gamma e} + \rho(\bar w - \bar y + \bar \lambda / \rho) +
		\mu \pmat{\bar \xi \\ \bar \xi \\ \bar x^+ +  \bar x^-} =0.
	\end{align}
\noindent	(ii)  {Suppose $(\bar \xi, \bar
			\xi, \bar x^+ + \bar x^-) = 0$, implying $\bar \xi = 0$ and $x^+ + x^- = 0$. Since $(\bar x^+; \bar x^-;
				\bar \xi)$ is \us{a} global optimizer of \eqref{upd1},
				when we fix $\yyxx{\xi \equiv \bar{\xi} = 0}$, the following must hold:}
	\begin{align*}
&	\pmat{ \bar x^+ \\ \bar x^- } \in \mbox{arg}\hspace{-0.051in}\min_{x^+,x^-}  f_Q(x^+ - x^-) + \frac{\rho}{2} \left( \left\| x^+ - \bar y_1
				+ \frac{\bar \lambda_1 }{ \rho } \right\|^2 + \left\| x^- - \bar y_2 + \frac{ \bar \lambda_2 }{ \rho }
				\right\|^2 \right) \\	 
	& \implies 0 =  \pmat{\nabla f_Q(\bar x^+ - \bar x^-) \\ -\nabla
		 f_Q(\bar x^+ - \bar x^-) } + \rho \pmat{ \bar x^+  - \bar y_1 +
			 \bar \lambda_1 / \rho \\ \bar x^- - \bar y_2 + \bar \lambda_2 / \rho}. 
\end{align*}
{If we fix $x^+ \equiv \bar x^+, x^- \equiv \bar x^-$, then }
\begin{align*}
	\bar \xi & \in \mbox{arg}\hspace{-0.051in}\min_{\xi \in \Real^n} \left[
	-\gamma e^T\xi + \frac{\rho}{2} \| \xi - \bar y_3 + \bar \lambda_3 / \rho
			\|^2 \right] 
	 \implies 0  =  -\gamma e + \rho(\bar \xi - \bar y_3 + \bar \lambda_3/ \rho),
	\end{align*}
	where $\bar y = (\bar y_1;\bar y_2;\bar y_3),\bar y_1,\bar y_2,\bar
	y_3 \in \Real^n$, $\bar \lambda = (\bar \lambda_1;\bar \lambda_2;\bar
		\lambda_3),\bar  \lambda_1,\bar \lambda_2,\bar \lambda_3 \in
	\Real^n$. {Thus, \eqref{KKT1} holds for every $\mu \in \bR$}.  {Next, in (Update-2), we need to \uvs{compute $y_{k+1}$}, where 
$ 
y_{k+1}  =  \mbox{arg}\hspace{-0.03in}\min_{y \in Z_2} \quad g(y^+ - y^-) + \frac{\rho}{2} \| y - w_{k+1} - \lambda_k/\rho \|^2.
$
Note that the following first order condition holds because it is a convex program:
\begin{align*}
 ( [ \nabla g(y_{k+1,1} - y_{k+1,2}) ; -\nabla g(y_{k+1,1} - y_{k+1,2}) ; {\bf 0}_{n} ] + \rho(y_{k+1} - w_{k+1} - \lambda_k/ \rho ) )^T  \\
 (y -  y_{k+1}) \geq 0, \\
\forall y \in Z_2, y_{k+1} = (y_{k+1,1};y_{k+1,2};y_{k+1,3}), y_{k+1,i} \in \bR^n, i = 1,2,3.
\end{align*}
By continuity of $\nabla g(\bullet)$, since $w_{k+1} \rightarrow \bar w, y_{k+1} \rightarrow \bar y, \lambda_k \rightarrow \bar \lambda$, we have that 
\begin{align*}
( [ \nabla g(\bar y_{1} - \bar y_{2}) ; -\nabla g(\bar y_{1} - \bar y_{2}) ; {\bf 0}_{n \times 1} ] + \rho(\bar y - \bar w - \bar \lambda / \rho ) )^T(y -  \bar y) \geq 0,
\forall y \in Z_2, \\
\implies \bar y \in \mbox{argmin}_{y \in Z_2} \left[ g(y^+ - y^-) + \frac{\rho}{2}\| y - \bar w - \bar \lambda / \rho \|^2 \right].
\end{align*}
}
Thus, by the definition of $Z_2$ in \eqref{def-Z1Z2}, $\exists \beta_1,\beta_2,\beta_3,\beta_4 \in \Real^n, \pi \in \Real^m$ such that
	\begin{align}\label{KKT2}
	& \pmat{ \nabla g(\bar y_1 - \bar y_2) \\ -\nabla g(\bar y_1 - \bar y_2) \\ {\bf 0}_{n \times 1} } + \rho(\bar y - \bar w - \bar \lambda / \rho) + \pmat{ -\beta_1  - A^T\pi \\ -\beta_2 + A^T\pi \\  \beta_4 -  \beta_3} = 0,  \\
	& \notag 0 \leq \beta_i \perp \bar y_i \geq 0,   i = 1, 2, 3, \ 
	0 \leq e - \bar y_3  \perp \beta_4 \geq 0, \ 
	0 \leq A(\bar y_1- \bar y_2) - b  \perp \pi \geq 0.
	\end{align}
	Note that
	$\lambda_k \rightarrow \bar \lambda$ implies that $ w_{k+1} - y_{k+1} = (\lambda_{k+1} - \lambda_k) / \rho
	\rightarrow 0$, which further implies that $\bar w = \bar y$. By
	combining \eqref{KKT1} and \eqref{KKT2}, letting $\bar y = \bar w$,
	and by adding $(\bar x^+ + \bar x^-)^T \bar \xi = 0$, we have exactly the KKT conditions (\eqref{KKT-F}) at $(\bar x^+; \bar x^-; \bar \xi)$ for \eqref{mpcc}.

 {{\bf (b).} {If $f$ is convex, then by Theorem~\ref{SOC}}, $\bar w = (\bar
		x^+; \bar x^-; \bar \xi)$ is a local minimum of \eqref{mpcc}.} \qed
\end{proof}

Next, suppose $h(w) \triangleq f_Q(x^+-x^-) + \gamma \sum_{i=1}^n (1-\xi_i)$, $p(y) \triangleq g(y^+ - y^-)$. Define:
\begin{align} \notag
{\mathcal L}(w,y,\lambda,\rho) & \triangleq 
 h(w) + p(y) + \lambda^T(w-y) + {\rho \over 2} \|w-y\|^2 \\
\label{def-scrH}
\notag
\scrH(w,y,\lambda) & \triangleq h(w) + \bk1_{Z_1}(w) + p(y) + \bk1_{Z_2}(y) + \lambda^T(w - y)  \\
\hspace{0.01in}\scrH_\rho (w,y,\lambda) & \triangleq h(w) + \bk1_{Z_1}(w) + p(y) + \bk1_{Z_2}(y) \\
\notag
& \quad + \lambda^T(w - y) + \frac{\rho}{2}\| w-y \|^2.
\end{align}
Then the updates of (ADMM$_{\rm cf}$) can be rewritten as follows:
\begin{align}
\label{update1}
	w_{k+1} &:= \mbox{arg}\hspace{-0.05in}\min_{w \in Z_1} \ {\mathcal L}(w,y_k,\lambda_k,\rho_k) = \mbox{arg}\hspace{-0.02in}\min_w \ \scrH_{\rho_k}(w,y_k,\lambda_k), \\
\label{update2}
	y_{k+1} &:= \mbox{arg}\hspace{-0.05in}\min_{y \in Z_2} \ {\mathcal L}(w_{k+1},y,\lambda_k,\rho_k) = \mbox{arg}\hspace{-0.02in}\min_y \ \scrH_{\rho_k}(w_{k+1},y,\lambda_k), \\
	\notag
	\lambda_{k+1} & := \lambda_k + \rho_k (w_{k+1} - y_{k+1}).
\end{align}
\uvs{Deriving convergence statements} of Algorithm~\ref{admm} necessitates the following lemma.
\begin{lemma}\label{lemma:LagConv} \rm 
Consider the sequence
$\{w_k,y_k,\lambda_k,\rho_k\}$ generated by (ADMM$_{\rm cf}$).\\ Then for all $k \geq 0$, where $\Delta \scrL_k \triangleq \scrL(w_{k+1},y_{k+1},\lambda_{k+1},\rho_{k+1}) - \scrL(w_k, y_k, \lambda_k, \rho_k),$
\begin{align} 
\Delta \scrL_k   \leq \left( \frac{1}{\rho_k} + \frac{\rho_{k+1} - \rho_k}{2 \rho_k^2} \right) \| \lambda_{k+1} - \lambda_k \|^2  
 - \frac{\rho_k}{2} \| y_{k+1} - y_k \|^2.\label{LagConv_ineq0}
\end{align}
\end{lemma}
\begin{proof} From the definition of augmented Lagrangian function, 
\begin{align}
\notag
& \scrL(w_{k+1},y_{k+1},\lambda_k,\rho_k) - \scrL(w_{k+1},y_k,\lambda_k,\rho_k) \\
\notag
& \leq - \nabla_y \Lscr(w_{k+1},y_{k+1},\lambda_k,\rho_k)^T ( y_k - y_{k+1} ) - \frac{\rho_k}{2} \| y_{k+1} - y_k \|^2 \\
\label{LagConv-ineq1}
& \leq - \frac{\rho_k}{2} \| y_{k+1} - y_k \|^2,
\end{align}
where the first inequality \vvs{follows} \vvs{since $\scrL(w_{k+1},y,\lambda_k,\rho_k)$ is $\rho_k$-strongly convex} {in $y$} {with constant} $\rho_k$, while the second inequality may be derived from the optimality {conditions} of update \eqref{update2} {whereby} $\nabla_y \Lscr(w_{k+1},y_{k+1},\lambda_k,\rho_k)^T ( y - y_{k+1} ) \geq 0, \forall y \in Z_2$. {Since $w_{k+1}$ is a minimizer associated with \eqref{update1}}, we have that 
\begin{align}
\label{LagConv-ineq2}
\scrL(w_{k+1},y_k,\lambda_k,\rho_k) - \scrL(w_k,y_k,\lambda_k,\rho_k) 
& \leq 0. 
\end{align}
\yxIII{By invoking the definition of the augmented Lagrangian function, and utilizing the update rule for $\lambda_{k+1}$, i.e. $\lambda_{k+1} = \lambda_k + \rho_k(w_{k+1}-y_{k+1})$,}
\begin{align}
\notag
& \scrL(w_{k+1},y_{k+1},\lambda_{k+1},\rho_k) - \scrL(w_{k+1},y_{k+1},\lambda_k,\rho_k) \\
\notag
& = (\lambda_{k+1} - \lambda_k)^T ( w_{k+1} - y_{k+1} ) \\
\label{LagConv-ineq3}
& = \| \lambda_{k+1} - \lambda_k \|^2 / \rho_k \\
\notag
& \scrL(w_{k+1},y_{k+1},\lambda_{k+1},\rho_{k+1}) - \scrL(w_{k+1},y_{k+1},\lambda_{k+1},\rho_k)  \\
\notag
& = \frac{\rho_{k+1} - \rho_k}{2} \| w_{k+1} - y_{k+1} \|^2 \\
\label{LagConv-ineq4}
& = \frac{\rho_{k+1} - \rho_k}{2\rho_k^2} \| \lambda_{k+1} - \lambda_k \|^2.
\end{align}
By adding \eqref{LagConv-ineq1}, \eqref{LagConv-ineq2}, \eqref{LagConv-ineq3}, and \eqref{LagConv-ineq4}, we obtain the {required} result. \qed
\end{proof}

Proving convergence requires the
	Kurdyka-{\L}ojasiewicz property and a requirement on the multiplier sequence. 


\begin{assumption}\label{upperbound for multiplier}
$ \liminf_{k \to +\infty} \| \lambda_k \| < +\infty.$
\end{assumption}
\begin{theorem} \label{thm: convergence}
\rm Consider the sequence
$\{w_k,y_k,\lambda_k,\rho_k\}$ generated by (ADMM$_{\rm cf}$). 
Suppose {that Assumption~\ref{upperbound for rho},~\ref{upperbound for multiplier} hold and $\{ y_k \}$ is bounded}.
Then the following hold:

\noindent(i) {$\exists K_0 \in \mathbb{N}, s.t. \rho_k \equiv \rho, \forall k \geq K_0$. $\| \lambda_{k+1} - \lambda_k \| \leq C \| y_{k+1} - y_k \|, \forall k \geq K_0$, $C \triangleq \frac{\rho-\delta}{\sqrt{2}}$.}\\
(ii) $\{ \scrL(w_k,y_k,\lambda_k,\rho_k) \}_{k \geq K_0}$ is a non-increasing sequence satisfying
\begin{align}
\notag
& \scrL(w_{k+1},y_{k+1},\lambda_{k+1},\rho) - \scrL(w_k, y_k, \lambda_k, \rho) \\
\notag
& \leq \left( C^2 / \rho - \rho / 2 \right) \| y_{k+1} - y_k \|^2 \\
\label{LagDesc_ineq0}
& \leq - ( \delta / 2 ) \| y_{k+1} - y_k \|^2, \quad \forall k \geq K_0.
\end{align}

\noindent(iii)  {$\{ \scrL(w_k,y_k,\lambda_k,\rho_k) \}_{k \geq K_0}$ is \uvs{bounded from below}. Furthermore, $y_{k+1} - y_k \to 0$, $y_k - w_k \to 0$ as $k \to \infty$ and $\{w_k\}$ is a bounded sequence. Therefore, $\{ (w_k; y_k; \lambda_k) \}$ has a convergent subsequence \vvs{with limit point given by} $z^* \triangleq (w^*; y^*; \lambda^*)$.}

\noindent(iv) {Suppose $\scrH_{\rho}$ satisfies the K{\L} property at $z^*$.} Then $\sum_{k=0}^\infty \| y_{k+1} - y_{k} \| < \infty$. 

\noindent(v) {Suppose $\scrH_{\rho}$ satisfies the K{\L} property at $z^*$.} 
Then $(w_k; y_k; \lambda_k)$ converges to $z^*$, {and $z^*$ satisfying the} first-order KKT conditions of \eqref{mpcc}.  
\end{theorem}
\begin{proof}
\noindent{\bf (i)}. {By Assumption~\ref{upperbound for rho}, $\rho_k$ remains unchanged for sufficiently large $k$, so we denote $\rho, K_0$ such that $\rho_k \equiv \rho \leq \rho_{\max}, \forall k \geq K_0.$ Moreover, by the update rule in step 2 of Alg.~\ref{admm}, $\| \lambda_{k+1} - \lambda_k \| \leq \frac{\rho - \delta}{\sqrt{2}} \| y_{k+1} - y_k \|, \forall k \geq K_0.$ }

\noindent{\bf (ii)}. From Lemma~\ref{lemma:LagConv} and (i), for $\forall k \geq K_0, \rho_k = \rho$, and
\begin{align} \label{LagConv_ineq2}
\notag
&  \scrL(w_{k+1},y_{k+1},\lambda_{k+1},\rho_{k+1})  - \scrL(w_k, y_k, \lambda_k, \rho_k) \\
\notag
& \leq \frac{1}{\rho} \| \lambda_{k+1} - \lambda_k \|^2 - \frac{\rho}{2} \| y_{k+1} - y_k \|^2 \\
\notag 
& \leq \frac{(\rho - \delta)^2}{2 \rho} \| y_{k+1} - y_k \|^2 - \frac{\rho}{2} \| y_{k+1} - y_k \|^2 \\
\notag
& = \frac{\delta^2-2\delta\rho}{2\rho} \| y_{k+1} - y_k \|^2 \\
\notag
& \overset{(\delta < \rho)}{\leq} \left( -\delta + \frac{\delta}{2} \right) \| y_{k+1} - y_k \|^2 \\
& = -\frac{\delta}{2}  \| y_{k+1} - y_k \|^2.
\end{align}
Thus, $\{ \scrL(w_k,y_k,\lambda_k,\rho_k) \}_{k \geq K_0}$ is a non-increasing sequence. 

\noindent{\bf (iii)}. We first show that $\inf_{k \geq 0} \{ h(w_k) + p(y_k) + \frac{\rho}{2}\| w_k - y_k \|^2 \}$ is finite. Note that
\begin{align*}
& \quad h(w) - n\gamma + p(y) + \frac{\rho}{2}\| w - y \|^2 \\
& = f_Q(x^+ - x^-) - \gamma e^T\xi + p(y) + \frac{\rho}{2}\| w - y \|^2 \\
& = (x^+ - x^-)^TM(x^+ - x^-) + d^T(x^+ - x^-) - \gamma e^T\xi + p(y) + \frac{\rho}{2}\| w - y \|^2 \\
& = w^T H w + [(d;-d;-\gamma e) - \rho y ]^Tw + \frac{\rho}{2}\| y \|^2 + p(y) \\
& = \left\| w + \frac{1}{2} H^{-1} [(d;-d;-\gamma e) - \rho y ] \right\|_H^2
  + \frac{\rho}{2} \| y \|^2 + p(y) \\
 & - \frac{1}{4} \| (d;-d;-\gamma e) - \rho y \|_{H^{-1}}^2\\
& \geq \frac{\rho}{2} \| y \|^2 + p(y) - \frac{1}{4} \| (d;-d;-\gamma e) - \rho y \|_{H^{-1}}^2, 
\end{align*}
where $H \triangleq \pmat{M+\frac{\rho}{2}I & -M & \\ -M & M+\frac{\rho}{2}I & \\ & & \frac{\rho}{2}I }$ and $H \succ 0$ (Note that $\rho_0 I + 4M \succ 0$ leading to $\rho I + 4M
\succ 0$, further implying $H \succ 0$). 
Since $\{ y_k \}$ is bounded by assumption, and $p(y) = g(y^+ - y^-)$ is smooth,
\begin{align*}
& \inf_{k \geq 0} \left[ h(w_k) - n\gamma + p(y_k) + ( \rho / 2 ) \| w_k - y_k \|^2 \right] \\ 
& \geq \inf_{k \geq 0} [ ( \rho / 2 ) \| y_k \|^2 + p(y_k) - ( 1 / 4 ) \| (d;-d;-\gamma e) - \rho y_k \|_{H^{-1}}^2] > -\infty.
\end{align*}
If $ \bar L \triangleq \inf_{k \geq 0} \{ h(w_k) + p(y_k) + \frac{\rho}{2}\| w_k - y_k \|^2 \}$, then 
\begin{align*}
\scrL(w_k,y_k,\lambda_k,\rho_k) & \geq \bar L + \lambda_k^T(w_k - y_k)  \\
& = \bar L + \lambda_k^T(\lambda_k - \lambda_{k-1})/\rho \\
& = \bar L + \frac{1}{2\rho}( \| \lambda_k \|^2 - \| \lambda_{k-1} \|^2 + \| \lambda_k - \lambda_{k-1} \|^2 ),
\end{align*}
implying that 
\begin{align*} 
 \sum_{k=K_0}^N ( \scrL(w_k,y_k,\lambda_k,\rho_k) - \bar L ) 
\geq \frac{ \| \lambda_N \|^2 - \| \lambda_{K_0-1} \|^2 }{2\rho} \geq \frac{-\| \lambda_{K_0-1} \|^2}{2\rho}  {>} -\infty,
\end{align*} 
\yxII{for all $N \geq K_0$.} Since $\{ \scrL(w_k,y_k,\lambda_k,\rho_k) - \bar L \}_{k \geq K_0}$ is a non-increasing sequence from (ii), it's nonnegative. Otherwise, $
\lim_{N \rightarrow +\infty} \sum_{k=K_0}^N ( \scrL(w_k,y_k,\lambda_k,\rho_k) - \bar L ) = -\infty.
$
This is a contradiction. Therefore, $\{ \scrL(w_k,y_k,\lambda_k,\rho_k) \}_{k \geq K_0}$ is \uvs{bounded from below}. Consequently, $h_k  
\triangleq \scrH_\rho(w_k,y_k,\lambda_k)$ is a convergent sequence because $\scrH_\rho(w_k,y_k,\lambda_k) =  \scrL(w_k,y_k,\lambda_k,\rho) = \scrL(w_k,y_k,\lambda_k,\rho_k), \forall k \geq K_0.$
Without loss of generality, suppose $h_k \rightarrow 0$. Then, by summing up
\eqref{LagConv_ineq2} for $k \geq K_0$, we have 
 $\sum_{k=K_0}^\infty \| y_{k+1} - y_k \|^2 \leq \frac{h_{K_0}}{\delta/2} < \infty.$
It follows that $y_{k+1}-y_k \to 0$ as $k \to \infty$. From (i),   
we also have $\|\lambda_{k+1}-\lambda_k\| \rightarrow 0$ as $k \to \infty$. In other words, $\rho\|w_k - y_k\| \to 0$ as $k \to \infty$. But $\{y_k\}$ is a bounded sequence, implying that $\{w_k\}$ is a bounded sequence. By Assumption~\ref{upperbound for multiplier}, there exists a convergent subsequence of $\{ \lambda_k \}$. Therefore, there exists a subsequence of $\{w_k,y_k,\lambda_k\}$ converging to a point denoted by $\{w^*,y^*,\lambda^*\} \triangleq z^*$. \\
 {\bf (iv).} Next we prove $\|y_{k+1} - y_k\|$ is summable by using
{the K{\L}} inequality. {By assumption}, $\scrH_\rho$
{admits the} K{\L} property at $z^*$ and suppose the concave function $\psi$,
a {neighborhood} $U$, and \us{a scalar} $\eta > 0$ are associated with the
{K{\L}} property. {Further, suppose $B(z^*,r) \subseteq U$ and denote $z_k =
(w_k; y_k; \lambda_k)$. We know that $h_k \to 0$. If for some $k_0 \geq K_0$, $h_{k_0} = 0$, then by \eqref{LagDesc_ineq0}, $y_k = y_{k+1}, \forall k \geq k_0$, the proof is complete. Therefore, let $h_k > 0, \forall k \geq K_0$. Since a subsequence of $\{ z_k \}$ converges to $z^*$, and $h_k \to 0$, $\exists K \geq K_0 + 1$ such that:}
\begin{align}\label{Initial Ineq}
\notag
& \left( \frac{2C}{\rho} + C + 2 \right)\sqrt{\frac{h_{K-1}}{ \rho/2 - C^2/\rho}} + \left( \frac{C}{\rho} + \frac{C}{2} + 1 \right) \times \\ 
&  \left[ \frac{\psi(h_K)}{C_0} + \left[ \frac{\psi(h_K)}{C_0} \left( \frac{\psi(h_K)}{C_0} + 4\sqrt{\frac{h_{K-1}}{\rho/2 - C^2/\rho }} \right) \right]^{1/2} \right] 
 + \| z_K - z^* \| < r,
\end{align}
$ \uvs{\mbox{ where } h_K < \eta \mbox{ and } C_0 = \frac{\frac{\rho}{2} - \frac{C^2}{\rho}}{\frac{C}{\rho} + 2C + \rho}}. $ 
Then we {inductively prove} \eqref{Induction} for $k \geq K+1$: 
\begin{align}\label{Induction}
z_k \in B(z^*, r), \quad \| y_{k-1} - y_{k-2} \| > 0,  \quad \frac{ C_0 \| y_{k} - y_{k-1} \|^2}{ \| y_{k-1} - y_{k-2} \| } \leq   \psi(h_{k-1}) - \psi(h_k).
\end{align}
\vvs{We first prove} two useful inequalities. From \eqref{LagDesc_ineq0}, we have \eqref{Conv-Ineq-1} for $k \geq K-1$:
\begin{align}\label{Conv-Ineq-1}
\notag
& \|y_{k+1} - y_k \|^2 \leq \frac{h_k - h_{k+1} }{ {\rho / 2} - { C^2 / \rho} } \leq \frac{h_{K-1}}{ {\rho / 2} - {C^2 / \rho} }  \\
\implies & \| y_{k+1} - y_k  \| \leq \sqrt{ \frac{h_{K-1}}{ {\rho / 2} - {C^2 / \rho} }  } .
\end{align}
Furthermore, {$\|z_{k+1} - z_k\|$ may be bounded as follows for $\forall k \geq K$.}
\begin{align}
\notag
\| z_{k+1} - z_k \| & = \sqrt{\| w_{k+1} - w_k \|^2 + \| y_{k+1} - y_k \|^2 + \| \lambda_{k+1} - \lambda_k \|^2 } \\
\notag
& \leq \| w_{k+1} - w_k \| +  \| y_{k+1} - y_k \| + \| \lambda_{k+1} - \lambda_k \| \\
\notag
& \leq  \| w_{k+1} - y_{k+1} \| + \| y_k - w_k \| + 2 \| y_{k+1} - y_k \| + \| \lambda_{k+1} - \lambda_k \| \\
\notag
& \leq ( 1/\rho + 1 ) \| \lambda_{k+1} - \lambda_k \| + \| \lambda_k - \lambda_{k-1} \|/\rho + 2 \| y_{k+1} - y_k \| \\
\label{Conv-Ineq0}
& \overset{(i)}{\leq} C \| y_k - y_{k-1} \| / \rho + (C/\rho + C + 2) \| y_{k+1} - y_k \|.
\end{align}
We utilize these inequalities to show \eqref{Induction} by induction.

\noindent $\pmb{ K+1 }$:
Through \eqref{Initial Ineq}, \eqref{Conv-Ineq-1}, \eqref{Conv-Ineq0}, the following holds
\begin{align*}
& \| z_{K+1} - z^* \|  \leq \| z_{K+1} - z_K \| + \| z_K - z^* \| \\
 & \overset{\eqref{Conv-Ineq0}}{\leq}  C \| y_K - y_{K-1} \| / \rho + (C/\rho + C + 2) \| y_{K+1} - y_K \|  + \| z_K - z^* \| \\
& 
 \overset{\eqref{Conv-Ineq-1}}{\leq} \left( {2C \over \rho} + C + 2 \right) \sqrt{ \frac{h_{K-1}}{ {\rho / 2} - {C^2 / \rho} } }  + \| z_K - z^* \| 
\overset{\eqref{Initial Ineq}}{<} r,
\end{align*}
\vvs{implying $z_{K+1} \in B(z^*,r).$} From the optimality conditions of \eqref{update1} and \eqref{update2},
\begin{align}
\notag
 0 & \in \nabla_w h(w_{K}) + \lambda_{K-1} + \rho(w_{K} - y_{K-1}) + \partial \bk1_{Z_1}(w_{K}) \\ 
\label{formula1-Conv}
 \implies
\Delta \lambda_K - \rho \Delta y_K & \in 
 \nabla_w h(w_{K}) + \lambda_{K} + \rho(w_{K} - y_{K}) + \partial \bk1_{Z_1}(w_{K}) \\
 \notag
 0 & \in \nabla_y p(y_{K}) -  \lambda_{K-1} + \rho(y_{K} - w_{K}) + \partial \bk1_{Z_2}(y_{K}) \\
 \label{formula2-Conv}
 \implies  - \Delta \lambda_{K}
 & \in \nabla_y p(y_{K}) - \lambda_{K} + \rho(y_{K} - w_{K}) + \partial \bk1_{Z_2}(y_{K}),
\end{align}
where $\Delta \lambda_K \triangleq \lambda_K - \lambda_{K-1}, \Delta y_K \triangleq y_K - y_{K-1}$. So \eqref{formula1-Conv}, \eqref{formula2-Conv} and the fact that
\begin{align}
   \notag 
\partial \scrH_\rho(z_{K}) & = ( \nabla_w h(w_{K}) + \lambda_{K} + \rho(w_{K} - y_{K}) + \partial \bk1_{Z_1}(w_{K}) )\\
\notag
& \quad \times 
(\nabla_y p(y_{K}) - \lambda_{K} + \rho(y_{K} - w_{K}) + \partial \bk1_{Z_2}(y_{K}) ) \times ( w_{K} - y_{K} )
\end{align}
imply that $ \partial \scrH_\rho(z_{K}) \ni
\pmat{ \Delta \lambda_{K} - \rho \Delta  y_K; - \Delta  \lambda_{K}; w_{K} - y_{K} }$ and 
\begin{align}
\notag
& \mbox{dist}(0, \partial \scrH_\rho(z_{K}) )  
 \leq \sqrt{ \| \Delta \lambda_{K} - \rho \Delta y_{K} \|^2 + \|  \Delta \lambda_{K} \|^2 + \| w_{K} - y_{K} \|^2 } \\
\notag 
& \leq \| \Delta \lambda_{K} - \rho \Delta y_{K} \| + \| \Delta \lambda_{K} \| + \| w_{K} - y_{K} \| \\ 
\label{Conv_ineq1}
& \leq (1/\rho + 2) \| \Delta \lambda_{K} \|  +  \rho \| \Delta y_{K} \| 
 \overset{(i)}{\leq} (C/\rho + 2C + \rho)  \| \Delta y_{K} \| \\
\notag
\implies  & - \Delta \psi(h_{K+1})
\geq \psi'(h_K) ( h_K - h_{K+1} )
\overset{\eqref{LagDesc_ineq0}}{\geq} \psi'(h_{K}) \left({ \rho \over 2} - {C^2 \over \rho} \right) \| \Delta  y_{K+1}  \|^2 \\
\label{Conv_Ineq3}
& \geq \frac{( { \rho \over 2} - {C^2 \over \rho} ) \| \Delta y_{K+1} \|^2}{\mbox{dist}(0,\partial \scrH_\rho(z_{K}))} \overset{\eqref{Conv_ineq1}}{\geq} \frac{( {\rho \over 2} - {C^2 \over \rho} ) \| \Delta y_{K+1} \|^2}{({C \over \rho} + 2C + \rho)  \| \Delta y_{K} \|} 
 = \frac{ C_0 \| \Delta y_{K+1} \|^2}{ \|  \Delta y_{K} \|},
\end{align}
\yxIII{where $\Delta \psi(h_{K+1}) \triangleq \psi(h_{K+1}) - \psi(h_K)$, the first inequality of \eqref{Conv_Ineq3} follows from the concavity of $\psi$, and the third inequality is due to the K{\L} inequality $\psi'(h_K)$ ${\rm dist}(0,\partial \scrH_\rho(z_K) )$ $\geq 1$ because $\| z_K - z^* \| < r$. Moreover, the K{\L} inequality indicates that ${\rm dist}(0,\partial \scrH_\rho(z_K) ) > 0$, thus $\| \Delta y_K \| > 0 $ by \eqref{Conv_ineq1}.}
\yxIII{Therefore, the inductive hypothesis holds for $K+1$. Assume that it holds for $ K+2, \hdots, k$ and consider $k+1.$ }

\noindent $\pmb{k +1}$: We begin by showing that:
\begin{align}\label{Conv_Ineq4}
\sum_{i=K}^{k-1} \| y_{i+1} - y_i \| \leq \frac{1}{2}  \left[ \frac{\psi(h_K)}{C_0} + \left[ \frac{\psi(h_K)}{C_0} \left( \frac{\psi(h_K)}{C_0} + 4\sqrt{\frac{h_{K-1}}{\rho/2 - C^2/\rho }} \right) \right]^{1/2} \right].
\end{align}
Combining inductive hypothesis \eqref{Induction} for $K+1, \hdots, k$, we have that
\begin{align*}
& \psi( h_K ) - \psi(h_k ) \geq C_0 \sum_{i = K}^{k-1} \frac{\| y_{i+1} - y_i \|^2}{ \| y_i - y_{i-1} \|} \\
 \implies & \frac{\psi( h_K ) - \psi( h_k ) }{C_0}  \sum_{i = K}^{k-1} \| y_i - y_{i-1} \| \geq \left( \sum_{i = K}^{k-1} \frac{\| y_{i+1} - y_i \|^2 }{ \| y_i - y_{i-1} \|} \right) \left( \sum_{i = K}^{k-1} \| y_i - y_{i-1} \| \right) \\
& \overset{(*)}{\geq} \left( \sum_{i = K}^{k-1} \| y_{i+1} - y_i \| \right)^2 \\
\implies & \left( \sum_{i = K}^{k-1} \| y_{i+1} - y_i \| \right)^2 \leq \frac{\psi( h_K )}{C_0} \sum_{i=K-1}^{k-2} \| y_{i+1} - y_i \| \\
& \leq \frac{\psi( h_K )}{C_0} \sum_{i=K-1}^{k-1} \| y_{i+1} - y_i \| \\
& = \frac{\psi( h_K )}{C_0} \left( \sum_{i=K}^{k-1} \| \Delta y_{i+1} \| + \| \Delta y_K \| \right) \\
& \overset{\eqref{Conv-Ineq-1}}{\leq} \frac{\psi( h_K )}{C_0} \left( \sum_{i=K}^{k-1} \|\Delta y_{i+1} \| + \sqrt{\frac{h_{K-1}}{\rho/2 - C^2/\rho}} \right), 
\end{align*}
where $(*)$ holds because of \uvs{H\"{o}lder's} inequality. 
If $x \triangleq  \sum_{i=K}^{k-1} \| y_{i+1} - y_i \|$, $C_1 \triangleq {\psi(h_K) \over C_0}$, and $C_2 \triangleq \frac{\psi( h_K )}{C_0} \sqrt{\frac{h_{K-1}}{\rho/2 - C^2/\rho}}$, then the above inequality is equivalent to  $x^2 - C_1x - C_2 \leq 0 \implies x \leq \frac{1}{2}\left( C_1 + \sqrt{C_1^2 + 4C_2} \right)$. This is exactly \eqref{Conv_Ineq4}.  Therefore,
\begin{align*}
&  \| z_{k+1} - z_K \| 
\leq \sum_{i=K}^k \| z_{i+1} - z_i \| \\
& \overset{\eqref{Conv-Ineq0}}{\leq} \sum_{i = K}^k  \left(\frac{ \| y_i - y_{i-1} \|}{\rho / C}+ \left( {C \over \rho} + C + 2 \right) \| y_{i+1} - y_i \|\right) \\
& =\frac{ C\| y_K - y_{K-1} \|}{\rho} + \sum_{i=K}^{k-1}\left( {2C \over \rho} + C + 2 \right)\| y_{i+1} - y_i \| \\
& + \left({C \over \rho} + C + 2\right) \| y_{k+1} - y_k \| \\
& \overset{\eqref{Conv-Ineq-1}\eqref{Conv_Ineq4}}{\leq} \left( \frac{2C}{\rho} + C + 2 \right) \sqrt{\frac{h_{K-1}}{\rho/2 - C^2 / \rho}} + 
\left( {C \over \rho} + \frac{C}{2} + 1 \right) \left( C_1 + \sqrt{C_1^2 + 4C_2} \right) \\
& \implies \| z_{k+1} - z^* \| \leq \| z_{k+1} - z_K \| + \| z_K - z^* \| \overset{\eqref{Initial Ineq}}{<} r. 
\end{align*}
Thus, $z_{k+1} \in B(z^*,r)$.
Since $z_k \in B(z^*,r)$, the K{\L} inequality holds for $z_{k}$, and \uvs{in a fashion similar to} \eqref{Conv_Ineq3}, we obtain that 
$
\| y_{k} - y_{k-1} \| > 0 \mbox{ and } \frac{C_0 \| y_{k+1} - y_k \|^2 }{\| y_k - y_{k-1} \|} \leq \psi(h_k) - \psi(h_{k+1}),
$
completing the proof of the inductive hypothesis. By the hypothesis, \eqref{Conv_Ineq4} holds for $k \geq K+1$. This indicates that
$
\sum_{i = K}^{+\infty} \| y_{i+1} - y_i \| < +\infty$, implying that $\sum_{i = 0}^{+\infty} \| y_{i+1} - y_i \| < +\infty.$
 
\noindent{\bf (v).} \uvs{From (iv) and by recalling that $ \| \lambda_{k+1} - \lambda_k \| \leq C \| y_{k+1} - y_k \|$ for $k$ sufficiently large, we have that $\{y_k\}$ and $\{\lambda_k\}$ are Cauchy sequences,  convergent to $y^*$ and $\lambda^*$, respectively}. Since  $w_k - y_k \rightarrow 0$, $\{ w_k \}$ also converges to $w^* = y^*$. By Proposition~\ref{limit KKT point}, $(w^*,y^*,\lambda^*)$ is a KKT point. \qed
\end{proof}

\begin{remark}
(i). To derive convergence of the sequence, we leverage the K{\L} property of $\scrH_\rho$. When $p(y)$ is semialgebraic \cite[Sec.~4.3]{10.2307/40801236}, $\scrH_\rho$ is a sum of semialgebraic functions \uvs{and is therefore} semialgebraic. Then the result follows from 
\cite[Sec.~4.3]{10.2307/40801236} whereby a semialgebraic function $\Lscr$ satisfies the K{\L} property at
every point in  $\mbox{dom}(\partial \Lscr)$. \\
(ii). If we cannot \uvs{invoke} the K{\L} property to show convergence of $\{ (w_k,y_k,\lambda_k) \}$, we may \uvs{merely} conclude that any cluster point of $\{ (w_k,y_k,\lambda_k) \}$ satisfies first order KKT conditions of \eqref{mpcc}. The proof is similar to Proposition~\ref{limit KKT point} thus omitted. \\
(iii). Boundedness of $\{ y_k \}$ holds by assuming compactness of $Z_2$ \vvs{(obtainable by adding constraints $x^+ \leq ub^+, x^- \leq ub^-$)}.
\end{remark} 

\end{document}